\documentclass[11pt]{article}
\usepackage{etex}
\usepackage[margin={1.2in}]{geometry}
\usepackage[titletoc]{appendix}
\usepackage{tikz,tikz-cd}
\usetikzlibrary{positioning}
\usetikzlibrary{matrix}
\usetikzlibrary{shapes, arrows}
\usetikzlibrary{decorations.markings,patterns,bending}
\usepackage{eucal}
\usepackage{mathrsfs}
\usepackage{stmaryrd}
\usepackage{mathtools}
\usepackage{enumerate}
\usepackage{float}
\usepackage{subcaption}

 \usepackage{color}
 \usepackage{caption}
\DeclarePairedDelimiter\abs{\lvert}{\rvert}%

\usepackage{rotating}

\usepackage{longtable}

\usepackage[sc]{mathpazo}
\usepackage{courier}

\usepackage[utf8]{inputenc}
\usepackage[T1]{fontenc}

\usepackage[english]{babel}
\usepackage{blindtext}

\usepackage{amsmath}

\usepackage{microtype}

\usepackage{amsmath}
\usepackage{amssymb}
\usepackage{amsthm}
\usepackage{color}
\usepackage{latexsym,extarrows}

\usepackage[all]{xy}
\usepackage{tabularx}

\usepackage[stable,multiple]{footmisc}

\usepackage{hyperref}

\RequirePackage[font=small,format=plain,labelfont=bf,textfont=it]{caption}
\makeatletter
\newsavebox{\@brx}
\newcommand{\llangle}[1][]{\savebox{\@brx}{\(\m@th{#1\langle}\)}%
  \mathopen{\copy\@brx\kern-0.5\wd\@brx\usebox{\@brx}}}
\newcommand{\rrangle}[1][]{\savebox{\@brx}{\(\m@th{#1\rangle}\)}%
  \mathclose{\copy\@brx\kern-0.5\wd\@brx\usebox{\@brx}}}
\makeatother

\usepackage{authblk}

\date{}  

\begin{document}
\def\e#1\e{\begin{equation}#1\end{equation}}
\def\ea#1\ea{\begin{align}#1\end{align}}
\def\eq#1{{\rm(\ref{#1})}}
\theoremstyle{plain}
\newtheorem{thm}{Theorem}[section]
\newtheorem{lem}[thm]{Lemma}
\newtheorem{prop}[thm]{Proposition}
\newtheorem{cor}[thm]{Corollary}
\theoremstyle{definition}
\newtheorem{dfn}[thm]{Definition}
\newtheorem{ex}[thm]{Example}
\newtheorem{rem}[thm]{Remark}
\newtheorem{conjecture}[thm]{Conjecture}
\newtheorem{convention}[thm]{Convention}
\newtheorem{assumption}[thm]{Assumption}
\newtheorem{notation}[thm]{Notation}

\newcommand{\D}{\mathrm{d}}
\newcommand{\A}{\mathcal{A}}
\newcommand{\LL}{\llangle[\Big]}
\newcommand{\RR}{\rrangle[\Big]}
\newcommand{\LD}{\Big\langle}
\newcommand{\RD}{\Big\rangle}
\newcommand{\F}{\mathcal{F}}
\newcommand{\HH}{\mathcal{H}}
\newcommand{\X}{\mathcal{X}}

\newcommand{\K}{\mathscr{K}}
\newcommand{\q}{\mathbf{q}}

\newcommand{\op}{\operatorname}
\newcommand{\C}{\mathbb{C}}
\newcommand{\N}{\mathbb{N}}
\newcommand{\R}{\mathbb{R}}
\newcommand{\Q}{\mathbb{Q}}
\newcommand{\Z}{\mathbb{Z}}
\renewcommand{\H}{\mathbf{H}}
\newcommand{\PP}{\mathbb{P}}

\newcommand{\Etau}{\text{E}_\tau}
\newcommand{\E}{{\mathcal E}}
\newcommand{\G}{\mathbf{G}}
\newcommand{\eps}{\epsilon}

\newcommand{\im}{\op{im}}

\newcommand{\h}{\mathbf{h}}

\newcommand{\Gmax}[1]{G_{#1}}
\newcommand{\AW}{E}
\newcommand{\abracket}[1]{\left\langle#1\right\rangle}
\newcommand{\bbracket}[1]{\left[#1\right]}
\newcommand{\fbracket}[1]{\left\{#1\right\}}
\newcommand{\bracket}[1]{\left(#1\right)}
\newcommand{\ket}[1]{|#1\rangle}
\newcommand{\bra}[1]{\langle#1|}

\newcommand{\ora}[1]{\overrightarrow#1}

\providecommand{\from}{\leftarrow}
\providecommand{\to}{\rightarrow}
\newcommand{\bl}{\textbf}
\newcommand{\mbf}{\mathbf}
\newcommand{\mbb}{\mathbb}
\newcommand{\mf}{\mathfrak}
\newcommand{\mc}{\mathcal}
\newcommand{\cinfty}{C^{\infty}}
\newcommand{\pa}{\partial}
\newcommand{\prm}{\prime}
\newcommand{\dbar}{\bar\pa}
\newcommand{\OO}{{\mathcal O}}
\newcommand{\hotimes}{\hat\otimes}
\newcommand{\BV}{Batalin-Vilkovisky }
\newcommand{\CE}{Chevalley-Eilenberg }
\newcommand{\suml}{\sum\limits}
\newcommand{\prodl}{\prod\limits}
\newcommand{\into}{\hookrightarrow}
\newcommand{\Ol}{\mathcal O_{loc}}
\newcommand{\mD}{{\mathcal D}}
\newcommand{\iso}{\cong}
\newcommand{\dpa}[1]{{\pa\over \pa #1}}
\newcommand{\Kahler}{K\"{a}hler }
\newcommand{\0}{\mathbf{0}}

\newcommand{\B}{\mathcal{B}}
\newcommand{\V}{\mathcal{V}}

\newcommand{\M}{\mathfrak{M}}

\newcommand{\DD}{\Omega^{\text{\Romannum{2}}}}


\numberwithin{equation}{section}

\newcommand{\comment}[1]{\textcolor{red}{[#1]}} 

\makeatletter
\newcommand{\subjclass}[2][2010]{%
  \let\@oldtitle\@title%
  \gdef\@title{\@oldtitle\footnotetext{#1 \emph{Mathematics Subject Classification.} #2}}%
}
\newcommand{\keywords}[1]{%
\textbf{\textit{Keywords---}} 
 {\emph{Gromov-Witten theory; Seiberg Duality conjecture; Quivers varieties.} #1}}%

\makeatother

\makeatletter
\let\orig@afterheading\@afterheading
\def\@afterheading{%
   \@afterindenttrue
  \orig@afterheading}
\makeatother

\makeatletter
\newcommand*\bigcdot{\mathpalette\bigcdot@{.5}}
\newcommand*\bigcdot@[2]{\mathbin{\vcenter{\hbox{\scalebox{#2}{$\m@th#1\bullet$}}}}}
\makeatother

\title{\bf Seiberg Duality Conjecture for Star-Shaped Quivers and Finiteness of Gromov-Witten Theory for D-Type Quivers}
\author{Weiqiang He\footnote{Department of Mathematics, Sun Yat-sen University, No.135 Xingang West Road, Haizhu District, Guangzhou 510275, P.R. China; email:  hewq25@mail.sysu.edu.cn.}\,\, and Yingchun Zhang \footnote{School of Mathematical Sciences, Shanghai Jiao Tong University, 800 Dongchuan Road, Minhang District,
Shanghai 200240, China; 
email: yczhang1992@gmail.com. Correspondence to be sent to yczhang1992@gmail.com}}
\date{\today}

\maketitle
\begin{abstract}
This work proves that the Seiberg Duality Conjecture holds for star-shaped quivers: 
the Gromov-Witten theories of mutation-related varieties are equivalent.

In particular, it is known that there are only finitely many quivers that are mutation equivalent to a $D$-type quiver. 
We prove that the Seiberg Duality Conjecture holds for all quivers that are mutation equivalent to a $D_3$-type quiver, and find the change of K\"ahler variables. 
 \end{abstract}
\setcounter{tocdepth}{2} \tableofcontents 
\section{Introduction}
The equivalence of two-dimensional (2D) quantum field theories has long been an intriguing topic in mathematical physics. There are notable examples of such equivalences, including the following two: (1)the equivalence between Gromov-Witten theories under torus-equivariant birational transformation \cite{LR, LLW, R, BG, CIJ, GW};
(2) the Landau-Ginzburg/Calabi-Yau correspondence (equivalence between Gromov-Witten theory of CY hypersurface and Fan-Jarvis-Ruan-Witten theory of a Landau-Ginzburg model) \cite{CR, CIR, CIRS}. 
Both these examples focus on a GIT quotient $[V\sslash G]$ (and its complete intersection) where the gauge group  $G$ is abelian. However, there are many essential and exciting GIT quotients where $G$ is nonabelian, such as quiver varieties and their complete intersections. This raises a natural question: is there any equivalence of gauge theories for nonabelian GIT quotients? Seiberg Duality gives a positive answer to the question.

The fascinating Seiberg Duality Conjecture asserts the equivalence of gauge theories of two quivers related by a quiver mutation.  Typically, such quivers are not simply related by a phase transition. 
While this topic is extensively investigated in the physics literature, it remains less explored in math. 
See \cite{Hori,HoriTong,benini2015cluster,gomis2016m2} for physics achievements. 

In this context, a mathematical version for Seiberg Duality Conjecture was proposed by Yongbin Ruan for the 2D Gauged Linear Sigma Model \cite{nonabelianGLSM:YR}. 
We have proved the Seiberg duality conjecture for $A$-type quivers in previous works, see \cite{donghai,zhang2021gromov}. 
In this work, we focus on star-shaped quivers and quivers that are mutation equivalent to a $D_3$-type quiver. 
This paper is self-contained and provides new insights into these classes of quivers.

\subsection{Introduction to Seiberg Duality Conjecture}
We begin by considering a quiver with a potential function $\mathbf Q=(Q_f\subset Q_0, Q_1, W)$, where $Q_0$ is the set of nodes, 
among which $Q_f$ is the set of framed nodes and $Q_0\backslash Q_f$ is the set of gauged nodes, $Q_1$ is the set of arrows, and $W$ is a potential function. 
We usually denote arrows from a node $i$ to another node $j$ by $i\rightarrow j$ where  $b_{ij}$ indicating the number of such arrows. See \cite{quiver}. 
We assume that there are no 1-cycles and no 2-cycles, and call such quivers cluster quivers following the terminology in \cite[Sec. 3.1]{benini2015cluster}. 

Decorate a quiver with potential by an integer vector $\vec v=(N_i)_{i\in Q_0}$,  one integer to a node. 
For each node $k$, define the outgoing and incoming to be
$N_f(k):=\sum_{k\rightarrow j}b_{kj}N_j$ and $N_{a}(k):=\sum_{i\rightarrow k}b_{ik}N_i$.
Then we have the input data  $(V, G, \theta)$ for a GIT quotient where $V=\oplus_{i\rightarrow j}\C^{N_i\times N_j}$, $G=\prod_{i\in Q_0\backslash Q_f}GL(N_i)$, and $\theta\in \chi(G)$. 
The quiver variety is defined to be the GIT quotient $V\sslash_\theta G$, see Definition \ref{dfn:quiver}. 
When the potential function $W$ is nontrivial, we need to consider the critical locus of the potential function denoted by $\mc Z:=\{dW=0\}\sslash_\theta G$. We call it the variety of a quiver with a potential.

Performing a quiver mutation as Definition \ref{dfn:mutation} at a gauge node $k$ (we will reserve this small letter $k$ for the node we perform a quiver mutation at), we obtain a new quiver with potential $\tilde{\mathbf Q}=(\tilde Q_f\subset \tilde Q_0, \tilde Q_1, \tilde W)$ together with the assigned integer vector $\vec v'$. 
Whenever there is a pair of opposite arrows between two nodes arising from a quiver mutation, we have to annihilate them, so $\tilde{\mathbf Q}$ is still a cluster quiver. Denote the input data of the GIT quotient by $(\tilde V, \tilde G, \tilde{\theta})$, and we can construct the critical locus $\tilde{\mc Z}=\{d\tilde W=0\}\sslash_{\tilde \theta}\tilde G$. 

We will consider Gromov-Witten theory of $\mc Z$ and $\tilde{\mc Z}$, which roughly speaking count genus-$g$ curves of some degree in the target varieties $\mc Z$ and $\tilde{\mc Z}$. 
Let $\mc F_g^{\mc Z}(\vec q)$ and $\mc F_g^{\tilde{\mc Z}}(\vec q')$ denote the generating functions of genus-$g$ Gromov-Witten invariants of $\mc Z$ and $\tilde {\mc{Z}}$ with $\vec q$ and $\vec q'$ their K\"ahler variables. 
The 2D Seiberg Duality Conjecture is stated as follows.
\begin{conjecture}[\cite{nonabelianGLSM:YR,benini2015cluster}]\label{conjecture}
The generating functions of two mutation-related varieties satisfy
\begin{equation}
    \mc F_g^{\mc Z}(\vec q)=\mc F_g^{\tilde{\mc Z}}(\vec q')\,,
\end{equation}
under the change of K\"ahler variables: 
$q_{k}'=q_k^{-1}$, and for $i\neq k$, 
\begin{itemize}
        \item  if $N_f(k)>N_a(k)$, $\frac{{e}^{\pi i(N_f(j)'-1)}q_j'}{ {e}^{\pi i(N_f(j)-1)}q_j}=({e}^{\pi i N_k'}q_k)^{[b_{kj}]_+}
        ({e}^{\pi i N_k'})^{[-b_{kj}]_+} \prod_{i\neq k}e^{\pi iN_i a_{ij}} $;
        \item if $N_f(k)=N_a(k)$, $\frac{{e}^{\pi i(N_f(j)'-1)}q_j'}{ {e}^{\pi i(N_f(j)-1)}q_j}=\left(\frac{{e}^{\pi i N_k'}q_k}{1+(-1)^{N_k'}q_k}\right)^{[b_{kj}]_+}
        \left({e}^{\pi i N_k'}(1+(-1)^{N_k'}q_k)\right)^{[-b_{kj}]_+} \prod_{i\neq k}e^{\pi iN_i a_{ij}} $.
        \item If $N_f(k)<N_a(k)$, $\frac{{e}^{(N_f(j)'-1)}q_j'}{ {e}^{(N_f(j)-1)}q_j}=({e}^{\pi i N_k'})^{[b_{kj}]_+}
        \left({e}^{\pi i(N_f(k)-N_k)}q_k\right)^{-[-b_{kj}]_+} \prod_{i\neq k}e^{\pi iN_i a_{ij}} $
    \end{itemize}
    where $a_{ij}$ denotes the number of “annihilated” 2-cycles between the nodes i and j in the quiver mutation, $[b]_+=\max\{b, 0\}$.
\end{conjecture}

The two characters $\theta$ and $\tilde \theta$ are not arbitrary and 
we propose that they are related in the following way. 
\begin{conjecture}\label{conj:phase}
    Write $\theta(g)=\prod_{i\in Q_0\backslash Q_f}\det(g_i)^{\sigma_i}$ for $g\in G$ and 
    $\tilde \theta(\tilde g)=\prod_{i\in Q_0\backslash Q_f} \det(\tilde g_i)^{\tilde{\sigma}_i}$ for $\tilde g\in \tilde G$. 
    When the two phases $\sigma_i$ and $\tilde \sigma_i$ are related in the following way, 
    \begin{itemize}
        \item when $\sigma_k>0$, $\tilde \sigma_k=-\sigma_k$, $\tilde \sigma_i=\sigma_i+[b_{ki}]_+\sigma_k$ for $i\neq k$,
        \item when $\sigma_k<0$, $\tilde \sigma_k=-\sigma_k$, $\tilde \sigma_i=\sigma_i+[-b_{ki}]_+\sigma_k$ for $i\neq k$,
    \end{itemize}
    the varieties $\mc Z$ and $\tilde{\mc Z}$ satisfy the relations in Seiberg duality Conjecture \ref{conjecture}.
 \end{conjecture}
The reason why we propose such a relation between the two characters is that the Seiberg duality is a local behavior, which means when we perform a quiver mutation at a node $k$, 
only adjacent nodes and arrows are affected and the semi-stability of remaining matrices of arrows that are not adjacent to the node $k$ are not changed.  
The stability conditions $\theta$ and $\tilde\theta$ are chosen to make this work. 

We will focus on the genus-zero version of the Conjecture. 
The genus zero wall-crossing Theorem states that the (equivariant) $\mc J$-function is equal to the (equivariant) quasimap small $I$-function under mirror map, see \cite{MR3126932}\cite{MR3412343} \cite{MR3272909}. 
Hence, we will instead investigate the transformations of the equivariant quasimap small $I$-functions (equivariant small $\mc J$-function)  under quiver mutations. 

We assume that $\mc Z$ and $\tilde{\mc Z}$ admit a common good torus action S, and denote their equivariant quasimap small $I$-functions by $I^{{\mc Z,S}}(q)$ and $I^{\tilde{\mc Z}, S}(\vec q')$. 

The first goal of this work is to prove the Seiberg duality conjecture for a star-shaped quiver described in Definition \ref{defn:starshapedquiver}, for example, a quiver in Figure \ref{fig:intr_star}.
\noindent
\begin{figure}[H]
    \centering
    \begin{tikzpicture}   
        \node[draw, circle, minimum size=0.8cm] (node-5) at (0,0){$N_5$};
        \node[draw, circle, minimum size=0.8cm] (node-3) at (-1.5,1){$N_3$};
        \node[draw, circle, minimum size=0.8cm] (node-4) at (-1.5,-1){$N_4$};
        \node[draw, circle, minimum size=0.8cm] (node-1) at (-3.5,1){$N_1$};
        \node[draw, circle, minimum size=0.8cm] (node-2) at (-3.5,-1){$N_2$};        
       \node[draw, circle, minimum size=0.8cm] (node-6) at (1.5,1){$N_6$};
        \node[draw, circle, minimum size=0.8cm] (node-7) at (1.5,-1){$N_7$};
        \node[draw, minimum width=0.8cm, minimum height=0.8cm] (node-8) at (3.5,1){$N_8$};
        \node[draw, minimum width=0.8cm, minimum height=0.8cm] (node-9) at (3.5,-1){$N_9$};
       \draw[-stealth] (node-1) -- (node-3); 
        \draw[-stealth] (node-2) -- (node-4);
        \draw[-stealth] (node-4) -- (node-5);
        \draw[-stealth] (node-3) -- (node-5);
        \draw[-stealth] (node-5) -- (node-6);
        \draw[-stealth] (node-5) -- (node-7);
        \draw[-stealth] (node-6) -- (node-8);
        \draw[-stealth] (node-7) -- (node-9);
        \node at (-3.5,0.4){$1$};
        \node at (-1.5,0.4){$3$};
        \node at (-1.5,-1.6){$4$};
        \node at (0,-0.7){$5$};
        \node at (-3.5,-1.6){$2$};
         \node at (3.5,0.4){$8$};
        \node at (1.5,0.4){$6$};
        \node at (1.5,-1.6){$7$};
        \node at (3.5,-1.6){$9$};
    \end{tikzpicture}
    \caption{A star-shaped quiver with additional conditions on $N_i$ as Example \ref{ex:generalstar}.}
    \label{fig:intr_star}
\end{figure}

In particular, the Seiberg Duality Conjecture holds for $ D,\, E$-type quivers if we view them as special star-shaped quivers.
By \cite{clusteralg:2}, $A,\,D,\,E$-type quivers  only have finitely many mutation equivalent quivers.  Let $\Omega$ denote the finite set of  quivers that are mutation equivalent to $D_3$-quiver in Figure \ref{fig:intrD4}, which are given in Section \ref{sec:D3mutequivquivers} explicitly.
\begin{figure}[ht]
    \centering
    \begin{tikzpicture}   
        \node[draw, circle, minimum size=0.8cm] (node-3) at (0,0){$N_3$};
        \node[draw, circle, minimum size=0.8cm] (node-1) at (-2,1){$N_1$};
        \node[draw, circle, minimum size=0.8cm] (node-2) at (-2,-1){$N_2$};
        \node[draw, minimum width=1cm, minimum height=1cm] (node-4) at (2,0){$N_4$};
       \draw[-stealth] (node-1) -- (node-3); 
        \draw[-stealth] (node-2) -- (node-3);
        \draw[-stealth] (node-3) -- (node-4);
        \node at (-2,0.4){$1$};
        \node at (-2,-1.6){$2$};
        \node at (0,-0.6){$3$};
        \node at (2,-0.6){$4$};
        \node at (-1,0.8){$A_1$};
        \node at (-1,-0.3){$A_2$};
    \end{tikzpicture}
    \caption{Assume $N_4=N_1+N_2,\,N_4>N_3,\,N_3>N_1,\,N_3>N_2$.}
    \label{fig:intrD4}
\end{figure}
\noindent
Our second goal is to investigate the relation of Gromov-Witten theories of quivers in $\Omega$. More explicitly, 
(1) we will find the transformation of the quasimap small $I$-functions including the change of K\"ahler variables of quivers in $\Omega$ under quiver mutations, 
(2) let $\mc C$ be the set of K\"ahler variables of all quivers in $\Omega$, the set $\mc C$ is finite. This is what we mean the finiteness of Gromov-Witten theory of $D_3$-quiver.
 
\subsection{Main Theorems}
\subsubsection{Seiberg Duality Conjecture for a star-shaped quiver}
Consider a star-shaped quiver first as Figure \ref{fig:intr_star}. For a chosen phase as Equation \eqref{eqn:phaseofstar}, 
we can define a quiver variety and denote it by $\mc X_s$, where the subscript $s$ is the initial of the word star. 
Performing a quiver mutation $\mu_5$ at the center node $5$, we obtain the quiver diagram in Figure \ref{fig:generalstarmc} with $N_{5}'=N_{6}+N_{7}-N_{5}$. The quiver has a potential function $\tilde W_{s}=tr(A_3B_1A_5)+tr(A_3B_3A_6)+tr(A_4B_2A_5)+tr(A_4B_4A_6)$.
For the phase \eqref{eqn:phaseofgeneralstarmc} chosen according to the Conjecture \ref{conj:phase}, we get the critical locus of the potential function, $\tilde{\mc Z}_s:=(d\tilde W_{s}=0)\sslash_{\tilde{\theta}_{s}}\tilde G_s$. 
Both $\mc X_s$ and $\tilde{\mc Z_s}$ admit a good torus action $S=(\C^*)^{N_8+N_9}$ which acts on matrices $A_7, A_8$ naturally.
\begin{thm}\label{thm:1st}
The equivariant quasimap small $I$-function of $\mc X_s$ and that of $\tilde{\mc Z}_s$ are related as follows. 
\begin{enumerate}[(a)]
    \item When $N_6+N_7\geq N_3+N_4+2$, 
    \begin{equation}
       I^{\mc X_s,S}(\vec q)=I^{\tilde{\mc Z_s},S}(\vec q')\,,
    \end{equation}
    under the change of K\"ahler variables \begin{align}\label{eqn:intrvariablechange}
        q_5'=q_5^{-1}, \, q_6'=q_6q_5,\,q_7'=q_7q_5,\,q_i'=q_i, \text{ for } i\neq 5,6,7\,.
    \end{align}
     \item When $N_6+N_7= N_3+N_4+1$, 
    \begin{equation}
    I^{\mc X_s,S}(\vec q)=e^{(-1)^{N_5'-1}q_5}I^{{\tilde{\mc Z}_s},S}(\vec q')\,,
    \end{equation}
    under the change of K\"ahler variables in \eqref{eqn:intrvariablechange}.
    \item When $N_6+N_7= N_3+N_4$,
    \begin{equation}
    I^{\mc X_s,S}(\vec q)=(1+(-1)^{N_5'}q_5)^{\sum_{i=3,4}\sum_{A=1}^{N_i}x^{i}_A-\sum_{j=6,7}\sum_{B=1}^{N_i}x^j_B+N_5'}
    I^{{\tilde{\mc Z}_s},S}(\vec q')\,,
    \end{equation}
    under change of K\"ahler variables,
    \begin{align}
      &q_3'=q_3(1+(-1)^{N_5'}q_5)\,,
       q_4'=q_4(1+(-1)^{N_5'}q_5)\,,
       q_5'=q_5^{-1}\,,\nonumber\\
       &q_6'=\frac{q_6q_5}{(1+(-1)^{N_5'}q_5)}\,,
       q_7'=\frac{q_7q_5}{(1+(-1)^{N_5'}q_5)}\,, q_i'=q_i,\text{ for } i=1,2\,.
    \end{align}
\end{enumerate}
\end{thm}
The above theorem can be generalized to any star-shaped quiver defined in \ref{defn:starshapedquiver} as discussed in Corollary \ref{cor:generalstar}.

\subsubsection{Seiberg Duality Conjecture for $D_{3}$ mutation equivalent quivers}
The $D_{3}$-type quiver in Figure \ref{fig:intrD4} is a special case of a star-shaped quiver with only one outgoing arrow and two incoming arrows. 
Denote the quiver variety of the $D_{3}$-type quiver by $\mc X_{0}$.  
Performing all possible quiver mutations, we get the finite set $\Omega$ of all quivers that are mutation equivalent to $D_3$, 
see Section \ref{sec:D3mutequivquivers}.
Performing quiver mutations $\mu_3\rightarrow \mu_1\rightarrow \mu_2$ repeatedly, we get almost all but five quivers in $\Omega$ displayed in Figure \ref{fig:extrafig} and Figure \ref{fig:extrafig2}. 
Notice that relations of the three quiver $(1)(2)(3)$ in Figure \ref{fig:extrafig} are similar with those of quivers in Figure \ref{fig:m1m3} and Figure\ref{fig:extrafig2}$(4)(5)$, so we will only discuss the quivers in Figure \ref{fig:m1m3} and Figure\ref{fig:extrafig2}$(4)(5)$. 

We label the nine quivers obtained by performing quiver mutations $\mu_3\rightarrow \mu_1\rightarrow \mu_2$ by $\{\mathbf Q_i\}_{i=1}^9$ and label the quivers in Figure\ref{fig:extrafig2} by $\mathbf Q_{10},\mathbf Q_{11}$. 
Note that the quiver $\mathbf Q_9$ is the same with the quiver $\mathbf Q_0$ by exchanging nodes 1 and 2. 
Denote the corresponding varieties by $\mc X_i$ $(\mc Z_i \text{ if the potential function is nontrivial})$ in the phases proposed in Conjecture \ref{conj:phase}, which are discussed in Section \ref{quivermutation}.
All these varieties admit a common torus action $R:=(\mathbb C^*)^{N_4}$. 
\begin{thm}\label{thm:2nd}
The equivariant quasimap small $I$-functions of quivers $\{\mathbf Q_i\}_{i=1}^{11}$ in $\Omega$ satisfy the following relations:
\begin{enumerate}[(1)]
    \item  
        $$I^{\mc X_0,R}(\vec q)=(1+(-1)^{N_3'}q_3)^{\sum_{I=1}^{N_1}x^1_I+\sum_{I=1}^{N_2}x^2_I-\sum_{F}^{N_3}\lambda_F+N_3'}I^{\mc Z_1,R}(\vec q'),$$
under change of K\"ahler variables
 $$ q_1'=(1+(-1)^{N_3'}q_3)q_1,\,q_2'=(1+(-1)^{N_3'}q_3)q_2,\,q_3'=q_3^{-1};$$
   \item 
    $$I^{\mc Z_1,R}(q_1,q_2,q_3)=I^{\mc Z_2,R}(q_1^{-1},q_2,q_3)\,;\,\,I^{\mc Z_2,R}(q_1,q_2,q_3)=I^{\mc Z_3,R}(q_1,q_2^{-1},q_3)\,;$$
\item $$I^{\mc X_4,R}(q_1,q_2,q_3)=
      (1+(-1)^{N_3^{\prime}}q_3)^{ \sum_{F=1}^{N_4}\lambda_F-\sum_{I=1}^{N_2}x^1_I-\sum_{I=1}^{N_1}x^2_I+N_3^{\prime}}
      I^{\mc Z_3,R}(q_1',q_2',q_3')$$
    under change of K\"ahler variables
    $$q_1'=\frac{q_3 q_1}{1+(-1)^{N_3^{\prime}}q_3}, \,
     q_2'=\frac{q_3 q_2}
    {1+(-1)^{N_3^{\prime}}q_3},\,
    q_3'=q_3^{-1}\,;$$
\item 
    $$I^{\mc X_4,R}(q_1,q_2,q_3)=I^{\mc X_5,R}(q_1^{-1},q_2,q_3q_1)\,; \,\,I^{\mc X_6,R}(q_1,q_2,q_3)=I^{\mc X_5,R}(q_1,q_2^{-1},q_3q_2)\,;$$
\item $$I^{\mathcal{X}_7, R}(q_1q_3,q_2q_3,q_3^{-1})=I^{\mathcal{X}_6, R}(q_1,q_2,q_3);$$
\item 
    $$I^{\mc X_7,R} (q_1,q_2,q_3)=I^{\mc X_8,R}(q_1^{-1},q_2,q_3q_1)\,;\,\,
    I^{\mc X_8,R}(q_1,q_2,q_3)=I^{\mc X_9,R}(q_1,q_2^{-1},q_3q_2)\,;$$
\item 
      $$I^{\mc X_{10},R} (q_1,q_2,q_3)=
       (1+(-1)^{N_3-N_1}q_3)^{\sum_{B=1}^{N_2}x^2_B-\sum_{A=1}^{N_2}x^1_A+N_3-N_1}
       I^{\mc Z_2,R}(q_1',q_2',q_3')$$ 
    under change of K\"haler variable 
    $$q_1'=\frac{q_1q_3}{1+(-1)^{N_3-N_1}q_3},q_2'=q_2(1+(-1)^{N_3-N_1}q_3),q_3'=q_3^{-1};$$
\item  $$I^{\mc X_8,R}(q_1,q_2,q_3)=I^{\mc Z_{10},R}(q_1^{-1},q_2,q_1q_3)\,;\,\,\,I^{\mc X_{11},R}(q_1,q_2,q_3)=I^{\mc Z_{10},R}(q_1,q_2^{-1},q_2q_3).$$
\end{enumerate}
\end{thm} 

\subsection{Ideas for proofs}
\begin{figure}[H]
    \centering
    \begin{tikzpicture}   
        \node[draw, circle, minimum size=0.6cm] (node-5) at (0,0){$N_5$};
        \node[draw, circle, minimum size=0.6cm] (node-3) at (-1.5,1){$N_3$};
        \node[draw, circle, minimum size=0.6cm] (node-4) at (-1.5,-1){$N_4$};
        \node[] (node-1) at (-2.7,1){};
        \node[] (node-2) at (-2.7,-1){};    
       \node[draw, circle, minimum size=0.6cm] (node-6) at (1.5,1){$N_6$};
        \node[draw, circle, minimum size=0.6cm] (node-7) at (1.5,-1){$N_7$};
        \node[] (node-8) at (2.7,1){};
        \node[] (node-9) at (2.7,-1){};
       \draw[-stealth] (node-1) -- (node-3); 
        \draw[-stealth] (node-2) -- (node-4);
        \draw[-stealth] (node-4) -- (node-5);
        \draw[-stealth] (node-3) -- (node-5);
        \draw[-stealth] (node-5) -- (node-6);
        \draw[-stealth] (node-5) -- (node-7);
        \draw[-stealth] (node-6) -- (node-8);
        \draw[-stealth] (node-7) -- (node-9);
        \node at (-1.5,0.3){$3$};
        \node at (-1.5,-1.7){$4$};
        \node at (0,-0.7){$5$};
        \node at (1.5,0.3){$6$};
        \node at (1.5,-1.7){$7$};
        \node at (3.5,0) {$\xrightarrow[]{\mu_5}$};

        \node[draw, circle, minimum size=0.6cm] (node-m5) at (7,0){$N_5'$};
        \node[draw, circle, minimum size=0.6cm] (node-m3) at (5.5,1){$N_3$};
        \node[draw, circle, minimum size=0.6cm] (node-m4) at (5.5,-1){$N_4$};
        \node[] (node-m1) at (4.3,1){};
        \node[] (node-m2) at (4.3,-1){};    
       \node[draw, circle, minimum size=0.6cm] (node-m6) at (8.5,1){$N_6$};
        \node[draw, circle, minimum size=0.6cm] (node-m7) at (8.5,-1){$N_7$};
        \node[] (node-m8) at (9.7,1){};
        \node[] (node-m9) at (9.7,-1){};
       \draw[-stealth] (node-m1) -- (node-m3); 
        \draw[-stealth] (node-m2) -- (node-m4);
        \draw[-stealth] (node-m5) -- (node-m3);
        \draw[-stealth] (node-m5) -- (node-m4);
        \draw[-stealth] (node-m6) -- (node-m5);
        \draw[-stealth] (node-m7) -- (node-m5);
        \draw[-stealth] (node-m6) -- (node-m8);
        \draw[-stealth] (node-m7) -- (node-m9);
        \draw[-stealth] (node-m3) to [bend left] (node-m6);
        \draw[-stealth] (node-m3) to [bend left] (node-m7);
        \draw[-stealth] (node-m4) to [bend right] (node-m6);
        \draw[-stealth] (node-m4) to [bend right] (node-m7);
        \node[] at (0,-2){$\downarrow $};
         \node[draw, circle, minimum size=0.6cm] (node-f5) at (0,-4){$N_5$};
        \node[draw, minimum width=0.6cm, minimum height=0.6cm] (node-f3) at (-1.5,-3){$N_3$};
        \node[draw, minimum width=0.6cm, minimum height=0.6cm] (node-f4) at (-1.5,-5){$N_4$};
        \node[] (node-f1) at (-2.7,-3){};
        \node[] (node-f2) at (-2.7,-5){};    
       \node[draw, minimum width=0.6cm, minimum height=0.6cm] (node-f6) at (1.5,-3){$N_6$};
        \node[draw, minimum width=0.6cm, minimum height=0.6cm] (node-f7) at (1.5,-5){$N_7$};
        \node[] (node-f8) at (2.7,-3){};
        \node[] (node-f9) at (2.7,-5){};
       \draw[-stealth] (node-f1) -- (node-f3); 
        \draw[-stealth] (node-f2) -- (node-f4);
        \draw[-stealth] (node-f4) -- (node-f5);
        \draw[-stealth] (node-f3) -- (node-f5);
        \draw[-stealth] (node-f5) -- (node-f6);
        \draw[-stealth] (node-f5) -- (node-f7);
        \draw[-stealth] (node-f6) -- (node-f8);
        \draw[-stealth] (node-f7) -- (node-f9);
        \node at (3.5,-4) {$\xrightarrow[]{\mu_5}$};
        \node[] at (7,-2){$\downarrow $};
         \node[draw, circle, minimum size=0.6cm] (node-fm5) at (7,-4){$N_5'$};
       \node[draw, minimum width=0.6cm, minimum height=0.6cm] (node-fm3) at (5.5,-3){$N_3$};
        \node[draw, minimum width=0.6cm, minimum height=0.6cm] (node-fm4) at (5.5,-5){$N_4$};
        \node[] (node-fm1) at (4.3,-3){};
        \node[] (node-fm2) at (4.3,-5){};    
       \node[draw, minimum width=0.6cm, minimum height=0.6cm] (node-fm6) at (8.5,-3){$N_6$};
        \node[draw, minimum width=0.6cm, minimum height=0.6cm] (node-fm7) at (8.5,-5){$N_7$};
        \node[] (node-fm8) at (9.7,-3){};
        \node[] (node-fm9) at (9.7,-5){};
       \draw[-stealth] (node-fm1) -- (node-fm3); 
        \draw[-stealth] (node-fm2) -- (node-fm4);
        \draw[-stealth] (node-fm5) -- (node-fm3);
        \draw[-stealth] (node-fm5) -- (node-fm4);
        \draw[-stealth] (node-fm6) -- (node-fm5);
        \draw[-stealth] (node-fm7) -- (node-fm5);
        \draw[-stealth] (node-fm6) -- (node-fm8);
        \draw[-stealth] (node-fm7) -- (node-fm9);
        \draw[-stealth] (node-fm3) to [bend left] (node-fm6);
        \draw[-stealth] (node-fm3) to [bend left] (node-fm7);
        \draw[-stealth] (node-fm4) to [bend right] (node-fm6);
        \draw[-stealth] (node-fm4) to [bend right] (node-fm7);
        \end{tikzpicture}
    \caption{From the left to the right, we perform the quiver mutation at the center node, and from top to bottom, we freeze the adjacent nodes of node $5$.}
    \label{fig:idea}
\end{figure}
The key is that quiver mutation is a local behavior: it only affects the behavior of nodes and arrows around the node $k$.
The idea for proving the equivalence between $I^{\mc X_s}$ and $I^{\tilde{\mc Z_s}}$ is to freeze the nodes that are related to the center node as shown in the Figure \ref{fig:idea}.

The two quiver diagrams on the bottom of the Figure \ref{fig:idea} are the two quivers in fundamental building block in Figure \ref{fig:fundamental} with outgoing $n=N_7+N_8$ and incoming $m=N_3+N_4$.
By some nontrivial combinatorics, the equivalence of $I$-functions of $\mc X_s$ and $\tilde{\mathcal{Z}}_s$ is reduced to the fundamental building block. This is known by \cite{benini2015cluster,donghai}, so we are done.

The key idea to prove the Theorem \ref{thm:2nd} is to find correct phase for each quiver that is mutation equivalent to $D_3$ so that we can identify their $I$-functions. 
We let the phases change as proposed in Conjecture \ref{conj:phase}, and construct the corresponding varieties $\mc X_i(\mc Z_i)$ of quivers $\mathbf Q_i$. Their $I$-functions turn out to satisfy the transformation rule in Seiberg duality Conjecture. Hence, we actually have proved that the following Corollary. 
\begin{cor}
    The Conjecture \ref{conj:phase} is true for quivers that are mutation equivalent to $D_3$-quiver. 
\end{cor}
\subsection{Outline}
We will introduce quiver varieties and quiver mutations in Section \ref{sec:quiver}, including star-shaped quivers, $D_3$-type quiver, and their mutations. We will introduce the Gromov-Witten theory in Section \ref{Sec:GW} and equivariant quasimap small $I$-functions in Section \ref{sec:equivIfunction} which includes the equivariant quasimap small $I$-functions of all examples we deal with. 
We will leave all proofs of Theorem \ref{thm:1st} and Theorem \ref{thm:2nd} in Section \ref{sec:proof}.
\section{Introduction to Quiver Varieties and Quiver Mutations}\label{sec:quiver} 
In Section \ref{sec:quivervariety}, we will introduce some basic definitions for quiver varieties, including prominent examples like general star-shaped quiver and $D_3$-type quiver varieties. We refer readers to the excellent book \cite{quiver} for an introduction to quiver varieties.
In Section \ref{quivermutation}, we will introduce the quiver mutation, perform quiver mutations to star-shaped quivers, and then construct the corresponding varieties. In Section \ref{sec:D3mutequivquivers}, we will find all quivers that are mutation equivalent to $D_3$-type quiver and the corresponding varieties. 
\subsection{Quiver varieties}\label{sec:quivervariety}
An input data of a GIT quotient consists of the following ingredients:
\begin{enumerate}[(a)]
    \item an affine algebraic variety $V=\op{Spec}(A)$ over $\C$ with at most lci singularities;
    \item a connected reductive algebraic group $G$ acting on $V$;
    \item a character $\theta$ in the character group of $G$ denoted by $\chi(G):=\op{Hom}(G,\C^*)$. 
\end{enumerate}
Each character $\theta \in \chi(G)$ determines an one-dimensional representation $\C_\theta$ of $G$ and a line bundle over $V$,
\begin{equation}
    L_\theta:=V\times \C_\theta\in \text{Pic}^G(V)\,.
\end{equation}
\begin{dfn} 
     Given an input data $(V,G,\theta)$,
    $x\in V$ is called $\theta$-semistable if 
     $\exists\, k>0$ and $s\in {H}^0(V, {L}_\theta^k)^G$, such that $s(x)\neq 0$ and every $G$-orbit in $D_s=\{s\neq 0\}$ is closed. 
     Further, a $\theta$-semistable point $x\in V$ is called $\theta$-stable if its
    stabilizer $\op{Stab}_{G}(x)=\{ g\in G, g\cdot x=x\}$ is finite.
Let  $V^{ss}_\theta (G)$ denote the set of semistable points, $V^{s}_\theta(G)$ the set of stable points, and $V^{us}_\theta(G)$ the set of unstable points. 
The GIT quotient of $(V, G, \theta)$ is defined as $V\sslash_\theta G:=V^{ss}_\theta (G)\slash G$. 
\end{dfn}
The following will be assumed throughout. 
\begin{enumerate}[(a)]
    \item $V^s=V^{ss}\neq \emptyset$.
    \item The subscheme $V^s$ is nonsingular. 
    \item The group $G$ acts freely on $V^s$.
\end{enumerate}
Therefore, the GIT quotient $V\sslash_\theta G$ is smooth.
\begin{dfn}[\cite{quiver}]\label{dfn:quiver}
A  {quiver diagram} is a finite oriented graph consisting of $(Q_f\subset Q_0, Q_1, W)$ where
\begin{itemize}
    \item $Q_0$ is the set of vertices among which $Q_f$ is the set of frame (frozen) nodes, denoted by $\framebox(3,3){}$ in the graph, and $Q_0\backslash Q_f$ is the set of gauge nodes, denoted by $\bigcirc $;
    \item $Q_1$ is the set of arrows; an arrow from nodes $i$ to $j$ is denoted by $i\rightarrow j\in Q_1$,  and the number of such arrows is denoted by $b_{ij}$;
    \item a cycle is a path $i_0\rightarrow i_1\rightarrow\ldots\rightarrow i_k\rightarrow i_0$ starting from and ending at some node $i_0$; and the potential $W$ is defined as a function on cycles. 
\end{itemize}
\end{dfn}
We always assume that the quiver diagram has no $1$-cycle or $2$-cycles, known as the \textit{cluster quiver}.

\begin{dfn}\label{dfn:quivervariety}
A decorated quiver consists of a quiver with potential function $\mathbf Q=(Q_f\subseteq Q_0, Q_1, W)$ together with an integer vector $\vec v=(N_i)_{i\in Q_0}\in \mathbb Z_{>0}^{\abs{Q_0}}$ where $\abs{Q_0}$ is the number of nodes.
Those give rise to input data for a GIT quotient $(V, G, \theta)$ where $V=\bigoplus_{i\rightarrow j\in Q_1}\C^{N_i\times N_j}$, $G=\prod_{i\in Q_0\backslash Q_f}GL(N_i)$, and $\theta$ is a chosen character of $G$. 
We firmly fix the action of $G$ on $V$ in the following way. For each $g=(g_i)_{i\in Q_0\backslash Q_f}\in G$ and each $A=(A_{i\rightarrow j})_{i\rightarrow j\in Q_1}\in V$ with $A_{i\rightarrow j}$ an $N_i\times N_j$ matrix in the vector space $\C^{N_i\times N_j}$, we have
\begin{equation}\label{eqn:sec2GonV}
    g\cdot (A_{i\rightarrow j})=(g_iA_{i\rightarrow j}g_{j}^{-1})\,.
\end{equation}
For a fixed character $\theta\in \chi(G)$
\begin{equation}\label{eqn:polarizationchar}
\theta(g)=\prod_{i\in Q_{0}\backslash Q_f} \det(g_i)^{\sigma_i}\,,\,\,\sigma_i\in \R\,,
\end{equation}
the quiver variety is defined as the GIT quotient $V\sslash _{\theta}G$.
For each cycle $k_1\rightarrow k_2\rightarrow \cdots\rightarrow k_1$ in the quiver diagram, there is a $G$-invariant function on $V$,
\begin{equation}
    tr(A_{k_1\rightarrow k_2}\cdots A_{k_i\rightarrow k_1})\,.
\end{equation}
The potential $W$ is a sum of such $G$-invariant functions on cycles.
\end{dfn}
There is usually no arrow between two frame nodes in a quiver diagram. 
Whenever there is an arrow $i\rightarrow j\in Q_1$ ending at (starting from) a frame node $j$ ($i$), there is a torus $(\C^*)^{N_j}$ ($(\C^*)^{N_i}$) acting on $A_{i\rightarrow j}$ as 
$t\cdot A_{i\rightarrow j}=A_{i\rightarrow j}t^{-1}$ ($t\cdot A_{i\rightarrow j}=tA_{i\rightarrow j}$) where we view the $t$ as a diagonal matrix. 
Hence the frame nodes $Q_f$ constitute a torus action $S=\prod_{i\in Q_f}(\C^*)^{N_i}$ on $V$, such that $(\C^*)^{N_i}$ acts on matrices of arrows starting from or ending at the node $i\in Q_f$.
It is evident that this torus action commutes with $G$, so $S$ acts on $V\sslash_\theta G$. 
\begin{dfn}\label{dfn:outgoingincoming}
Given a quiver with potential function $\mathbf Q=(Q_f\subset Q_0, Q_1,W)$, the outgoing and incoming of a node $k$ are defined as $N_f(k):=\sum_{i}[b_{ki}]_+N_i$, and $N_a(k):=\sum_{i}[b_{ik}]_+N_i$, with $[b]_{+}:=\max\{ b,0\}$ for any integer $b$.
\end{dfn}
\noindent
Notice that the potential $W$ is $G$-invariant, so $W$ descends to a function on $V\sslash_\theta G$.
\begin{ex}\label{ex:star}
We start from a $D_3$-type quiver in Figure \ref{fig:star1} with an additional condition,
\begin{align}\label{eqn:Ns}
    N_4> N_3,\,N_3> N_1\geq 1,\,N_3> N_2\geq 1,\,N_4= N_1+N_2\,.
\end{align} 
\begin{figure}[H]
    \centering
    \begin{tikzpicture}   
        \node[draw, circle, minimum size=0.8cm] (node-3) at (0,0){$N_3$};
        \node[draw, circle, minimum size=0.8cm] (node-1) at (-2,1){$N_1$};
        \node[draw, circle, minimum size=0.8cm] (node-2) at (-2,-1){$N_2$};
        \node[draw, minimum width=0.8cm, minimum height=0.8cm] (node-4) at (2,0){$N_4$};
       \draw[-stealth] (node-1) -- (node-3); 
        \draw[-stealth] (node-2) -- (node-3);
        \draw[-stealth] (node-3) -- (node-4);
        \node at (-2,0.4){$1$};
        \node at (-2,-0.4){$2$};
        \node at (0,-0.7){$3$};
        \node at (2,-0.7){$4$};
        \node at (-1,0.7){$A_1$};
        \node at (-1,-0.3){$A_2$};
        \node at (1,0.2){$A_3$};
    \end{tikzpicture}
    \caption{The numerals are used to label the nodes.}
    \label{fig:star1}
\end{figure}
\noindent
We denote this $D_3$ quiver by $\mathbf Q_0$.
Denote $V_0=\C^{N_1\times N_3}\oplus \C^{N_2\times N_3}\oplus \C^{N_3\times N_4}$, and $G_0=\prod_{i=1}^3GL(N_i)$.
Choose the phase \eqref{eqn:polarizationchar} as
\begin{equation}
    \sigma_i>0, \text{ for each } i\,.
\end{equation}
Let $G_0$ act on $V_0$ in the standard way \eqref{eqn:sec2GonV}. Then we can obtain the corresponding quiver variety which we denote by $\mc X_0:=V_0\sslash_{\theta_0}G_0$.
\end{ex}
\begin{dfn}\label{defn:starshapedquiver}
    A \text{star-shaped quiver} $(Q_f\subset Q_0,Q_1,W)$ is defined by the following conditions,
\begin{itemize}
    \item it is acyclic, which implies $W=0$,
    \item  there are only single arrows, which means the number of arrows between two nodes is at most $1$,
    \item the quiver diagram is star-shaped at a gauge node $k$, which means the node $k$ can have several incoming arrows and outgoing arrows, and all remaining nodes have at most 2 adjacent arrows,
    \item for each arrow $i\rightarrow j$ with $i\neq k$, we have $N_j>N_i$; at node $k$, $N_f(k)>N_k$, and $N_f(k)\geq N_a(k)$, 
    \item each outgoing path of node $k$ ends at a frozen node. 
\end{itemize}
\end{dfn}
\begin{ex}\label{ex:generalstar}
We consider a star-shaped quiver in Figure \ref{fig:gengeralstar} with two outgoing arrows and two incoming arrows in this example.
\begin{figure}[H]
    \centering
    \begin{tikzpicture}   
        \node[draw, circle, minimum size=0.8cm] (node-5) at (0,0){$N_5$};
        \node[draw, circle, minimum size=0.8cm] (node-3) at (-1.5,1){$N_3$};
        \node[draw, circle, minimum size=0.8cm] (node-4) at (-1.5,-1){$N_4$};
        \node[draw, circle, minimum size=0.8cm] (node-1) at (-3.5,1){$N_1$};
        \node[draw, circle, minimum size=0.8cm] (node-2) at (-3.5,-1){$N_2$};        
       \node[draw, circle, minimum size=0.8cm] (node-6) at (1.5,1){$N_6$};
        \node[draw, circle, minimum size=0.8cm] (node-7) at (1.5,-1){$N_7$};
        \node[draw, minimum width=0.8cm, minimum height=0.8cm] (node-8) at (3.5,1){$N_8$};
        \node[draw, minimum width=0.8cm, minimum height=0.8cm] (node-9) at (3.5,-1){$N_9$};
       \draw[-stealth] (node-1) -- (node-3); 
        \draw[-stealth] (node-2) -- (node-4);
        \draw[-stealth] (node-4) -- (node-5);
        \draw[-stealth] (node-3) -- (node-5);
        \draw[-stealth] (node-5) -- (node-6);
        \draw[-stealth] (node-5) -- (node-7);
        \draw[-stealth] (node-6) -- (node-8);
        \draw[-stealth] (node-7) -- (node-9);
        \node at (-3.5,0.4){$1$};
        \node at (-1.5,0.4){$3$};
        \node at (-1.5,-1.6){$4$};
        \node at (0,-0.7){$5$};
        \node at (-3.5,-1.6){$2$};
         \node at (3.5,0.4){$8$};
        \node at (1.5,0.4){$6$};
        \node at (1.5,-1.6){$7$};
        \node at (3.5,-1.6){$9$};
         \node at (-2.5,1.2){$A_1$};
         \node at (-2.5,-0.8){$A_2$};
        \node at (-0.7,0.7){$A_3$};
        \node at (-0.7,-0.8){$A_4$};
        \node at (0.7,0.8){$A_5$};
        \node at (0.7,-0.8){$A_6$};
        \node at (2.5,1.2){$A_8$};
         \node at (2.5,-0.8){$A_9$};
    \end{tikzpicture}
    \caption{A star-shaped quiver with two outgoing arrows and two incoming arrows.}
    \label{fig:gengeralstar}
\end{figure}
\noindent
We have $N_j>N_i$ for each arrow $i\rightarrow j$, $i\neq 5$,  and $N_6+N_7>N_5$, $N_6+N_7\geq N_3+N_4$. 
Denote input data for the GIT quotient by $(V_s,G_s,\theta_s)$ where the subscript $s$ represents the star-shaped.
In the phase
\begin{equation}\label{eqn:phaseofstar}
    \sigma_i>0,\,\,i=1,\ldots,7,
\end{equation}
we have
\begin{equation}
    V_s^{ss}(G_s)=\{A_i|\,A_1,A_2,A_3,A_4,A_7,A_8,\begin{bmatrix}A_5&A_6\end{bmatrix} \text{ are non-degenerate}\}.
\end{equation}
The quiver variety is the GIT quotient $\mc X_s:=V_s\sslash_{\theta_s}G_s$.
\end{ex}
\begin{ex}\label{ex:ggeneral}
    Consider a general star-shaped quiver as in Definition \ref{defn:starshapedquiver}, which is as shown in Figure \ref{fig:regenealstar}. The vertical dots represent several nodes. We have used $j_1,\ldots,j_h$ to represent the incoming nodes and $i_1,\ldots,i_l$ to represent the outgoing nodes. The conditions for integers are 
    \begin{align}
        &N_{j}>N_i, \text{ if } \exists \,i\rightarrow j\in Q_1 \text{ and } i\neq k,\nonumber\\
        & N_f(k)>N_k,\,N_f(k)\geq N_a(k)\,.
    \end{align}
    \begin{figure}[H]
    \centering
    \begin{tikzpicture}   
        \node[draw, circle, minimum size=0.8cm] (node-5) at (0,0){$N_k$};
        \node[draw, circle, minimum size=0.8cm] (node-3) at (-1.5,1.2){$N_{j_1}$};
         \node[] (node-dl) at (-1.5,0){$\vdots$};
         \node[] (node-dll) at (-3,0){};
        \node[draw, circle, minimum size=0.8cm] (node-4) at (-1.5,-1.2){$N_{j_h}$};
        \node[] (node-1) at (-3,1.2){};
        \node[] (node-2) at (-3,-1.2){};        
       \node[draw, circle, minimum size=0.8cm] (node-6) at (1.5,1.2){$N_{i_1}$};
       \node[] (node-dr) at (1.5,0){$\vdots$};
       \node[] (node-drr) at (3,0){};
        \node[draw, circle, minimum size=0.8cm] (node-7) at (1.5,-1.2){$N_{i_l}$};
        \node[] (node-8) at (3,1.2){};
        \node[] (node-9) at (3,-1.2){};
       \draw[-stealth] (node-1) -- (node-3); 
       \draw[-stealth] (node-dll)--(node-dl);
       \draw[-stealth] (node-dl)--(node-5);
        \draw[-stealth] (node-2) -- (node-4);
        \draw[-stealth] (node-4) -- (node-5);
        \draw[-stealth] (node-3) -- (node-5);
        \draw[-stealth] (node-5) -- (node-6);
        \draw[-stealth] (node-5) -- (node-7);
        \draw[-stealth] (node-6) -- (node-8);
        \draw[-stealth] (node-7) -- (node-9);
        \draw[-stealth] (node-5)--(node-dr);
        \draw[-stealth] (node-dr)--(node-drr);
    \end{tikzpicture}
     \caption{A general star-shaped quiver}
        \label{fig:regenealstar}
\end{figure}
    \noindent
    Let the input data for GIT quotient be $(V_g, G_g,\theta_g)$ where the subscript g represents the word general. 
    For character $\theta_g=\prod_{i\in Q_0\backslash Q_f}\det(g_i)^{\sigma_i}$, choose phase $\sigma_i>0,\forall i\in Q_0\backslash Q_f$. 
    Then
\begin{align}
    V^{ss}_{\theta_g}(G_g)=
    \left\{A_{i\rightarrow j}|\, \text{matrices } A_{i\rightarrow j} \text{ for }i\neq k \text{ and matrix }\begin{bmatrix}A_{k\rightarrow i_1} \\ \vdots\\ A_{k\rightarrow i_l}\end{bmatrix} \text{ nondegenerate} \right\}
\end{align}
Denote the quiver variety by $\mc X_g:=V_g\sslash_{\theta_g} G_g$.
\end{ex}
\subsection{Quiver Mutation}\label{quivermutation}
We introduce the quiver mutation applet in this section.
Fix a decorated quiver with potential $\mathbf Q=(Q_f\subseteq Q_0,Q_1,W)$ and an integer vector $\vec v=(N_i)_{i\in Q_0}$.
\begin{dfn}\label{dfn:mutation}
A quiver mutation at a specific gauge node $k$, denoted by $\mu_k$, is defined by the following steps,
\begin{itemize}
    \item \textbf{Step (1)} For each path $i\rightarrow k\rightarrow j$ passing through $k$, add another arrow $i\rightarrow j$, invert directions of all arrows that start from and end at the node $k$, and denote the new arrows by $j\xrightarrow[]{*} k, k\xrightarrow[]{*} i$. 
    \item \textbf{Step (2)} Convert $N_k$ to $N'_k=\max(N_f(k), N_a(k))-N_k$, where $N_a(k)$ and $N_f(k)$ are defined in Definition \ref{dfn:outgoingincoming}. 
    \item \textbf{Step (3)} Remove all pairs of opposite arrows between two nodes introduced by the mutation until all arrows between the two nodes are in a unique direction. 
     \item \textbf{Step (4)} 
     Replace the path $i\rightarrow k\rightarrow j$ by $i\rightarrow j$ whenever it appears in the potential $W$.
     Add a new cubic term of the $3$-cycle $j\xrightarrow[]{*}k\xrightarrow[]{*} i\rightarrow j$ to $W$. Denote the resulting new potential by $ W'$. 
\end{itemize}
\end{dfn}
There are subtleties for the potential in Step $(4)$ when the potential $W$ is nontrivial, and we cannot delete the terms containing annihilated arrows in step $(2)$ directly. Instead, for each path $i\rightarrow k\rightarrow j$, denote the matrices of its inverted arrows $j\xrightarrow[]{*} k\,,\,k\xrightarrow[]{*} i $  by  $A_{j\rightarrow k}^*,\, A_{k\rightarrow i}^*$ and the matrix of the added arrow $i\rightarrow j$ by $A_{i\rightarrow j}$. 
We rewrite the original potential as $W=W_0+W_1$ where $W_0$ contains all terms with the factor $A_{i\rightarrow k}A_{k\rightarrow j}$ for each path $i\rightarrow k\rightarrow j$. Then the step ${(4)}$ in Definition \ref{dfn:mutation} converts $W$ to $W'=\sum A_{j\rightarrow k}^*A_{k\rightarrow i}^*A_{i\rightarrow j}+W_0'+W_1$ where the sum is over all new cubic terms arising from the 3-cycles $j\xrightarrow[]{*} k\xrightarrow[]{*} i\rightarrow j$ and $W_0'$ is obtained by replacing $A_{i\rightarrow k}A_{k\rightarrow j}$ in $W_0$ by $A_{i\rightarrow j}$. 
There might be a quadratic term $tr(A_{i\rightarrow j}A_{j\rightarrow i})$ in $W_0'$ when $W_0$ has a $3$-cycle containing the path $i\rightarrow k\rightarrow j$. 
Whenever this happens, we need to consider the constraints 
\begin{align}\label{eqn:constraints}
    \frac{\partial W'}{\partial A_{i\rightarrow j}^{ab}}=0\,,\,
    \frac{\partial W'}{\partial A_{j\rightarrow i}^{ab}}=0
\end{align}
where we write the potential $W$ as function on entries of matrices and $A_{i\rightarrow j}^{ab}$ denotes the $(a,b)$-entry of the matrix $A_{i\rightarrow j}$ and so does $A_{j\rightarrow i}^{ab}$. 
Replace $A_{i\rightarrow j}$ and $A_{j\rightarrow i}$ in $W'$ accordingly by the constraints in \eqref{eqn:constraints}, and we obtain the new potential $\tilde W$ in the dual side. See \cite[Section 5]{clusterpot1} and \cite[Section 3.4]{benini2015cluster}. 
This subtlety happens in Example \ref{ex:m1m3}. 

Via the quiver mutation and the above recipe for the potential function, we obtain a new quiver with superpotential, denoted by $\tilde{\mathbf Q}=(\tilde Q_f\subset \tilde Q_0,\tilde Q_1,\tilde W)$.
Quiver mutations do not generate any 1-cycle or 2-cycles by step $(3)$, so the resulting quiver is still a cluster quiver. 

One can check that the quiver mutation is an involution, which means $\mu_k^2=Id$, for any $k$.
\begin{rem}
\begin{enumerate}[(1)]
    \item In the above definition, we assume that  $\max(N_f(k), N_a(k))-N_k>0$. Otherwise, the resulting quiver fails to define a variety. 
    \item In the third step of quiver mutation, 
    when we remove pairs of opposite arrows, it doesn't depend on the order we remove. However it is unclear whether it depends on the choices of arrows, when $b_{ij}\neq b_{ji}$.
    This will be further studied in our future work. In this work, there is no such issue, since there will be only at most one pair of opposite arrows between two nodes in all our examples. 
\end{enumerate}
\end{rem}
In order to obtain a variety after a quiver mutation, we need to know the character $\tilde\theta$. 
We use the proposed rule in Conjecture \ref{conj:phase} to find the new character $\tilde\theta$. 
\begin{ex}\label{ex:starmc}
We perform a quiver mutation $\mu_5$ at the center node to the general star-shaped quiver introduced in Example \ref{ex:generalstar}, and get a quiver diagram in Figure \ref{fig:generalstarmc} with four 3-cycles
\begin{figure}[H]
    \centering
    \begin{tikzpicture}   
        \node[draw, circle, minimum size=0.8cm] (node-5) at (0,0){$N_5'$};
        \node[draw, circle, minimum size=0.8cm] (node-3) at (-1.5,1){$N_3$};
        \node[draw, circle, minimum size=0.8cm] (node-4) at (-1.5,-1){$N_4$};
        \node[draw, circle, minimum size=0.8cm] (node-1) at (-3.5,1){$N_1$};
        \node[draw, circle, minimum size=0.8cm] (node-2) at (-3.5,-1){$N_2$};        
       \node[draw, circle, minimum size=0.8cm] (node-6) at (1.5,1){$N_6$};
        \node[draw, circle, minimum size=0.8cm] (node-7) at (1.5,-1){$N_7$};
        \node[draw, minimum width=0.8cm, minimum height=0.8cm] (node-8) at (3.5,1){$N_8$};
        \node[draw, minimum width=0.8cm, minimum height=0.8cm] (node-9) at (3.5,-1){$N_9$};
       \draw[-stealth] (node-1) -- (node-3); 
        \draw[-stealth] (node-2) -- (node-4);
        \draw[-stealth] (node-5) -- (node-4);
        \draw[-stealth] (node-5) -- (node-3);
        \draw[-stealth] (node-6) -- (node-5);
        \draw[-stealth] (node-7) -- (node-5);
        \draw[-stealth] (node-6) -- (node-8);
        \draw[-stealth] (node-7) -- (node-9);
         \draw[-stealth] (node-3) to [bend left] (node-6);
        \draw[-stealth] (node-3) to [bend right] (node-7);
        \draw[-stealth] (node-4) to [bend left] (node-6);
        \draw[-stealth] (node-4) to [bend right] (node-7);
         \node at (-2.5,1.2){$A_1$};
         \node at (-2.5,-0.8){$A_2$};
        \node at (-0.7,0.6){$A_3$};
        \node at (-0.7,-0.6){$A_4$};
        \node at (0.7,0.6){$A_5$};
        \node at (0.7,-0.6){$A_6$};
        \node at (2.5,1.2){$A_8$};
         \node at (2.5,-0.8){$A_9$};
         \node at (0,0.9){$B_2$};
          \node at (0,1.7){$B_1$};
          \node at (0,-0.9){$B_3$};
          \node at (0,-1.7){$B_4$};
    \end{tikzpicture}
    \caption{$N_5'=N_6+N_7-N_5$}
    \label{fig:generalstarmc}
\end{figure}
 \noindent
It has a potential function
\begin{equation}
    \tilde W_s=tr(B_1A_5A_3)+tr(B_2A_5A_4)+tr(B_3A_6A_3)+tr(B_4A_6A_4).
\end{equation}
Denote the input data for the GIT by 
$(\tilde V_s, \tilde G_s, \tilde \theta_s)$.
We denote the character $\tilde \theta_s(\tilde g)=\prod_{i\in Q_0\backslash Q_f}\det(\tilde g_i)^{\tilde \sigma_i}$ temporarily. 
According to the Conjecture \ref{conj:phase}, 
$\tilde \sigma_6=\sigma_5+\sigma_6, \tilde \sigma_7=\sigma_5+\sigma_7, \tilde\sigma_5=-\sigma_5$, and $\tilde\sigma_i=\sigma_i$, for $i=1,2,3,4$. 
Then substitute $\sigma_6=\tilde\sigma_6-\sigma_5=\tilde\sigma_6+\tilde\sigma_5$, $\sigma_7=\tilde\sigma_7+\tilde\sigma_5$ to equality \ref{eqn:phaseofstar}, and we have  
\begin{equation}\label{eqn:phaseofgeneralstarmc}
    \tilde \sigma_i>0, \text{ for } i\neq 5,6,7,\, \tilde \sigma_5<0, \, \tilde \sigma_5+\tilde \sigma_6>0,\,\tilde \sigma_5+\tilde \sigma_7>0\,.
\end{equation}

Consider the critical locus of the potential function $\tilde W_s$, which we denote by $\tilde Z_s=d(\tilde W_s)$. One can check that 
\begin{align}
   &\tilde Z_s^{ss}(\tilde G_s)=
   \bigg\{\begin{bmatrix}B_1&B_3
   \end{bmatrix}
   \begin{bmatrix}A_5\\A_6\end{bmatrix}=0,
   \begin{bmatrix}B_2&B_4\end{bmatrix}
   \begin{bmatrix}A_5\\A_6\end{bmatrix}=0, A_3=0,A_4=0\,\nonumber|
    \\
   & A_1,A_2,A_7,A_8, \begin{bmatrix}B_1&B_3\end{bmatrix},\,\begin{bmatrix}B_2&B_4\end{bmatrix},\,\begin{bmatrix}A_5\\A_6\end{bmatrix} \text{ non-degenerate} 
    \bigg\}.
\end{align}
    Consider another quiver obtained by deleting the arrows $5\rightarrow 3$ and $5\rightarrow 4$ of the quiver in Figure \ref{fig:generalstarmc} whose matrices vanish in $\tilde Z_s^{ss}(\tilde G_s)$. 
    Denote the input data of the resulting new quiver by $(\overline V_s, \tilde G_s,\tilde \theta_s)$.
    It has the same gauge group $\tilde G_s$ and character $\tilde\theta_s$ as above.
    Then
\begin{align}
       \overline V_{s, \tilde\theta_s}^{ss}(\tilde G_s)=\bigg\{A_i,B_j| A_1,A_2,A_7,A_8,\begin{bmatrix}B_1&B_3\end{bmatrix},\,\begin{bmatrix}B_2&B_4\end{bmatrix},\,\begin{bmatrix}A_5\\A_6\end{bmatrix} \text{ non-degenerate}
    \bigg\}.
\end{align}
Denote the new quiver variety by $\tilde X_s:=\overline V_s\sslash_{\tilde \theta_s} \tilde G_s$. Then
$\tilde{\mc Z}_s$ is a subvariety of $\tilde{\mc X}_s$ defined by 
\[\begin{bmatrix}B_1&B_3\end{bmatrix}\begin{bmatrix}A_5\\A_6\end{bmatrix}=0,\begin{bmatrix}B_2&B_4\end{bmatrix}\begin{bmatrix}A_5\\A_6\end{bmatrix}=0 .\]
\end{ex}
\begin{ex}\label{ex:ggeneralstarmc}
    In this example, we consider the quiver mutation to the general star-shaped quiver in Example \ref{ex:ggeneral} at the center node $k$, and obtain the quiver with potential in Figure \ref{fig:regenealstarmc}.
       \begin{figure}[H]
    \centering
    \begin{tikzpicture}   
        \node[draw, circle, minimum size=0.8cm] (node-5) at (0,0){$N_k'$};
        \node[draw, circle, minimum size=0.8cm] (node-3) at (-1.5,1.2){$N_{j_1}$};
         \node[] (node-dl) at (-1.5,0){$\vdots$};
         \node[] (node-dll) at (-3,0){};
        \node[draw, circle, minimum size=0.8cm] (node-4) at (-1.5,-1.2){$N_{j_h}$};
        \node[] (node-1) at (-3,1.2){};
        \node[] (node-2) at (-3,-1.2){};        
       \node[draw, circle, minimum size=0.8cm] (node-6) at (1.5,1.2){$N_{i_1}$};
       \node[] (node-dr) at (1.5,0){$\vdots$};
       \node[] (node-drr) at (3,0){};
        \node[draw, circle, minimum size=0.8cm] (node-7) at (1.5,-1.2){$N_{i_l}$};
        \node[] (node-8) at (3,1.2){};
        \node[] (node-9) at (3,-1.2){};
       \draw[-stealth] (node-1) -- (node-3); 
       \draw[-stealth] (node-dll)--(node-dl);
       \draw[-stealth] (node-5)--(node-dl);
        \draw[-stealth] (node-2) -- (node-4);
        \draw[-stealth] (node-5) -- (node-4);
        \draw[-stealth] (node-5) -- (node-3);
        \draw[-stealth] (node-6) -- (node-5);
        \draw[-stealth] (node-7) -- (node-5);
        \draw[-stealth] (node-6) -- (node-8);
        \draw[-stealth] (node-7) -- (node-9);
        \draw[-stealth] (node-dr)--(node-5);
        \draw[-stealth] (node-dr)--(node-drr);
        \draw[-stealth] (node-3) to [bend left] (node-6);
        \draw[-stealth] (node-3) to [bend left] (node-7);
        \draw[-stealth] (node-3) to [bend left] (node-dr);
        \draw[-stealth] (node-4) to [bend right] (node-6);
        \draw[-stealth] (node-4) to [bend right] (node-7);
        \draw[-stealth] (node-4) to [bend right] (node-dr);
        \draw[-stealth] (node-dl) to [bend left] (node-6);
        \draw[-stealth] (node-dl) to [bend left] (node-dr);
        \draw[-stealth] (node-dl) to [bend right] (node-7);
    \end{tikzpicture}
      \caption{The potential function is $W=\sum_{a=1}^{h}\sum_{b=1}^l
        tr(A_{j_a\rightarrow i_b}A_{i_b\rightarrow k}A_{k\rightarrow j_a}),\, N_k'=N_f(k)-N_k$.}
        \label{fig:regenealstarmc}
\end{figure}
    \noindent
The proposed phase is 
\begin{align}\label{eqn:phaseforgeneralstar}
    \sigma_k<0;\,\, \sigma_i>0 \text{ for }i\neq k,i_1,\ldots,i_l;\,\, \sigma_k+\sigma_{i_b}>0\,,b=1,\ldots,l.
\end{align}
The semistable locus of critical locus $Z_g=\{dW=0\}$ is 
\begin{align}\label{eqn:semistgenstar}
   Z_g^{ss}(\tilde G_g)=&\Bigg\{
   \begin{bmatrix}
   A_{j_a\rightarrow i_1}&\cdots& A_{j_a\rightarrow i_l}
   \end{bmatrix}
   \begin{bmatrix}
   A_{i_1\rightarrow k}\\ \cdots\\ A_{i_l\rightarrow k}
   \end{bmatrix}=0, \,
   A_{j_a\rightarrow k}=0,\, a=1\ldots h\, \nonumber \\
   &\big|\,\text{matrices } 
   \begin{bmatrix}
   A_{j_a\rightarrow i_1}&\cdots& A_{j_a\rightarrow i_l}
   \end{bmatrix} a=1\ldots h \text{ nondegenerate};
   \nonumber\\
   &
   \text{matrix} 
   \begin{bmatrix}
   A_{i_1\rightarrow k}\\ \cdots\\ A_{i_l\rightarrow k}
   \end{bmatrix} \text{ nondegenerate}
   \Bigg\}
\end{align}
Denote the GIT quotient of the critical locus by $\tilde{\mc Z}_g=Z_g^{ss}\slash\tilde G_g$.
\end{ex}
\subsection{Quivers that are mutation equivalent to $D_3$ quiver}\label{sec:D3mutequivquivers}
In this section, we will find all quivers that are mutation equivalent to the $D_3$-type quiver. Let $\Omega$ denote such a set. Furthermore, we will find the correct phase proposed by Conjecture \ref{conj:phase} for each quiver, and find the corresponding variety for each one. We will first perform $\mu_3\rightarrow \mu_1\rightarrow \mu_2$ repeatedly which gets most quivers in $\Omega$, and then we will give the remaining elements in $\Omega$. 

\begin{ex}\label{ex:m3}
We perform a quiver mutation $\mu_3$ to the quiver diagram in Figure \ref{fig:star1}, 
and obtain a quiver in Figure \ref{fig:mu3} with a potential function $W_1=tr(B_1A_3A_{1})+tr(B_2A_3A_2)$. We denote this quiver by $\mathbf Q_1$.
\begin{figure}[H]
    \centering
    \begin{tikzpicture}   
        \node[draw, circle, minimum size=0.8cm] (node-3) at (0,0){$N_3'$};
        \node[draw, circle, minimum size=0.8cm] (node-1) at (-2,1){$N_1$};
        \node[draw, circle, minimum size=0.8cm] (node-2) at (-2,-1){$N_2$};
        \node[draw, minimum width=0.8cm, minimum height=0.8cm] (node-4) at (2,0){$N_4$};
       \draw[-stealth] (node-3) -- (node-1); 
        \draw[-stealth] (node-3) -- (node-2);
        \draw[-stealth] (node-4) -- (node-3);
        \draw[-stealth] (node-1) to [bend left](node-4);
         \draw[-stealth] (node-2) to [bend right](node-4);
        \node at (-1,0.7){$A_1$};
        \node at (-1,-0.2){$A_2$};
        \node at (1,0.2){$A_3$};
         \node at (0,1.4){$B_1$};
          \node at (0,-1){$B_2$};
    \end{tikzpicture}
   \caption{$N_3'=N_4-N_3$}
    \label{fig:mu3}
\end{figure}
\noindent
Denote the input data of the quiver by $(V_1 ,G_1,\theta_1)$.
Under the rule in Conjecture \ref{conj:phase}, the phase for $G_1$ is 
\begin{equation}\label{eqn:phasemu3}
 \sigma_{1}>0\,,\sigma_2>0\,, \sigma_3<0.
\end{equation}
Consider the critical locus of the potential $Z_1:=Z(d W_1)$, which is equivalent to the following equations,
\begin{equation}\label{eqn:1}
    B_1A_3=0\,,\,B_2A_3=0\,,\,A_3A_{1}=0\,,\, A_3A_2=0\,,\, A_1B_1+A_2B_2=0\,.
\end{equation}
Then $$Z^{ss}_{1,\theta_1}(G_1)=\{A_1=0,\,A_2=0,\,  B_1A_3=0,\,B_2A_3=0|\,B_1,\,B_2,\,A_3 \text{ non-degenerate }\}.$$
Consider another quiver obtained by deleting arrows $3\rightarrow 1, 3\rightarrow 2$ whose matrices in $Z^{ss}_{1,\theta_1}(G_1)$ vanish.
 Let $(\tilde V_1, G_1,\theta_1)$ be the input data of the new quiver where $\theta_1$ is as \eqref{eqn:phasemu3}. Let $\mc X_1:=\tilde V_1\sslash _{\theta_1}G_1$ be the corresponding quiver variety. 
The GIT quotient $\mc Z_1:=Z^{ss}_{1,\theta_1} /G_1$ is a subvariety in $\mc X_1$ defined by equations
\begin{equation}\label{eqn:Z1}
   B_1A_3=0, \, B_2A_3=0\,.
\end{equation}
\end{ex}
\begin{ex}\label{ex:m1m3}
We perform another quiver mutation $\mu_1$ to the quiver in Figure \ref{fig:mu3} and get that in Figure \ref{fig:m1m3}. We denote this new quiver by $\mathbf Q_2$.
\begin{figure}[H]
    \centering
    \begin{tikzpicture}   
        \node[draw, circle, minimum size=0.8cm] (node-3) at (0,0){$N_3'$};
        \node[draw, circle, minimum size=0.8cm] (node-1) at (-2,1){$N_2$};
        \node[draw, circle, minimum size=0.8cm] (node-2) at (-2,-1){$N_2$};
        \node[draw, minimum width=0.8cm, minimum height=0.8cm] (node-4) at (2,0){$N_4$};
       \draw[-stealth] (node-1) -- (node-3); 
        \draw[-stealth] (node-3) -- (node-2);
        \draw[dashed,->] (1.6,-0.1) -- (0.5,-0.1);
         \draw[dashed,->] (0.5,0.1) -- (1.6,0.1);
        \draw[-stealth] (node-4) to [bend right](node-1);
         \draw[-stealth] (node-2) to [bend right](node-4);
        \node at (-1,0.7){$A_1$};
        \node at (-1,-0.2){$A_2$};
        \node at (1,0.3){$A_{12}$};
        \node at (1,-0.3){$A_{3}$};
         \node at (0,1.4){$B_1$};
          \node at (0,-1){$B_2$};
          \node at (-2,0.4){$1$};
          \node at (-2,-0.4){$2$};
          \node at (0,0.6){$3$};
        \node at (2,-0.6){$4$};
    \end{tikzpicture}
    \caption{The integer assigned to node 1 is $N_4-N_1=N_2$. The dashed opposite arrows are annihilated.}
    \label{fig:m1m3}
\end{figure}
\noindent
To construct the new potential, we first replace the factor $A_1B_1$ by the matrix $A_{12}$ and add a new cubic term $tr(B_1A_1A_{12})$  arising from the 3-cycle $4\rightarrow 1\rightarrow 3\rightarrow 4$, and we get a new potential 
$W_2'=tr(A_{12}A_3)+tr(A_2B_2A_3)+tr(B_1A_1A_{12})$. In $ W_2'$, there is a quadratic term $tr(A_{12}A_3)$, so we take the derivative to $W'$ in terms of these two factors
\begin{equation}
    d_{A_{12}^{ab}} W_2'=0,\,\,d_{A_3^{ab}}W_2'=0\,,
\end{equation}
and get constraints for $W_2'$,
\begin{equation}
    A_3+B_1A_1=0,\,\, A_{12}+A_2B_2=0\,.
\end{equation}
Substituting $A_{3}=-A_2B_2$ and $A_{12}=-A_2B_2$, we get the potential  $W_2=tr(B_1A_1A_2B_2)$, where we have neglected the negative sign in front of it.

Let $(V_2, G_2, \theta_2)$ be the input data for the quiver Figure \ref{fig:m1m3}. 
Let $Z_2=\{dW_2=0\}\subset V_2$.
In the proposed phase
\begin{equation}\label{eqn:Q2phase}
\sigma_1<0,\,\,\sigma_2>0,\,\,\sigma_3<0\,.
\end{equation}
the semistable locus is 
\begin{equation}
  Z^{ss}_{2,\theta_2}(G_2)=
  \{A_2=0,B_2B_1A_1=0| B_1, A_1,B_2 \text{ nondegenerate}\}\,.
\end{equation}
Consider a new quiver by deleting the arrow $3\rightarrow 2$ in the Figure \ref{fig:m1m3}.
Denote the corresponding input data of the quiver by $(\tilde V_2, G_2, \theta_2)$, and denote its quiver variety by $\mc X_2:=\tilde V_2\sslash_{\theta_2} G_2$.
We find that the variety  $\mc Z_2= Z_2 \sslash_{\theta_2} G_2$ is a subvariety in the quiver variety $\mc X_2$.
\end{ex}
\begin{ex}\label{ex:m2m1m3}
performing $\mu_2$ to the quiver in Figure \ref{fig:m1m3}, 
we obtain a new quiver in Figure \ref{fig:m2mu1mu3}, which we denote by $\mathbf Q_3$.
\begin{figure}[H]
    \centering
    \begin{tikzpicture}   
        \node[draw, circle, minimum size=0.8cm] (node-3) at (0,0){$N_3'$};
        \node[draw, circle, minimum size=0.8cm] (node-1) at (-2,1){$N_2$};
        \node[draw, circle, minimum size=0.8cm] (node-2) at (-2,-1){$N_1$};
        \node[draw, minimum width=0.8cm, minimum height=0.8cm] (node-4) at (2,0){$N_4$};
       \draw[-stealth] (node-1) -- (node-3); 
        \draw[-stealth] (node-2) -- (node-3);
        \draw[->] (node-3) -- (node-4);
        \draw[-stealth] (node-4) to [bend right](node-1);
         \draw[-stealth] (node-4) to [bend left](node-2);
        \node at (-1,0.7){$A_1$};
        \node at (-1,-0.2){$A_2$};
        \node at (1,0.3){$A_{3}$};
         \node at (0,1.4){$B_1$};
          \node at (0,-1){$B_2$};
          \node at (-2,0.4){$1$};
          \node at (-2,-0.4){$2$};
          \node at (0,0.6){$3$};
        \node at (2,-0.6){$4$};
    \end{tikzpicture}
    \caption{The integer assigned to node 2 is $N_4-N_2=N_1$. $N_3'=N_4-N_3$.}
    \label{fig:m2mu1mu3}
\end{figure}
\noindent 
The potential is obtained by replacing $A_2B_2$ by $A_3$ and adding the cubic term $tr(B_2A_2A_3)$, so $W_3=tr(B_1A_1A_3)+tr(B_2A_2A_3)$.
Denote the input data for the quiver variety by $(V_3, G_3, \theta_3)$.

The natural character is 
$\theta_3(g)=\det(g_1)^{\sigma_1}\det(g_2)^{\sigma_2}\det(g_3)^{\sigma_3}$ with 
\begin{equation}\label{phase:3-}
\sigma_1<0,\,\,\sigma_2<0,\,\,\sigma_3<0\,,
\end{equation}
by our Conjecture \ref{conj:phase}. 
The critical locus $Z_3=\{dW=0\}$ is equivalent to 
\begin{equation}\label{eqn:dW3}
    A_1A_3=0,\,A_2A_3=0,\, B_1A_1+B_2A_2=0,\, A_3B_1=0,\,A_3B_2=0.
\end{equation}
In the phase \eqref{phase:3-}, 
\begin{equation}
    Z^{ss}_{3,\,\theta_3}(G_3)=\{ A_3=0, \begin{bmatrix} B_1& B_2 \end{bmatrix} \begin{bmatrix} A_1\\ A_2 \end{bmatrix}=0\rvert\, B_1,B_2,
    \begin{bmatrix} A_1\\ A_2 \end{bmatrix} \text{nondegenerate}\, \}\,.
\end{equation}

Consider another quiver obtained by deleting $3\rightarrow 4$ in quiver $\mathbf Q_3$.
Denote the corresponding input data by $(U_3, G_3,\theta_3)$, and then 
\begin{equation}
    U^{ss}_{3,\,\theta_3}(G_3)=
    \{B_1,B_2,
    \begin{bmatrix}A_1\\A_2\end{bmatrix} \text{non-degenerate}\}\,.
\end{equation}
The critical locus $\mc Z_3=Z_{3,\,\theta_3}^{ss}\slash G_3$ can be viewed as a subvariety in $U_3\sslash_{\theta_3}G_3$ defined by equations
\begin{equation}
    \begin{bmatrix} B_1& B_2 \end{bmatrix} \begin{bmatrix} A_1\\ A_2 \end{bmatrix}=0
\end{equation}
\end{ex}
\begin{ex}\label{ex}
We perform quiver mutations $\mu_3\rightarrow \mu_1\rightarrow \mu_2\rightarrow\mu_3\rightarrow\mu_1\rightarrow \mu_2$ to the Figure \ref{fig:m2mu1mu3} and get quivers in Figure \ref{fig:D4-6mutations} listed from left to right in the first row and right to left in the second row.  
We label those quivers by $\{\mathbf Q_i\}_{i=4}^9$. 
For each quiver $\mathbf Q_i$, $i=4,\ldots,9$, we denote the input data for its GIT quotient by $(V_i,G_i,\theta_i)$. 
\begin{figure}[ht]
    \centering
    \begin{tikzpicture}   
        \node[draw, circle, minimum size=0.6cm] (node-a3) at (0,0){$N_3$};
        \node[draw, circle, minimum size=0.6cm] (node-a1) at (-1.8,1){$N_2$};
        \node[draw, circle, minimum size=0.6cm] (node-a2) at (-1.8,-1){$N_1$};
        \node[draw, minimum width=0.6cm, minimum height=0.6cm] (node-a4) at (1.8,0){$N_4$};
       \draw[-stealth] (node-a3) --node[midway, above]{$A_1$} (node-a1); 
        \draw[-stealth] (node-a3) --node[midway, above]{$A_2$} (node-a2);
        \draw[-stealth] (node-a4) -- node[midway, above]{$A_3$}(node-a3);
        \node at (-1.8,0.4){$1$};
        \node at (-1.8,-0.4){$2$};
        \node at (0,-0.6){$3$};
        \node at (1.8,-0.6){$4$};
        \node at (0,-1.5){$(a)$};
        \node at (2.3,1){$\xrightarrow{\mu_1}$};
        \node[draw, circle, minimum size=0.6cm] (node-b3) at (5.2,0){$N_3$};
        \node[draw, circle, minimum size=0.6cm] (node-b1) at (3.4,1){$N_2'$};
        \node[draw, circle, minimum size=0.6cm] (node-b2) at (3.4,-1){$N_1$};
        \node[draw, minimum width=0.6cm, minimum height=0.6cm] (node-b4) at (7,0){$N_4$};
       \draw[-stealth] (node-b1) -- (node-b3); 
        \draw[-stealth] (node-b3) -- (node-b2);
        \draw[-stealth] (node-b4) -- (node-b3);
        \node at (4.3,0.7){$A_1$};
        \node at (4.3,-0.2){$A_2$};
        \node at (6.1,0.2){$A_3$};
        \node at (5.2,-1.5){$(b)$};
        \node at (7.6,1){$\xrightarrow[]{\mu_2}$};
        \node[draw, circle, minimum size=0.6cm] (node-c3) at (10.4,0){$N_3$};
        \node[draw, circle, minimum size=0.6cm] (node-c1) at (8.6,1){$N_2'$};
        \node[draw, circle, minimum size=0.6cm] (node-c2) at (8.6,-1){$N_1'$};
        \node[draw, minimum width=0.6cm, minimum height=0.6cm] (node-c4) at (12.2,0){$N_4$};
       \draw[-stealth] (node-c1) -- (node-c3); 
        \draw[-stealth] (node-c2) -- (node-c3);
        \draw[-stealth] (node-c4) -- (node-c3);
        \node at (9.5,0.7){$A_1$};
        \node at (9.5,-0.2){$A_2$};
        \node at (11.4,0.2){$A_3$};
        \node at (10.4,-1.5){$(c)$};
        \node at (12.2, -2){$\downarrow \mu_3$};
        \node[draw, circle, minimum size=0.6cm] (node-d3) at (10.4,-3.5){$N_3$};
        \node[draw, circle, minimum size=0.6cm] (node-d1) at (8.6,-2.5){$N_2'$};
        \node[draw, circle, minimum size=0.6cm] (node-d2) at (8.6,-4.5){$N_1'$};
        \node[draw, minimum width=0.6cm, minimum height=0.6cm] (node-d4) at (12.2,-3.5){$N_4$};
       \draw[-stealth] (node-d3) -- (node-d1); 
        \draw[-stealth] (node-d3) -- (node-d2);
        \draw[-stealth] (node-d3) -- (node-d4);
        \node at (9.5,-2.8){$A_1$};
        \node at (9.5,-3.7){$A_2$};
        \node at (11.4,-3.2){$A_3$};
        \node at (10.4,-5){$(d)$};
        \node at (7.6,-2.5){$\xleftarrow[]{\mu_1}$};
       \node[draw, circle, minimum size=0.6cm] (node-e3) at (5.2,-3.5){$N_3$};
        \node[draw, circle, minimum size=0.6cm] (node-e1) at (3.4,-2.5){$N_2$};
        \node[draw, circle, minimum size=0.6cm] (node-e2) at (3.4,-4.5){$N_1'$};
        \node[draw, minimum width=0.6cm, minimum height=0.6cm] (node-e4) at (7,-3.5){$N_4$};
       \draw[-stealth] (node-e1) -- (node-e3); 
        \draw[-stealth] (node-e3) -- (node-e2);
        \draw[-stealth] (node-e3) -- (node-e4);
        \node at (4.3,-2.8){$A_1$};
        \node at (4.3,-3.7){$A_2$};
        \node at (6.1,-3.2){$A_3$};
        \node at (5.2,-5){$(e)$};
        \node at (2.4,-2.5){$\xleftarrow[]{\mu_2}$};
       \node[draw, circle, minimum size=0.6cm] (node-f3) at (0,-3.5){$N_3$};
        \node[draw, circle, minimum size=0.6cm] (node-f1) at (-1.8,-2.5){$N_2$};
        \node[draw, circle, minimum size=0.6cm] (node-f2) at (-1.8,-4.5){$N_1$};
        \node[draw, minimum width=0.6cm, minimum height=0.6cm] (node-f4) at (1.8,-3.5){$N_4$};
       \draw[-stealth] (node-f1) -- (node-f3); 
        \draw[-stealth] (node-f2) -- (node-f3);
        \draw[-stealth] (node-f3) -- (node-f4);
        \node at (-0.9,-2.8){$A_1$};
        \node at (-0.9,-3.7){$A_2$};
        \node at (0.9,-3.2){$A_3$};
        \node at (0,-5){$(f)$};
    \end{tikzpicture}
    \caption{In the diagram, $N_2'=N_3-N_2,N_1'=N_3-N_1$. The quivers are related via the shown mutations.}
    \label{fig:D4-6mutations}
\end{figure}
\noindent
In order to construct the corresponding quiver varieties, we fix the phases of those gauge groups under the rule in Conjecture \ref{conj:phase}, which are listed in the Table \ref{table:phases}. 

We now explain how to obtain the phase $\theta_4$ from $\theta_3$ in \eqref{phase:3-} and leave the others to readers. 
We temporarily use $\tilde \sigma_i$ to represent phase of $\theta_4$ and $\sigma_i$ phase of $\theta_3$. 
According to the Conjecture \ref{conj:phase}, $\tilde\sigma_1=\sigma_1+\sigma_3,
\tilde\sigma_2=\sigma_2+\sigma_3,
\tilde\sigma_3=-\sigma_3$, since $\sigma_3<0$. 
Then we get the three inequalities for $\tilde \sigma_i$ of $\theta_4$ by substituting $\sigma_1=\tilde \sigma_1+\tilde\sigma_3$, $\sigma_2=\tilde\sigma_2+\tilde\sigma_3,\sigma_3=-\tilde\sigma_3$ to the three inequalities in \eqref{phase:3-}.
\begin{table}[H]
\centering
\begin{tabular}{ |c | c|c|} 
  \hline
  Figure& character& Phase\\
  \hline
  (a) & $\theta_4$&$\sigma_3>0$,\, $\sigma_1+\sigma_3<0$, $\sigma_2+\sigma_3<0$ \\ 
  \hline
  $(b)$& $\theta_5$&
  $\sigma_3<0,\,
  \sigma_1+\sigma_3>0,
  \sigma_1+\sigma_2+\sigma_3<0$\\
  \hline
  $(c)$& $\theta_6$& $\sigma_1+\sigma_3<0,\sigma_2+\sigma_3<0,
  \sigma_1+\sigma_2+\sigma_3>0$ \\
  \hline
  $(d)$&$\theta_7$& $\sigma_1<0, \sigma_2<0, \sigma_1+\sigma_2+\sigma_3>0$ \\
  \hline
  $(e)$&$\theta_8$& $\sigma_1>0,\,\sigma_2<0,\,\sigma_2+\sigma_3>0$ \\
  \hline
  $(f)$&$\theta_9$& $\sigma_1>0,\,\sigma_2>0,\,\sigma_3>0$ \\
  \hline
\end{tabular}
  \caption{Phases of the quivers in Figure \ref{fig:D4-6mutations}.}
  \label{table:phases}
\end{table}
We will give the semistable loci of the $G_i$ action with character $\theta_i$, which is enough to get all quiver varieties $\mc X_i=V_i\sslash_{\theta_i}G_i$. 
\begin{subequations}
    \begin{align}
         &V^{ss}_{4,\theta_4}(G_4)=\{(A_1, A_2, A_3)\big|A_1,A_2,\,\begin{bmatrix}A_1&A_2\end{bmatrix} ,A_3A_1 ,A_3A_2 \,\, \text{all non-degenerate}\}\label{subeqn:semQ4} \\
        & V^{ss}_{5,\theta_5}=\{(A_1,A_2,A_3)\big| A_1,A_2, \begin{bmatrix}A_1\\A_3 \end{bmatrix}, A_1A_2,A_3A_2 \text{ all non-degenerate}\} \label{subeqn:semQ5} \\
        &V^{ss}_{6,\theta_6}=\{(A_1,A_2,A_3)\big| A_1,A_2, \begin{bmatrix}A_1\\A_2 \end{bmatrix}, \begin{bmatrix}A_2\\A_3 \end{bmatrix}, \begin{bmatrix}A_1\\A_3 \end{bmatrix} \text{ non-degenerate}\}.\label{subeqn:semQ6}
        \\
        & V^{ss}_{i,\theta_i}=\{(A_1,A_2,A_3) \big| A_1,A_2,A_3 \text{ non-degenerate }\} \,i=7,\,8,\,9\, \label{subeqn:semQ789}.
    \end{align}
\end{subequations}
See Appendix \ref{sec:Appsemistablelocus} for proofs of those semistable locus, which are elementary. 

One can find that the quiver variety $\mc X_9$ is exactly the same with the $D_3$ quiver variety $\mc X_0$ by exchanging the nodes $1$ and $2$.
\end{ex}
There are five more quivers that are mutation equivalent to the $D_3$-quiver. We display them in Figure \ref{fig:extrafig} and Figure \ref{fig:extrafig2}.
\begin{figure}[H]
    \centering
    \begin{tikzpicture}   
        \node[draw, circle, minimum size=0.6cm] (node-a3) at (0,0){$N_3'$};
        \node[draw, circle, minimum size=0.6cm] (node-a1) at (-1.8,1){$N_1$};
        \node[draw, circle, minimum size=0.6cm] (node-a2) at (-1.8,-1){$N_1$};
        \node[draw, minimum width=0.6cm, minimum height=0.6cm] (node-a4) at (1.8,0){$N_4$};
       \draw[-stealth] (node-a3) -- (node-a1); 
        \draw[-stealth] (node-a2) -- (node-a3);
        \draw[-stealth] (node-a1)  to [bend left] (node-a4);
        \draw[-stealth] (node-a4)  to [bend left] (node-a2);
        \node at (-1.8,0.4){$1$};
        \node at (-1.8,-0.4){$2$};
        \node at (0,-0.6){$3$};
        \node at (1.8,-0.6){$4$};
        \node at (0,-1.6){$(1)$};
        \node at (2.3,1){$\xrightarrow{\mu_3}$};
        \node[draw, circle, minimum size=0.6cm] (node-b3) at (5.2,0){$N_2'$};
        \node[draw, circle, minimum size=0.6cm] (node-b1) at (3.4,1){$N_1$};
        \node[draw, circle, minimum size=0.6cm] (node-b2) at (3.4,-1){$N_1$};
        \node[draw, minimum width=0.6cm, minimum height=0.6cm] (node-b4) at (7,0){$N_4$};
       \draw[-stealth] (node-b1) -- (node-b3); 
        \draw[-stealth] (node-b3) -- (node-b2);
        \draw[-stealth] (node-b2)--  (node-b1);
        \draw[-stealth] (node-b1)  to [bend left] (node-b4);
         \draw[-stealth] (node-b4)  to [bend left] (node-b2);
        \node at (3.6,0.4){$1$};
        \node at (3.6,-0.4){$2$};
        \node at (5.2,-0.6){$3$};
        \node at (5.2,-1.6){$(2)$};
        \node at (7.6,1){$\xrightarrow[]{\mu_2}$};
        \node[draw, circle, minimum size=0.6cm] (node-c3) at (10.4,0){$N_2'$};
        \node[draw, circle, minimum size=0.6cm] (node-c1) at (8.6,1){$N_1$};
        \node[draw, circle, minimum size=0.6cm] (node-c2) at (8.6,-1){$N_3$};
        \node[draw, minimum width=0.6cm, minimum height=0.6cm] (node-c4) at (12.2,0){$N_4$};
       \draw[-stealth] (node-c1) -- (node-c2); 
        \draw[-stealth] (node-c2) -- (node-c3);
        \draw[-stealth] (node-c2)  to [bend right] (node-c4);
        \node at (8.8,0.4){$1$};
        \node at (8.8,-0.4){$2$};
        \node at (10.4,-0.6){$3$};
        \node at (10.4,-1.6){$(3)$};
    \end{tikzpicture}
    \caption{$N_3'=N_4-N_3, N_2'=N_3-N_2$.
    Performing the quiver mutation $\mu_2$ to the quiver in Figure \ref{fig:mu3}, we can get the $(1)$. The quivers $(2),(3)$ are obtained by quiver mutations shown in 
    the Figure. Performing $\mu_1$ to the $(2)$ we get the  Figure \ref{fig:D4-6mutations} (b). Performing $\mu_1$ and $\mu_3$ to $(3)$ we get quivers in Figure \ref{fig:D4-6mutations} (d) and Figure \ref{fig:star1} respectively. In these quivers, superpotentials are sum of traces of all cycles. }
    \label{fig:extrafig}
\end{figure}
\noindent
\begin{figure}[H]
    \centering
    \begin{tikzpicture}   
        \node[draw, circle, minimum size=0.6cm] (node-b3) at (5.2,0){$N_1'$};
        \node[draw, circle, minimum size=0.6cm] (node-b1) at (3.4,1){$N_2$};
        \node[draw, circle, minimum size=0.6cm] (node-b2) at (3.4,-1){$N_2$};
        \node[draw, minimum width=0.6cm, minimum height=0.6cm] (node-b4) at (7,0){$N_4$};
       \draw[-stealth] (node-b3) -- node[midway, above] {$B_1$} (node-b1); 
        \draw[-stealth] (node-b2) --node[midway, above] {$B_2$} (node-b3);
        \draw[-stealth] (node-b1)--  node[midway, left] {$C$}(node-b2);
        \draw[-stealth] (node-b2)  to [bend right]  (node-b4);
         \draw[-stealth] (node-b4)  to [bend right]  node[midway, above] {$A_1$} (node-b1);
        \node at (3.6,0.4){$1$};
        \node at (3.6,-0.4){$2$};
        \node at (5.2,-0.6){$3$};
        \node at (5.6,-1.2){$A_2$};
        \node at (5.2,-1.7){$(4)$};
        \node at (8,1){$\xrightarrow[]{\mu_2}$};
        \node[draw, circle, minimum size=0.6cm] (node-c3) at (11,0){$N_1'$};
        \node[draw, circle, minimum size=0.6cm] (node-c1) at (9.2,1){$N_2$};
        \node[draw, circle, minimum size=0.6cm] (node-c2) at (9.2,-1){$N_3$};
        \node[draw, minimum width=0.6cm, minimum height=0.6cm] (node-c4) at (12.8,0){$N_4$};
       \draw[-stealth] (node-c2) --node[midway, left] {$C$} (node-c1); 
        \draw[-stealth] (node-c3) -- node[midway, above] {$B_2$}(node-c2);
        \draw[-stealth] (node-c4)  to [bend left] (node-c2);
        \node at (9.4,0.4){$1$};
        \node at (9.4,-0.4){$2$};
        \node at (11,-0.6){$3$};
        \node at (11.4,-1.2){$A_2$};
        \node at (11,-1.7){$(5)$};
    \end{tikzpicture}
    \caption{ In the diagram, $N_1'=N_3-N_1$. Figure $(4)$ is obtained by performing $\mu_3$ to the quiver in Figure \ref{fig:m1m3}. Performing $\mu_1$ we get the quivers $(e)$ in Figure \ref{fig:D4-6mutations}. Performing 
    Performing $\mu_3$ to the quiver $(5)$, we get the quiver $(a)$ in Figure \ref{fig:D4-6mutations} by relabeling nodes.
    For the quiver $(4)$, the superpotential is $W=tr(A_2A_1C)+tr(B_2B_1C)$. 
    }
    \label{fig:extrafig2}
\end{figure}
We say two quivers are the same if they are the same via permuting the order of nodes. One can check that performing quiver mutation to any quiver in $\Omega$, one gets a quiver in $\Omega$ up to a permutation of nodes. 

Notice that the Figure \ref{fig:extrafig} $(1)$ is similar with the Figure \ref{fig:m1m3} by switching the $N_1$ and $N_2$, so one can find the corresponding variety by mimicking the Example \ref{ex:m1m3}. 
The Figure \ref{fig:extrafig} $(2)$ and Figure \ref{fig:extrafig2} $(4)$ are similar, so we will only write down the quiver variety of \ref{fig:extrafig2} $(4)$ in detail. 
The Figure \ref{fig:extrafig} $(3)$ is similar to the Figure \ref{fig:D4-6mutations} $(e)$ and the Figure \ref{fig:extrafig2} $(5)$ is similar to Figure \ref{fig:D4-6mutations} $(b)$ by switching the $N_1$ and $N_2$. We denote the quiver in Figure \ref{fig:extrafig2} $(5)$ by $\mathbf Q_{11}$, and the corresponding input data of GIT quotient by $(V_{11}, G_{11},\theta_{11})$. 
\begin{ex}\label{ex:Q10}
    In this example, we will find the variety of the quiver with superpotential $\mathbf Q_{10}$ in Figure \ref{fig:extrafig2} $(4)$. 
    Denote the input data of the GIT quotient by $(V_{10}, G_{10},\theta_{10})$. 
    The critical locus $Z_{10}=\{dW_{10}=0\}$
    of the superpotential $W_{10}=tr(A_2A_1C)+tr(B_2B_1C)$ is defined by the following equations
    \begin{align}\label{dW10=0}
        &A_2A_1+B_2B_1=0, \nonumber\\
        &CA_2=0,\,A_1C=0,\, B_1C=0,\,CB_2=0.
    \end{align}
    Choose character $\theta_{10}$ with
\begin{align}\label{eqn:phaseforQ10}
   \sigma_2>0,\,
    \sigma_3>0,\, 
    \sigma_1+\sigma_3<0.
    \end{align}
 One can find the above phase satisfies the relation in Conjecture \ref{conj:phase} with phase of character $\theta_2$ in
\eqref{eqn:Q2phase}.
    Then in this phase, the semistable locus is 
    \begin{align}
        Z_{10,\theta_{10}}^{ss}(G_{10})=\{C=0,A_2A_1+B_2B_1=0\big|\, B_1,A_1, A_2\text{ non-degenerate }\}.
    \end{align}
\noindent
    See Lemma \ref{lem:Q10semapp} for a proof of this semistable locus. 
    Consider another quiver obtained by deleting the arrow $1\rightarrow 2$ in Figure \ref{fig:extrafig2} (4). 
    We denote this new quiver by $\tilde{\mathbf Q}_{10}$ and the corresponding input data for GIT quotient by $(\tilde V_{10}, G_{10},\theta_{10})$ 
    where $\theta_{10}$ has the same phase with \eqref{eqn:phaseforQ10}.
    Denote the quiver variety by $\mc X_{10}$. 
    Then the critical locus $\mathcal Z_{10}$ can be viewed as a subvariety of $\mc X_{10}$ defined by $\mc Z_{10}=\{A_2A_1+B_2B_1=0\}\sslash_{\theta_{10}}G_{10}$.
\end{ex}
The phase of $\theta_{11}(g)=\prod_{i=1}^3\det(g_i)^{\sigma_i
}$ is as follows,
\begin{align}
    \sigma_2<0,\,
    \sigma_2+\sigma_3>0,\,
    \sigma_1+\sigma_2+\sigma_3<0,
\end{align}
according to the Conjecture \ref{conj:phase}. 
One can find the phase of $\theta_{11}$ is similar with that of $\theta_{5}$.
In this phase, the semistable locus is 
\begin{equation}
    V^{ss}_{11,\theta_{11}}=
    \{(A_2,B_2,C)\big|\, \begin{bmatrix}A_2\\ B_2 \end{bmatrix}, A_2C,\,B_2C, \,B_2,\,C \text{ nondegenerate}\}.
\end{equation}
\section{Gromov-Witten Invariants and Wall-Crossing Theorem}\label{Sec:GW}
We will introduce the GW theory and the wall-crossing theorem. Readers who are familiar with related materials can skip this section.
\subsection{Gromov-Witten invariants}
We refer to the beautiful book \cite{GW:mirrorsym} about the fundamental properties of GW theory. 
\begin{dfn}
Let $\mc X$ be a smooth projective variety.
A stable map to $\mc X$ denoted by $(C, p_1,\ldots,p_n;f)$ consists of the following data:
\begin{enumerate}[(a)]
    \item
    a nodal curve $(C,p_1,\ldots,p_n)$ with $n\geq 0$ distinct nonsingular markings, 
    \item a \textit{stable map} $f:(C,p_1,\ldots,p_n)\rightarrow \mc X$ such that every component of $C$ of genus 0, which is contracted by $f$, must have at least three special (marked or singular) points, and every component of $C$ of genus one which is contracted by $f$, must have at least one special point. 
    \end{enumerate}
\end{dfn}
The degree of a stable map $(C, p_1,\ldots, p_n; f)$ is defined as the homology class of the image $\beta=f_*[C]$.
For a fixed curve class $\beta\in H_2(\mc X, \Z)$, let $\overline{M}_{g,n}(\mc X, \beta)$ denote the stack of stable maps from $n$-marked and genus-g curves $C$ to $\mc X$ such that $f_*[C]=\beta$.
When $\mc X$ is projective, $\overline{M}_{g,n}(\mc X, \beta)$ is a proper separated DM stack and admits a perfect obstruction theory. Hence we can construct the virtual fundamental class $[\overline{M}_{g,n}(\mc X, \beta)]^{vir}\in A_{\text{vdim}}(\overline{M}_{g,n}(\mc X, \beta))$ where $\text{vdim}=\int_{\beta}c_1(X)+(\dim(\mc X)-3)(1-g)+n$. See  \cite{GW:LT,GW:BFintrinsic,GW:B}.

Let
$\pi:\mc C_{g,n}\rightarrow\overline{M}_{g,n}(\mc X,\beta)\,,$ be the universal curve 
and $s_i$ are sections of $\pi$ for each marking $p_i$. 
Let $\omega_\pi$ be the relative dualizing sheaf and $\mc P_i=s_i^*(\omega_\pi)$ be the cotangent bundle at the $i$-th marking.
Define the $\psi$-class by $\psi_i:=c_1(\mc P_i) \in H^2(\overline M_{g,n}(\mc X,\beta))$. 
Define evaluation maps by
\begin{align}
   ev_i:\overline M_{g,n}(\mc X,\beta)&\longmapsto \mc X\nonumber\\
   (C, p_1,\ldots,p_n;f)&\longmapsto f(p_i)\,.
\end{align}
Let $\gamma_1,\ldots,\gamma_n\in H^*(\mc X)$ be cohomology classes and $a_i$ $i=1,\ldots,n$ be positive integers. The GW invariant is defined as
\begin{equation}\label{eqn:GWinv}  \langle\tau_{a_1}\gamma_1,\ldots,\tau_{a_n}\gamma_n \rangle_{g,n,\beta}:=\int_{[\overline M_{g,n}(\mc X, \beta)]^{vir}}\prod_{i=1}^{n}\psi_i^{a_i}ev_i^*(\gamma_i)\,.
\end{equation}
Let $\alpha_0=1,\alpha_1,\ldots, \alpha_m\in H^*(\mc X)$ be a set of generators of cohomology group, and $\alpha^0,\alpha^1,\ldots, \alpha^m\in H^*(\mc X)$ be the Poincar\'e dual. 
The small $\mc J$-function of $\mc X$, which comprises genus-zero GW invariants, is defined by
\begin{equation}\label{eqn:GWJ}
    \mc J^{\mc X}(Q,\mathbf{t},u)=
  \sum_{i=0}^{m} \sum_{(k\geq 0,\beta )} \alpha^i \langle \frac{\alpha_i}{u(u-\psi_{\bullet})}\mathbf{t}\ldots \mathbf{t}\rangle_{0,k+1,\beta}\frac{Q^{\beta}}{k!} \,.
\end{equation}
where $\mathbf{t}\in H^{\leq 2}(\mc X)$.

When $\mc X$ admits a torus action, denoted by $S$, then $S$ induces an action on $\overline M_{g,n}(\mc X, \beta)$ by sending a stable map $(C, p_1,\ldots, p_n;\,f)$ to $(C, p_1,\ldots, p_n;\, s\circ f)$ for each $s\in S$. 
Let $F$ denote a torus fixed locus of $\overline M_{g,n}(\mc X, \beta)$. 
There is an induced equivariant perfect obstruction theory on $\overline M_{g,n}(\mc X, \beta)$, hence the equivariant virtual fundamental class. 
Let $H^*_S(\mc X):=H^*(\mc X\times_G EG)$ be equivariant cohomology group of $\mc X$. 
For $\omega_i\in H^*_S(\mc X)$, the equivariant GW invariants are defined via the virtual localization theorem as follows,
\begin{equation}     \langle\tau_{a_1}\omega_1,\ldots,\tau_{a_n}\omega_n \rangle_{g,n,\beta}^S:=\sum_{F}\int_{F}  \frac{i_F^*\left(\prod_{i=1}^{n}\psi_i^{a_i}ev_i^*(\omega_i)\right)}{e^S(N^{vir}_F)}\,.
\end{equation}
The summation is over all torus fixed locus $F$, the map $i_F: F\rightarrow \overline M_{g,n}(\mc X, \beta)$ is the embedding, and $N_F^{vir}$ is the virtual normal bundle of $F$.
Suppose $\mc X$ is projective and $\gamma_i's$ are the non-equivariant limit of $\omega_i's$ via the map $H^*_S(\mc X)\rightarrow H^*(\mc X)$, and then
the nonequivariant limit of $\langle\tau_{d_1}\omega_1,\ldots,\tau_{d_n}\omega_n \rangle_{g,n,\beta}^S$ is equal to the regular GW invariant $\langle\tau_{d_1}\gamma_1,\ldots,\tau_{d_n}\gamma_n \rangle_{g,n,\beta}$.
See \cite{GW:GPequiv}.

Similarly, we can define the equivariant small $\mc J$-function of $\mc X$ by changing each correlator in \eqref{eqn:GWJ} to the equivariant version. We denote the equivariant small $\mc J$ function by $\mc J^{\mc X, S}( Q,\mathbf{t},  u)$.
\subsection{Genus-zero wall-crossing theorem}
In this subsection, we introduce the genus-zero wall-crossing theorem in the context of Cheong, Ciocan-Fontanine, Kim, and Maulik  \cite{MR3126932,MR3412343,MR3272909,MR3586512}.  
We only involve necessary parts for our purpose. 

Fix a valid input data for a GIT quotient $(V, G, \theta)$, and denote the corresponding GIT quotient by $\mc X:=V\sslash_\theta G$.
\begin{dfn}\label{dfn:stablequasigraphmap}
A quasimap from $\mathbb P^1$ to $V\sslash_\theta G$ consists of the data $(P,\sigma)$ where
\begin{itemize}
\item $P$ is a principle $G$-bundle on $\mbb P^1$,
\item $\sigma$ is a section of the induced bundle $P\times_G V$ with the fiber $V$ on $\mbb P^1$.
\end{itemize}
The class of a quasimap is defined as $\beta\in \op{Hom}(\op{Pic}^G(V), \Z)$, such that for each line bundle $L\in \op{Pic}^G(V)$, 
\begin{equation}
    \beta(L)=\deg_{\mbb P^1}(\sigma^*(P\times_G L))\,.
\end{equation}
\end{dfn}
\begin{dfn}\label{dfn:effectiveclass}
An element
$\beta\in \op{Hom}(\op{Pic}^G(V),\Z)$
is called an $I$-effective class if it is the class of a quasimap from $\mbb P^1$ to $V\sslash_\theta G$. Denote the semigroup of $I$-effective classes by $\op{Eff}(V, G, \theta)$.
\end{dfn}
\begin{dfn}
A quasimap $(P, \sigma)$ from $\mbb P^1$ to $V\sslash_\theta G$ is stable if
\begin{enumerate}
    \item the set $B:=\sigma^{-1}(V^{us})\subset \mbb P^1$ is finite, and points in $B$ are called base points of the quasimap,
    \item $\mathbf L_\theta:=\sigma^*(P\times_GL_\theta)$ is ample, where $L_{\theta}=V\times \C_\theta$.
\end{enumerate}
\end{dfn}
Denote the moduli stack of all stable qusimaps from $\mathbb P^1$ to $V\sslash_\theta G$ of class $\beta$ as $QG_{\beta}(V\sslash G)$. This moduli stack is the so-called stable quasimap graph space in \cite{MR3126932}.
\begin{thm}[\cite{MR3126932}]\label{thm:quasimappot}
The stack $QG_{\beta}(V\sslash_\theta G)$ is a separated Deligne-Mumford stack of finite type, proper over the affine quotient $Spec(H^0(V, \mc O_V)^G)$. It admits a canonical perfect obstruction theory if $V$ has at most lci singularities.
\end{thm}
Let $[\zeta_0,\zeta_1]$ be homogeneous coordinates on $\mathbb P^1$, and it has a standard $\C^*$ action given by 
\begin{equation}
    t[\zeta_0,\zeta_1]=[t\zeta_0,\zeta_1], t\in \C^* \,.
\end{equation}
The $\C^*$-action on $\mbb P^1$ induces an action on $QG_{\beta}(V\sslash_\theta G)$. If a quasimap 
$(P,\sigma)\in QG_{\beta}(V\sslash_\theta G)$ is $\C^*$-fixed, then all base points and the entire degree $\beta$ must be supported over the torus fixed points $[0:1]$ or $[1:0]$.

Consider the $\C^*$-fixed locus $F_\beta$ where everything is supported over the point $[0:1]\in \mbb P^1$ and the map $ev_\bullet: \mathbb P^1\backslash \{[0:1]\}\rightarrow V\sslash_\theta G$ is constant. 

 \begin{dfn}\label{dfn:smallIfunction}
Define the quasimap small $I$-function of a projective GIT quotient $V\sslash_\theta G$ as
 \begin{equation}
     I^{V\sslash_\theta G}(q,u)=1+\sum_{\beta\neq 0}q^{\beta} I_{\beta}^{V\sslash_\theta G}(u) \,\,, I_{\beta}^{V\sslash_\theta G}(u)=(ev_{\bullet})_*\left( \frac{[F_{\beta}]^{vir}}{e^{\C^*}(N_{F_{\beta}}^{vir})}  \right)\,,
 \end{equation}
 where the sum is over all $I$-effective classes of $(V, G,\theta)$.
 \end{dfn}

Assume $V\sslash_\theta G$ is projective, and $V$ admits a torus action ${S}$ which commutes with the action of $G$ on $V$. Hence the $S$ acts on $V\sslash_\theta G$.
The torus action is good if the torus fixed locus  $(V\sslash_\theta G)^S$ is a finite set.
There is an induced action of $S$ on $QG_\beta(V\sslash_\theta G)$ by sending $(P,u)\in QG_\beta(V\sslash_\theta G)$ to $s\circ u$ for each $s\in S$. 
 Moreover, the perfect obstruction theory is canonical $S$-equivariant \cite{MR3126932}. 
The same formula defines the equivariant quasimap small $I$-function of $V\sslash_\theta G$ as Definition \ref{dfn:smallIfunction} with all characteristic classes and pushforwards replaced by the equivariant version. 
We denote the equivariant quasimap small $I$-function of $V\sslash_\theta G$ by $I^{V\sslash_\theta G, S}(q, z)$.
\begin{thm}[\cite{MR3272909}]\label{thm:wallcrossing}
Assume $V\sslash_\theta G$ is a (quasi-)projective variety with a good torus action, and $V$ admits at most lci singularities. Then the following (equivariant) wall-crossing formula holds when $(V,G,\theta)$ is semi-positive,
\begin{equation}
    \mc J^{V\sslash_\theta G, S}( q, \mbf{t},u)=\frac{I^{V\sslash_\theta G,S}(q, u)}{I_0(q)}\,,
\end{equation} via  mirror map
$\mbf{t}=\frac{I_1(q)}{I_0(q)} \in H^{\leq 2}(V\sslash G)$,
where the $I_0(q)$, $I_1(q)$ are defined as coefficients of $1$ and $u^{-1}$ in the following expansion,
\begin{equation}
    I^{V\sslash_\theta G,S}(q, u)=I_0(q)+\frac{I_1}{ u}(q)+O(\frac{1}{ u^2})\,.
\end{equation}
\end{thm}
One can check that all the quiver varieties and their subvarieties we consider in Section \ref{sec:quiver} satisfy the assumptions of wall-crossing theorem, so in the following sections, when we talk about the genus-zero Gromov-Witten theories of varieties we mean the quasimap small $I$-functions.

\section{Equivariant Quasimap Small $I$-Functions}\label{sec:equivIfunction}
\subsection{Abelian/nonabelian correspondence for  $I$-functions}
We will mainly follow the work of Rachel Webb about the abelian-nonabelian correspondence to display the quasimap small $I$-functions of our examples, see \cite{webb2021abelianization, abelianization:Webb}. 

Fix a valid input $(V, G, \theta)$ for a GIT quotient  $V\sslash_\theta G$, and assume that $V$ has at most lci singularities.
Let $T=(\C^*)^{r}$ be the maximal torus of $G$ and $W_T=N_T/T$ the Weyl group. 
We denote the semistable, stable and unstable locus of $V$ under the action of $T$ in character $\theta$ by $V^{ss}_\theta(T)$, $V^s_\theta(T)$, and $V^{us}_\theta(T)$.
Assume that $V^{ss}(T)=V^{s}(T)$ and $T$ acts freely on $V^{ss}(T)$, 
so that we obtain a smooth variety $V\sslash_\theta T:=V^{s}(T)/T$. 
Assume there is a torus S acting on V which commutes with the action of G and the actions of $S$ on $V\sslash_\theta G$ and $V\sslash_\theta T$ are both good. 

The relation between $H^*(V\sslash_\theta G)$ and $H^*(V\sslash_\theta T)$ is studied by \cite{Ab_nab:EG,Ab_nab:Mar,Kir}. 
The map $V\sslash_\theta G\dashrightarrow V\sslash_\theta T$ is realized as follows
\begin{equation}\label{diag}
\begin{tikzcd}
V^s(G)/T \arrow[hookrightarrow]{r}{a} \arrow[]{d}{p}
  & V^s(T)/T \\
V^s(G)/G & 
\end{tikzcd}
\end{equation}
The Weyl group $W_T$ acts on $V^s(G)\slash T$, and therefore on $H^*(V^s(G)\slash T)$. 
 The above diagram induces the following classical identification for the cohomology groups
\begin{equation}\label{eqn:cohabelianization}
    H_S^*(V\sslash_\theta G, \Q)\iso H_S^*(V^s(G)\slash T,\Q)^W\,.
\end{equation}
See \cite[Proposition 2.4.1]{abelianization:Webb} for a proof of the above isomorphism for a chow group version.
For each $\gamma \in H_S^*(V\sslash_\theta G,\Q)$, we call $\tilde \gamma \in H_S^*(V\sslash_\theta T, \Q)^W$ a lifting of $\gamma$ if $a^*(\tilde\gamma)=p^*(\gamma)$. 
For each $\eta\in \chi(G)\subset \chi(T)$, there are line bundles $V\times \C_\eta\in \op{Pic}^G(V)$ and $V\times \C_\eta\in \op{Pic}^T(V)$.
Also, there is a natural map from $\op{Pic}^G(V)$ to $\op{Pic}^T(V)$ by restriction.
Therefore we have the following commutative diagram 
\begin{equation}
  \begin{tikzpicture}
\node (VG) at (-1,1) {$\op{Pic}^G(V)$};
\node[right=of VG] (VT) {$\op{Pic}^T(V)$};
\node[below=of VT] (XT) {$\chi(T)$};
\node [below=of VG] (XG) {$\chi(G)$};
\draw[->] (VG)--(VT) node [] {};
\draw[->] (XT)--(VT) node [] {};
\draw[->] (XG)--(VG) node [] {};
\draw[->] (XG)--(XT) node [] {};
\end{tikzpicture}  
\end{equation}
Taking $\op{Hom}(-,\Z)$ to the above diagram, we get the following commutative diagram, 
\begin{equation}\label{diag:pic}
  \begin{tikzpicture}
\node (VT) at (-1,1) {$\op{Hom}(\op{Pic}^T(V),\Z)$};
\node[right=of VT] (VG) {$\op{Hom}(\op{Pic}^G(V),\Z)$};
\node[below=of VT] (XT) {$\op{Hom}(\chi(T),\Z)$};
\node [below=of VG] (XG) {$\op{Hom}(\chi(G),\Z)$};
\draw[->] (VT)--(VG) node [midway,above] {$r_1$};
\draw[->] (VT)--(XT) node [midway, right] {$v_1$};
\draw[->] (VG)--(XG) node [midway, right] {$v_2$};
\draw[->] (XT)--(XG) node [midway,above] {$r_2$};
\end{tikzpicture}  
\end{equation}
For any $\xi\in \chi(T)$, denote by $\mc L_\xi:=V^s(T)\times_T\C_\xi$ the line bundle over $V\sslash_\theta T$. For any $\tilde\beta\in \op{Hom}(\op{Pic}^T(V),\Z)$, denote by $\tilde\beta(\xi):=\tilde\beta(c_1(\mc L_\xi))$, and it also equals $v_1(\tilde \beta)(\xi)$ by the above diagram. 
 
\begin{lem}\label{lem:r1}(\cite{MR3126932})
When $r_1$ restricts to $I$-effective classes $\op{Eff}(V, T,\theta)\subseteq \op{Hom}(\op{Pic}^T(V),\Z)$ in the source and $\op{Eff}(V, G, \theta)\subseteq \op{Hom}(\op{Pic}^G(V),\Z)$ in the target, it has finite fibers. 
\end{lem}
 
 \begin{thm}[\cite{abelianization:Webb}]\label{thm:abelian-nonabelianequiv}
The equivariant quasimap small $I$-functions of $V\sslash_\theta G$ and $V\sslash_\theta T$ satisfy 
 \begin{equation}\label{eqn:abeliannonabelianforI}
     p^*I_{\beta}^{V\sslash_\theta G,S}(u)=\left[ \sum_{\tilde\beta\rightarrow \beta} \prod_{\rho}\frac{\prod_{k\leq \tilde\beta(\rho)}(c_1(\mathcal{L}_{\rho})+ku)}{\prod_{k\leq 0}(c_1(\mathcal{L}_{\rho})+ku)}a^*I_{\tilde \beta}^{V\sslash_\theta T,S}(u)\right]
 \end{equation}
 where the sum is over all preimages $\tilde\beta$ of $\beta$ under the map $r_1$ in above diagram \eqref{diag:pic} and the product is over all roots $\rho$ of $G$. 
 \end{thm}
 Since the map $p$ is surjective, $p^*$ is injective,  then $I^{V\sslash_\theta G}$ is uniquely determined by $p^*I^{V\sslash_\theta G}$. 
In the following, we will make no difference between $I_\beta^{V\sslash_\theta G, S}$ and $p^*I_\beta^{V\sslash_\theta G,S}$.
  
Consider a $G$-equivariant bundle $E$ over $V$, and assume $s$ is a $G$-equivariant regular section of the bundle $E\times V\rightarrow V$. Let $Z:=Z(s)\subseteq V$ be the zero loci of $s$.
Taking $Z$ into consideration, we can extend the diagram \eqref{diag} to
 \begin{equation}
\begin{tikzcd}
Z^{s}_{\theta}(G)\slash T \arrow[hookrightarrow]{r}{b} \arrow[]{d}{\phi}&V^s_{\theta}(G)/T \arrow[hookrightarrow]{r}{a} \arrow[]{d}{p} & V\sslash_\theta T \\
Z\sslash_\theta G \arrow[hookrightarrow]{r}{\psi}  & V\sslash_\theta G&
\end{tikzcd}
\end{equation}
 and extend the diagram \eqref{diag:pic}  to
 \begin{equation}
\begin{tikzcd}
\op{Hom}(\op{Pic}^T(Z),\Q) \arrow[hookrightarrow]{r}{b_*} \arrow[]{d}{}&\op{Hom}(\op{Pic}^T(V),\Q)  \arrow[]{d}{r_1}  \\
\op{Hom}(\op{Pic}^G(Z),\Q) \arrow[hookrightarrow]{r}{\psi_*}  & \op{Hom}(\op{Pic}^G(V),\Q)
\end{tikzcd}
\end{equation}
For each $\xi \in \chi(T)$, and $\beta\in \op{Home}(\op{Pic}^T(V),\Z)$,
denote
\begin{equation}
    C(\beta,\xi):=\frac{\prod_{k\leq 0}(c_1(\mathcal{L}_\xi)+ku)}{\prod_{k\leq \beta(\xi)}(c_1(\mathcal{L}_\xi)+ku)}\,.
\end{equation}
The equivariant quasimap small $I$-functions of $Z\sslash_\theta G$ and $V\sslash_\theta T$ satisfy the following relation, which can be viewed as an abelian/nonabelian quantum Lefschetz theorem. 
\begin{thm}[\cite{webb2021abelianization, abelianization:Webb}]\label{thm:abeliannonabelianlefchetz}
 Assume that weights of $E$ with respect to the action of $T$ are $\epsilon_j$, for $j=1,\ldots,m$, and $\rho_i$ for $i=1,\ldots,r$ are roots of $G$. Then for a fixed $\delta\in \op{Hom}(\op{Pic}^G(V),\Q) $,  we have the following relation between $I$-functions of $Z\sslash_\theta G$ and $V\sslash_\theta T$,
 \begin{equation}
    \sum_{\beta\rightarrow \delta} \phi^*I_{\beta}^{Z\sslash_\theta G, S}(u)=\sum_{\tilde\delta\rightarrow \delta}
    \left(\prod_{i=1}^m C(\tilde \delta, \epsilon_i)^{-1} \right)\left(\prod_{i=1}^r  C(\tilde \delta, \rho_i)^{-1}  \right) b^*a^* I_{\tilde\delta}^{V\sslash T, S}(u)
 \end{equation}
  where $\tilde\delta\in \op{Hom}(\op{Pic}^T(V),\Q)$ are preimages of $\delta$ via $r_1$, and $\beta\in \op{Hom}(\op{Pic}^G(Z),\Q)$.
\end{thm}

\subsection{Quasimap small $I$-functions of our examples}\label{quasimap}
We will apply the abelian/nonabelian correspondence for $I$-functions to find the equivariant quasimap small $I$-functions of the varieties displayed in Section \ref{sec:quiver}.

\textbf{Conventions and Notations}
\begin{enumerate}[(1)]
    \item Denote $[N]:=\{1,\ldots, N\}$, and $\vec C_{[M]}:=\{f_1<\ldots<f_M\}\subset [N]$ a subset of $M$ integers in $[N]$.
    \item Fix a decorated quiver with superpotential $\mathbf Q=(Q_f\subset Q_0, Q_1, W)$ and an integer vector $\vec v=(N_i)_{i\in Q_0}$. Let $T=\prod_{i\in Q_0\backslash Q_f}(\C^*)^{N_i}$ be the maximal torus in the gauge group $G$.
    Consider a line bundle $V\times \C$ over $V$, and $t=(t_{I}^i)_{i\in Q_0\backslash Q_f,\,I\in [N_i]}\in T$ acts on it by $t\cdot((A_{i \rightarrow j}),v)=(t(A_{i\rightarrow j}),t^i_Iv)$. This action defines a line bundle $L_{I}^i:=V^{ss}(T)\times_{T}\C$. 
    Denote by $x^i_I:=c_1(L_I^i)\in H^*(V\sslash_\theta T)$ the first Chern class of such bundle for each $i\in Q_0\backslash Q_f, I\in [N_i]$.
    \item Let $S=(\C^*)^{N_8+N_9}$ and $R=(\C^*)^{N_4}$. 
    Denote the equivariant cohomology ring of a point under a trivial action of $S$ by $H^*_{S}(pt,\Q)=Q[\lambda_1,\ldots,\lambda_{N_8},\lambda_{N_8+1},\ldots,\lambda_{N_8+N_9}]$ and that of $R$ by $H^*_{R}(pt,\Q)=Q[\lambda_1,\ldots,\lambda_{N_4}]$.
    \item For each variety, we use the same notation $\vec q=(q_i)_{i\in Q_0\backslash Q_f}$ to denote the K\"ahler variables except when we need to consider transformations of K\"ahler variables under quiver mutations.
    \item Denote by $\op{Eff}^s:=\op{Eff}(V_s,G_s,\theta_s)$ and $\op{Eff}^{ms}:=\op{Eff}(\tilde V_s,\tilde G_s,\tilde{\theta}_s)$ the semigroup of $I$-effective classes of the star-shaped quiver variety and the variety $\tilde{\mc X}_s$ in Example \ref{ex:starmc}. Denote by $\op{Eff}^{s}_T$ and $\op{Eff}^{ms}_T$ their lifting to $\op{Hom}(\op{Pic}^T(V_s),\Q)$ and $\op{Hom}(\op{Pic}^T(\overline V_s),\Q)$. 
    Denote by $\op{Eff}^i$ the semigroups of $I$-effective classes of $\mc X_i$, and by $\op{Eff}^i_T$ their lifting via $r_1$.
\end{enumerate}
    For a general quiver $(Q_f\subset Q_0, Q_1, W)$ with integer vector $\vec v=(N_i)$, let $(V, G, \theta)$ be the input data of the quiver variety $V\sslash_\theta G$.
    A stable quasimap $(P,\sigma)$ from $\mathbb P^1$ to $\mc X=V\sslash_\theta G$ is equivalent to the following ingredients:
\begin{enumerate}[(a)]
    \item a vector bundle of matrices $P$ which can be written as  $\oplus_{i\rightarrow j\in Q_1}\oplus_{I=1}^{N_i}\oplus_{J=1}^{N_j}\mc O(n^{i}_I-n^{j}_J)$,
    \item a section $\sigma$ of the above bundle which maps all but finite points of $\mbb P^1$ to semi-stable locus.
\end{enumerate}
By our examples in Section \ref{sec:quiver}, 
a semistable point $(A_{i\rightarrow j})_{i\rightarrow j\in Q_1}$ in $V$ is described by the non-degeneracy of some matrices. 
If a matrix $A_{i\rightarrow j}$ is non-degenerate in $V_\theta^{ss}(G)$, the corresponding vectors $\vec n^i=(n^{i}_I)_{I=1}^{N_i}, \vec n^j=(n^{j}_J)_{J=1}^{N_j}$ satisfy the following conditions:
\begin{equation}\label{eqn:ruleofeffclass}
\begin{cases}
    \exists \text{ distinct }\{J_I\}_{I=1}^{N_i}\subset [N_j], \,\mathrm{s.t.} n^i_I-n^j_{J_I}\geq 0\,\,& \text{ if }
    N_j\geq N_i,\\
    \exists \text{ distinct }\{I_J\}_{J=1}^{N_j}\subset [N_i], \,\mathrm{s.t.} n^i_{I_J}-n^j_{J}\geq 0\,\,& \text{ if } N_j\leq N_i.
\end{cases}
\end{equation}
Those vectors $(n^i_I)_{i\in Q_0\backslash Q_f, I=1,\ldots,N_i}$ actually are the preimages of of $\mathrm{Eff}(V, G,\theta)$ in the diagram \eqref{diag:pic} under $r_1$ and Lemma \ref{lem:r1}, denoted by $\op{Eff}_T$, which explicitly are $\op{Eff}^{s}_T$, $\op{Eff}^{ms}_T$ and $\op{Eff}^{i}_T$, $i=1,\ldots,11$ for our examples. The map $r_1$ sends $(n^i_I)_{i,I}$ to $(\abs{\vec n^{i}})_{i\in Q_0\backslash Q_f}$ where $\abs{\vec n^{i}}=\sum_{I=1}^{N_i}n_{I}^i$.

\begin{lem}\label{lem:Iofquiver}
For a quiver variety $\mc X=V\sslash_\theta G$ with an $S$ action which comes from frame nodes, its equivariant quasimap small $I$-function is
\begin{align}
   p^*I^{\mc X, S}(\vec q,u)&=
    \sum_{(\vec n^i)\in \op{Eff}_T}\prod_{i\in Q_0\backslash Q_f}\prod_{\substack{I,J=1\\I\neq J}}^{N_i}
   a^*\Big( \frac{\prod_{l\leq n^i_I-n^i_J}(x^i_I-x^i_J+lu)}{\prod_{l\leq 0}(x^i_I-x^i_J+lu)}\nonumber\\
    &
    \prod_{i\rightarrow j\in Q_1}\prod_{I=1}^{N_i}\prod_{J=1}^{N_j}\frac{\prod_{l\leq 0}(x^i_I-x^j_J+lu)}{\prod_{l\leq n^i_I-n^j_J}(x^i_I-x^j_J+lu)}\Big)\prod_{i\in Q_0\backslash Q_f}q_i^{\abs{\vec n^i}}\,.
\end{align}
In the above formula, when one node $i$ is in $Q_f$, we let $n^i_I=0$. For quivers $\mathbf Q_i$ that are mutation equivalent to $D_3$, $x^4_I=\lambda_I$. For the star-shaped quivers and its quiver mutation,  $x^8_I=\lambda_I$, $I=1,\ldots,N_8$ and $x^9_J=\lambda_{N_8+J}$, $J=1,\ldots,N_9$.
\end{lem}
For convenience, denote the degree $\beta=(\vec n^i)$ term of $p^*I^{\mc X, S}$ by $I^{\mc X, S}_{\beta}$.
Suppose $\mc Z:=Z\sslash_\theta G$ is a subvariety in a quiver variety $\mc X=V\sslash_\theta G$, such that $Z$ is the zero loci of a regular section of a bundle $E$ over $V$. Suppose the weights of $T$ action on $E$ are $({\epsilon_a})_{a=1}^{m}$. 
\begin{lem}
The equivariant quasimap small $I$-function of $\mc Z$ can be written as follows by Theorem \ref{thm:abeliannonabelianlefchetz}
\begin{align}
   p^*I^{\mc Z, S}(\vec q,u)&=
    \sum_{\beta\in \op{Eff}_T}I^{\mc X, S}_{\beta}\prod_{a=1}^m\frac{\prod_{l\leq \beta(c_1(L_{\epsilon_a}))}(c_{1}(L_{\epsilon_a})+lu)}{\prod_{l\leq 0}(c_{1}(L_{\epsilon_a})+lu)}\prod_{i\in Q_0\backslash Q_f}q_i^{\abs{\vec n^i}}
\end{align}
\end{lem}
Thus we can obtain the quasimap small $I$-functions of all varieties in our story.

\section{Proofs for the Theorem \ref{thm:1st} and Theorem \ref{thm:2nd}}\label{sec:proof}
We first spell out our strategy to prove the equivalence of two quasimap small $I$-functions $I^{\mc Z}$ and $I^{\tilde{\mc Z}}$ of two varieties $\mc Z$ and $\tilde{\mc Z}$ related by a quiver mutation. 
In all examples we discuss, there is a common torus action $S$ on $\mc Z$ and $\tilde{\mc Z}$ such that the torus fixed loci $\mc Z^S$ and $\tilde{\mc Z}^S$ are discrete and finite with the same cardinality. Hence by the localization theorem \cite{ATIYAH19841},
We have $$ H^*_S(\mc Z)\iso H^*_S(\mc Z')\iso\oplus_{P\in \mc Z^S}H^*_S(P).$$
Let $\iota:\mc Z^S\rightarrow \tilde{\mc Z}^S$ be a natural bijection. 
Then in order to prove the relations in Theorem \ref{thm:1st} and Theorem \ref{thm:2nd} of quasimap small $I$-functions $I^{\mc Z}$ and $I^{\tilde{\mc Z}}$, we only have to prove the corresponding relations of restrictions of $I^{\mc Z}$ and $I^{\tilde{\mc Z}}$ to point $P\in \mc Z^S$ and $\iota(P)\in (\mc Z')^S$ for each $P\in \mc Z^S$.

In this section, we will find the torus fixed points of all varieties in Section \ref{sec:torusfixedloci}, recall the fundamental building block in Section \ref{sec:fundamentalblock}, and prove the main Theorems in Section \ref{sec:prooftheorem1} and Section \ref{sec:prooftheorem2}.
\subsection{Equivariant cohomology groups}\label{sec:torusfixedloci}
\subsubsection{Equivariant cohomology groups of general star-shaped quivers}\label{sec:torusfixedpointsstar}
 Denote by $\mathfrak F_s^b$ and $\mathfrak F_s^{a}$ the torus fixed loci of the star-shaped quiver $\mc X_s$ and $\tilde{\mc Z}_s$ under $S$ action. 
 Denote by $\mathfrak F^b_g$ and $\mathfrak F^a_g$ the $\tilde S=\prod_{i\in Q_f}(\mathbb C^*)^{N_i}$-fixed points in general star-shaped quiver $\mc X_g$ and $\tilde{\mc Z}_g$ discussed in Example \ref{ex:ggeneral} and \ref{ex:ggeneralstarmc}.

 \begin{lem}\label{lem:Fsab}
The $S$-fixed locus $\mathfrak F_s^b$ can be parameterized by the following set
\begin{align}\label{eqn:Fsb}
    \{ (\vec C_{[N_i]})_{i\in Q_0\backslash Q_f}\,\rvert\, \vec C_{[N_i]}\subset \vec C_{[N_j]} \text{ for } i\rightarrow j\in Q_1 \text{ and } i\neq 5;\,\vec C_{[N_5]}\subset \vec C_{[N_6]}\cup \vec C_{[N_7]} \},
\end{align}
and the $S$-fixed locus $\mathfrak F_s^a$ can be parameterized by the following set
\begin{align}\label{eqn:Fsa}
       \Big\{ &
       (\vec C_{[N_i]})_{i\in Q_0\backslash Q_f}\,\rvert\,
       \vec C_{[N_1]}\subset \vec C_{[N_3]}\subset \vec C_{[N_6]}\cup \vec C_{[N_7]};\,\,
       \vec C_{[N_2]}\subset \vec C_{[N_4]}\subset \vec C_{[N_6]}\cup \vec C_{[N_7]};\,\nonumber\\
       &\vec C_{[N_6]}\subset [N_8], \vec C_{[N_7]}\subset [N_9];\,
        \vec C_{[N_5']}\subset \vec C_{[N_6]}\cup \vec C_{[N_7]};\,
        \vec C_{[N_3]}\cap \vec C_{[N_5']}=
       \vec C_{[N_4]}\cap \vec C_{[N_5']}=\emptyset
      \Big \}. 
\end{align}
Furthermore, there is a canonical bijection 
\begin{align}
    \iota_s: \mathfrak F_s^b\rightarrow \mathfrak F_s^a\,,
\end{align}
such that for a general point $(\vec C_{[N_i]})\in \mathfrak F_s^b$,  $\iota_s$ keeps $\vec C_{[N_i]}$ for $i\neq 5$ and sends $\vec C_{[N_5]}$ to $(\vec C_{[N_6]}\cup \vec C_{[N_7]})\backslash \vec C_{[N_5]}$.
\end{lem}
\begin{proof}
According to the discussion in Example \ref{ex:generalstar}, each point  $(A_{i}) \in \mc X_s$ is a set of non-degenerate matrices. Such a point is $S$-fixed if and only if it has a representative such that $A_i$ for $i\neq 5,6$ and $\begin{bmatrix}A_5 &A_6\end{bmatrix}$ are all in reduced row echelon forms and these matrices have all entries except for pivots zero. Those matrices are totally determined by the column numbers of pivots. Therefore, we use the column numbers  of pivots $\vec C_{[N_i]}$ to represent these matrices.  

For a row reduced echelon form $A_{i\rightarrow j}$ with $i\rightarrow j\in Q_1$, we can relabel the columns by $\vec C_{N_j}$, and then use the numbers of columns its pivots lie in to represent $A_{i\rightarrow j}$. Hence, we have the inclusion relations among the sets in \ref{eqn:Fsa}. 

As to the points in $\mathfrak F_s^a$, everything is the same except that the augmented matrix $\begin{bmatrix}A_5\\ A_6\end{bmatrix}$ is column full-rank. 
Hence we consider its column reduced echelon form and use the set $\vec C_{[N_5']}$ to represent the rows its pivots lie in when we relabel the rows by integers $\vec C_{[N_6]}\cup \vec C_{[N_7]}$.  
The map $\iota_s$ is naturally bijective. 
\end{proof}

The above result can be generalized to a general star-shaped quiver $\mc X_g$ and its quiver mutation $\tilde{\mc Z}_g$, whose proof will be omitted.
\begin{lem}\label{lem:torusggeneralstar}
    The torus fixed locus $\mathfrak F_g^b$ can be described as follows
\begin{align}
     \{(\vec C_{[N_i]})_{i\in Q_0}\big|\,\vec C_{[N_i]}\subset \vec C_{[N_j]},\,\text{when }i\rightarrow j\in Q_1\text{ and } i\neq k\,,\,
    \vec C_{[N_k]}\subset \cup_{b=1}^l\vec C_{[N_{i_b}]}\}\,.
\end{align}
The torus fixed locus $\mathfrak F^a_g$ can be described as follows,
\begin{align}
    &\big\{(\vec C_{[N_i]})_{i\in Q_0}\big| \vec C_{[N_i]}\subset \vec C_{[N_j]},\,\text{when }i\rightarrow j\in Q_1\text{ and } i\neq k,j_1\ldots j_h\,,\nonumber\\
    &\vec C_{[N_{j_a}]}\subset \cup_{b=1}^l\vec C_{[N_{i_b}]},\forall a=1,\ldots,h\,,\,
    \vec C_{[N_{k}']}\subset \cup_{b=1}^l\vec C_{[N_{i_b}]}\,,\, \vec C_{[N_{j_a}]}\cap \vec C_{[N_{k}']}=\emptyset\,\big\}.
\end{align}
There is a bijection 
\begin{align}\label{eqn:iotagenstar}
    \iota:\mathfrak F^b_g\rightarrow \mathfrak F_g^a,
\end{align}
which preserves $\vec C_{[N_i]},i\neq k$ and sends $\vec C_{[N_k]}$ to $\cup_{b=1}^l\vec C_{[N_{i_b}]}\backslash \vec C_{[N_k]}$.
\end{lem}
Then we can easily prove that equivariant cohomology groups of star-shaped quivers are preserved by quiver mutations according to localization theorem \cite{ATIYAH19841},
\begin{equation}
    H_S^*(\mc X_s,\mathbb Q)\cong H^*_S(\tilde {\mc Z_s},\mathbb Q)\,,
\end{equation}
and 
\begin{equation}
    H_{\tilde S}^*(\mc X_g,\mathbb Q)\cong H^*_{\tilde S}(\tilde {\mc Z}_g,\mathbb Q)\,.
\end{equation}
\subsubsection{Equivariant cohomology groups of $D_3$ mutation equivalent quivers}\label{sec:torusfixedpointsD3}
Similarly as the previous subsection, the torus $R:=(\C^*)^{N_4}$ acts on all varieties that are mutation equivalent to $D_3$ and fixes finitely many points.
Denote by $\mathfrak{F}_i$ the torus fixed locus for the $i$-th variety $\mc X_i$ when there is no potential and those of $\mc Z_i$ when there is a potential function.

 Similar to the Lemma \ref{lem:Fsab} and Lemma \ref{lem:torusggeneralstar}, one can check that $R$-fixed loci $\mathfrak{F}_0$ and $\mathfrak F_1$ can be parameterized as follows,
\begin{equation}
    \mathfrak{F}_0=\big\{(\vec C_{[N_1]},\vec C_{[N_2]},\vec C_{[N_3]})\,\rvert \,\vec C_{[N_1]}, \vec C_{[N_2]}\subset\vec C_{[N_3]}\subset [N_4]\big\}\,
\end{equation}
and 
\begin{equation}
    \mathfrak{F}_1:=\big\{(\vec C_{[N_1]},\vec C_{[N_2]},\vec C_{[N_3']})\,\rvert\,\vec C_{[N_3']}\subset [N_4],\,\vec C_{[N_1]},\,\vec C_{[N_2]}\subset [N_4]\backslash \vec C_{[N_3']} \big\}\,,
\end{equation}
Their cardinalities are the same,
$\abs{\mathfrak{F}_1}=\abs{\mathfrak{F}_0}=C_{N_4}^{N_3}C_{N_3}^{N_1}C_{N_3}^{N_2}$.

\begin{lem}\label{lem:torusfixedF2}
The $R$-fixed locus $\mathfrak F_2$ can be expressed as 
\begin{align}
    \{\vec A_{[N_2]},
    \vec B_{[N_2]},
    \vec C_{[N_3']}|\, \vec A_{[N_2]}, \vec B_{[N_2]}\subset [N_4]; \vec C_{N_3'}\subset \vec A_{[N_2]};
    \vec C_{[N_3']}\cap \vec B_{[N_2]}=\emptyset
    \}\,.
\end{align}
There are in total $\abs{\mathfrak{F}_2}=C_{N_2}^{N_3'}C_{N_4}^{N_2}C_{N_3}^{N_2}$ torus fixed points which is equal to $\abs{\mathfrak F_1}$. Furthermore, there is a bijection
\begin{align}
    \iota_1: \mathfrak{F}_1\longrightarrow \mathfrak{F}_2\,,
\end{align}
via the map
\begin{align}
   (\vec A_{[N_1]},\vec B_{[N_2]},\vec C_{[N_3']})\longrightarrow 
    ([N_4]\backslash\vec A_{[N_1]},\vec B_{[N_2]},\vec C_{[N_3']})\,.
\end{align}
\end{lem}
\begin{proof}
As explained in Lemma \ref{lem:Fsab}, matrices $B_1$  $B_2$ $A_1$ can be represented by  $\vec A_{[N_2]}$ $\vec B_{[N_2]}$ and $\vec C_{[N_3']}$. 
The condition $\vec C_{[N_3']}\cap \vec B_{[N_2]}=\emptyset$ is equivalent to saying that $B_2B_1A_1=0$. 

The map $\iota_{1}$ sends an element $(\vec C_{[N_1]},\vec C_{[N_2]},\vec C_{[N_3']})\in \mathfrak{F}_1$ to $\mathfrak F_2$ because $\vec C_{[N_1]}\cap \vec C_{[N_3']}=\emptyset$ and $\vec C_{[N_2]}\cap \vec C_{[N_3']}=\emptyset$. 
\end{proof}
\begin{lem}
    The torus fixed points of $\mc Z_3$ are 
\begin{align}
   \mathfrak F_3 =\{(\vec A_{[N_2]},\vec B_{[N_1]},\vec C_{[N_4-N_3]})\big| \,\vec C_{[N_4-N_3]}\subset \vec A_{[N_2]}\cap\vec B_{[N_1]},\, \vec A_{[N_2]},\vec B_{[N_1]} \subset [N_4] \}
\end{align}
 and $\abs{\mathfrak F_3}=C_{N_4}^{N_3'}C_{N_3}^{N_3-N_1}C_{N_3}^{N_3-N_2}$.
There is a bijection 
\begin{align}
   \iota_2: \mathfrak F_2\rightarrow \mathfrak F_3
\end{align} 
sending $(\vec A_{[N_2]},\vec B_{[N_2]},\vec C_{[N_3']})$ to $(\vec A_{[N_2]}, [N_4]\backslash \vec B_{[N_2]}, \vec C_{[N_3']})$.
\end{lem}
\begin{proof}
    The sets $\vec A_{[N_2]}:=
    \{i_1,\ldots,i_{N_2}\},\vec B_{[N_1]}=\{j_1,\ldots,j_{N_1}\}$ are row numbers of pivots of column-reduced-echelon forms of matrices $B_1,B_2$. 
    Due to the relation $B_1A_1+B_2A_2=0$, columns of matrix $A_1,A_2$ are $\pm \vec e_i$ where $\vec e_i$ is a column vector with $i$-th component 1 and others zero. Furthermore, if one column of $A_1$ is $\vec e_a$ and then the same column  of $A_2$ is $-\vec e_b$ such that $i_a=j_b$. The set $\vec C_{[N_3']}$ is the set of distinct integers $\{i_a\}$ of number $N_3'$ such that there is a $j_b\in \vec B_{[N_2]}$ with $j_b=i_a$.
    
    One can check the bijection easily, so it is omitted. 
\end{proof}

The torus fixed locus $\mathfrak F_9$ is exactly the same with $\mathfrak F_0$.
The torus fixed loci $\mathfrak F_i$, $i=4,5,6, 7,10,11$ are complicated, and they are given in Appendix \ref{sec:Appsemistablelocus}. 

One can check that $\abs{\mathfrak F_i}$ are equal for all $i$. Hence, we know that the equivariant cohomology groups of all $\mc X_i$ (when there is no potential) or $\mc Z_i$ (when there is a potential function) are isomorphic. 
\subsection{Review for a fundamental building block}\label{sec:fundamentalblock} 
We refer to \cite{benini2015cluster,donghai,zhang2021gromov} for the detailed discussion of the fundamental building block. In this subsection, we only display statements we need. 

From now on, we will let the equivariant parameter be $u=1$ in the $I$-functions and denote $I^{\mc X, S}(\vec q):=I^{\mc X, S}(\vec q, 1)$, $I_\beta^{\mc X, S}(\vec q):=I^{\mc X, S}_\beta(\vec q, 1)$.

The \textit{fundamental building block} is about the following Figure \ref{fig:fundamental} containing two mutation-related quivers.
\noindent
\begin{figure}[H]
    \centering
    \begin{tikzpicture}   
         \node[draw, circle, minimum size=0.8cm] (node-0) at (0,0){$r$};
        \node[draw, minimum width=0.8cm, minimum height=0.8cm] (node-1) at (-2,0){$m$};
        \node[draw, minimum width=0.8cm, minimum height=0.8cm] (node-2) at (2,0){$n$};   
       \draw[-stealth] (node-1) -- (node-0); 
        \draw[-stealth] (node-0) -- (node-2);
        \node at (3.5,0){$\xrightarrow[]{\mu}$};
         \node[draw, circle, minimum size=0.8cm] (node-d0) at (7,0){$r'$};
        \node[draw, minimum width=0.8cm, minimum height=0.8cm] (node-d1) at (5,0){$m$};
        \node[draw, minimum width=0.8cm, minimum height=0.8cm] (node-d2) at (9,0){$n$};   
       \draw[-stealth] (node-d0) --node [midway, below]{$A_1$} (node-d1); 
        \draw[-stealth] (node-d2) -- node [midway, below]{$A_2$}(node-d0);
        \draw[-stealth] (node-d1) to [bend left] node [midway, above]{$B$} (node-d2);
        \end{tikzpicture}
    \caption{Assume $n\geq m$. $r'=m-r$. The potential of the right hand side quiver is $W=tr(BA_2A_1)$. }
    \label{fig:fundamental}
\end{figure}
\noindent
For the left hand side, choose character $\theta(g)=\det(g)^{\sigma}$ with $\sigma>0$, and then the quiver variety is the total space of $m$-copies of tautological bundle over Grassmannian $S^{\oplus m}\rightarrow Gr(r,n)$.  The character for the right hand side one will be $\theta(g)=\det(g)^{\tilde\sigma}$ with $\tilde\sigma<0$. The critical locus of superpotential is $\{BA_2=0\}\sslash_{\theta} G\subset \mathbb C^{m\times n}\times Gr(n-r,n)$ which can be viewed as the total space of $m$-copies of the dual of quotient bundles over the dual Grassmannian $(Q^\vee)^{\oplus m}\rightarrow Gr(n-r,n)$. 
We denote the two varieties by $Gr$ and $Gr^\vee$ respectively. 

There is a good torus action $S'=(\C^*)^{m}\times (\C^*)^{n}$ on the two varieties. 
Let $\mathfrak F$ be the torus fixed locus of $Gr$ and $\mathfrak F^\vee$ that of $Gr^\vee$. Adopting the notations and conventions we have used in the above section, one can check that 
\begin{equation}\label{eqn:torufixedGrass}
    \mathfrak F=\big\{\vec C_{[r]}=\{f_1<\ldots<f_r\} \subset [n]\big\},\,\,\,
         \mathfrak F^\vee=\big\{\vec C_{[n-r]}=\{f_1'<\ldots<f_{n-r}'\} \subset [n]\big\}\,.
\end{equation}
The canonical bijective map from $\mathfrak{F}$ to $\mathfrak F^\vee$ can be defined as
\begin{equation}\label{eqn:iotaGrass}
    \iota^{Gr}: (\vec C_{[r]}\subset [n])\rightarrow ([n]\backslash\vec C_{[r]}\subset [n])\,.
\end{equation}

Denote the equivariant parameters of $(\C^*)^{m}$-action by $\eta_A$,$A=1,\ldots,m$, and the equivariant parameters of $(\C^*)^{n}$-action by $\lambda_F$, $F=1,\ldots,n$. 
The equivariant quasimap small ${I}$-function of $Gr$, denoted by ${I}^{Gr,S'}(q)$, can be written as follows by Lemma \ref{lem:Iofquiver}
\begin{equation}
  p^*{I}^{Gr,S'}(q)=\sum_{\vec d\in \Z^r_{\geq 0}} 
   \prod_{I\neq J}^{r}\frac{\prod_{l\leq d_I-d_J}(x_I-x_J+l)}{\prod_{l\leq 0}(x_I-x_J+l)}  \prod_{I=1}^{r}\frac{\prod_{A=1}^{m}\prod_{l=0}^{d_I-1}(-x_I+\eta_A-l)}{ \prod_{F=1}^{n}\prod_{l=1}^{d_I}(x_I-\lambda_F+l)}q^{\abs{\vec d}}\,.
\end{equation}
The equivariant quasimap small $I$-function of $Gr^\vee$ denoted by $I^{Gr^\vee, S'}(q')$ is
\begin{equation}
    p^*{I}^{Gr^\vee,S'}(q')=\sum_{\vec d\in \Z^{n-r}_{\leq 0}}
    \prod_{I\neq J}^{n-r}\frac{\prod_{l\leq d_I-d_J}(x_I-x_J+l)}{\prod_{l\leq 0}(x_I-x_J+l)}
    \prod_{I=1}^{n-r}\frac{\prod_{A=1}^{m}\prod_{l=1}^{-d_I}(-x_I+\eta_A+l)}{ \prod_{F=1}^{n}\prod_{l=1}^{-d_I}(-x_I+\lambda_F+l)}(q')^{\abs{\vec d}}\,.
\end{equation}
For an arbitrary $S'$-fixed point $P=(\{f_1<\ldots<f_{r}\})  \in \mathfrak F$, denote the image $\iota^{Gr}(P)$ by $P^c=(\{f_1'<\ldots<f_{n-r}'\}=[n]\backslash \vec C_{[r]})\in \mathfrak F^\vee$. The restriction of $I^{Gr,S'}$ to $P$ is 
\begin{equation}\label{eqn:Ifungrbmres}
   p^*I^{Gr,S'}(q)\rvert_{P}=\sum_{\vec d\in \Z^r_{\geq 0}}
    \prod_{I\neq J}^{r}\frac{\prod_{l\leq d_I-d_J}(\lambda_{f_I}-\lambda_{f_J}+l)}{\prod_{l\leq 0}(\lambda_{f_I}-\lambda_{f_J}+l)}
\prod_{I=1}^{r}\frac{\prod_{A=1}^{m}\prod_{l=0}^{d_I-1}(-\lambda_{f_I}+\eta_A-l)}
    {\prod_{F=1}^{n}\prod_{l=1}^{d_I}(\lambda_{f_I}-\lambda_F+l)}q^{\abs{\vec d}}\,.
\end{equation}
The restriction of $I^{Gr^\vee,S'}$ to $P^c$ is
\begin{equation}\label{eqn:Ifungramres}
   p^*I^{Gr^\vee,S'}(q')\rvert_{P^c}=\sum_{\vec d\in \Z^{n-r}_{\leq  0}}
    \prod_{I\neq J}^{n-r}\frac{\prod_{l\leq d_I-d_J}(\lambda_{f_I'}-\lambda_{f_J'}+l)}{\prod_{l\leq 0}(\lambda_{f_I'}-\lambda_{f_J'}+l)}
    \prod_{I=1}^{n-r}\frac{\prod_{A=1}^{m}\prod_{l=1}^{-d_I}(-\lambda_{f_I'}+\eta_A+l)}{ \prod_{F=1}^{n}\prod_{l=1}^{-d_I}(-\lambda_{f_I'}+\lambda_F+l)}{q'}^{\abs{\vec d}}\,.
\end{equation}
\begin{thm}[\cite{benini2015cluster,donghai}] \label{thm:Haidong}
\begin{enumerate} 
    \item When $n\geq m+2$, $p^*{I}^{Gr,S'}(q)\rvert_{P}=p^*{I}^{Gr^\vee,S'}(q^{-1})\rvert_{P^c}$.
    \item When $n=m+1$, $p^*{I}^{Gr,S'}(q)\rvert_{P}={e}^{(-1)^{n-r-1}q}\, p^*{I}^{Gr^\vee,S'}(q^{-1})\rvert_{P^c}$.
    \item When $n=m$, 
    $p^*{I}^{Gr,S'}(q)\rvert_{P}=(1+(-1)^{n-r}q)^{\sum_{A=1}^{m}\eta_A-\sum_{F=1}^{n}\lambda_{F}+n-r}p^*{I}^{Gr^\vee,S'}(q^{-1})\rvert_{P^c}$.
\end{enumerate}
\end{thm}
\begin{proof}
    The proof of the second and third items can be found in Appendix A of physics work \cite{benini2015cluster} and the proof of the 1st item is given in Hai Dong's work which is unavailable online. We refer to the whole proof of the theorem  in \cite[Appendix]{zhang2021gromov}.
\end{proof}
\subsection{Proof for Theorem \ref{thm:1st}: the equivalence between  ${I^{\mc{X}_s, S}}$ and ${I^{\tilde{\mc Z}_s,S}}$}\label{sec:prooftheorem1}
In this section, we will prove the Theorem \ref{thm:1st}. Moreover, we will explain that it can be generalized to a general star-shaped quiver in Definition \ref{defn:starshapedquiver}.

We first utilize the localization theorem to investigate the relation between $I^{\mc X_s,S}(\vec q)\rvert_{P}$ and $I^{\tilde{\mc Z}_s,S}(\vec q')\rvert_{\iota_s(P)}$ for each pair of $S$-fixed points $(P, \iota(P))\in \mathfrak F_s^b\times \mathfrak F_s^a$.
The $I$-effective classes of $\mc X_s$ and $\tilde{\mc Z}_s$ are as follows. 
\begin{align}
    \mathrm{Eff}_T^{s}&=
    \Big\{(\vec n^i)_{i\in Q_0\backslash Q_f}\in \prod_{i=1}^{7} \mathbb Z^{N_i}_{\geq 0}
    \big|\, \forall\, i\rightarrow j\in Q_1, i\neq 5, \exists \text{ distinct } 
    \{J_I\}_{I=1}^{N_i}
    \subset [N_j], \mathrm{s.t.} n^i_{I}-n^j_{J_I}\geq 0;\nonumber\\
    &\exists \text{ distinct }\{k_{I}\}\subset [N_6]\sqcup [N_7], \mathrm{s.t.} n^5_I-n^6_{k_I}\geq 0 \text{ if } k_I\in [N_6],\text{ or }n^5_I-n^7_{k_I}\geq 0 \text{ if }k_I\in [N_7]
    \Big\}.\nonumber\\
    \mathrm{Eff}_T^{ms}&=
    \Big\{
    (\vec n^i)_{i\neq 5}\times \vec n^5\in \prod_{i\neq 5} \mathbb Z^{N_i}_{\geq 0}\times \mathbb Z^{N_5'}\big|
    \forall\, i\rightarrow j\in Q_1, i,j\neq 5, \exists \text{ distinct } 
    \{J_I\}_{I=1}^{N_i}
    \subset [N_j], \nonumber\\
    &\mathrm{s.t.} n^i_{I}-n^j_{J_I}\geq 0; 
    \exists \text{ distinct }\{k_{I}\}\subset [N_6]\sqcup [N_7], \mathrm{s.t.} -n^5_I+n^6_{k_I}\geq 0 \text{ if } k_I\in [N_6],\nonumber\\
    &
    \text{ or }-n^5_I+n^7_{k_I}\geq 0 \text{ if }k_I\in [N_7]
    \Big\}
\end{align}
In the above expression, we let $\vec n^i=0$ if $i\in Q_f$.
Without loss of generality, we choose a torus fixed point $P=(\vec C_{[N_i]})\in \mathfrak F_s^b$ described as follows.
\begin{itemize}
    \item  Let $\vec C_{[N_6]}=[N_6]$ and $\vec C_{[N_7]}=[N_7]$. We relabel the integers $[N_6]\cup [N_7]$ by $\{1,\ldots,N_6+N_7\}$, choose $\vec C_{[N_i]}=[N_i]\subset[N_6+N_7]$ for $i=1,\ldots,5$.
    \end{itemize}

\begin{prop}\label{prop:Istartransform}
For the pair of $S$-fixed points $(P,\iota_s(P))\in \mathfrak F_s^b\times \mathfrak F_s^a$, 
the restricted quasimap small $I$-functions $p^*I^{\mc X_s,S}\rvert_P$  and $p^*I^{\tilde{\mc X}_s,S}\rvert_{\iota(P)}$ satisfy the following relations.
\begin{enumerate}[(a)]
    \item When $N_6+N_7\geq N_3+N_4+2$, 
    \begin{equation}
       p^*I^{\mc X_s, S}(\vec q)\rvert_{P}=p^*I^{{\tilde{\mc Z}_s}, S}(\vec q')\rvert_{\iota_s(P)}\,,
    \end{equation}
    under the change of K\"ahler variables 
    \begin{align}\label{eqn:variablechange}
        q_5'=q_5^{-1}, \, q_6'=q_6q_5,\,q_7'=q_7q_5,\,q_i'=q_i, \text{ for } i\neq 5,6,7\,.
    \end{align}
     \item When $N_6+N_7= N_3+N_4+1$, 
    \begin{equation}
    p^*I^{\mc X_s, S}(\vec q)\rvert_{P}=e^{(-1)^{N_5'-1}q_5}p^*I^{{\tilde{\mc Z}_s}, S}(\vec q')\rvert_{\iota_s(P)}\,,
    \end{equation}
    and the map between K\"ahler variables is as \eqref{eqn:variablechange}.
    \item When $N_6+N_7= N_3+N_4$,
    \begin{equation}
    p^*I^{\mc X_s, S}(\vec q)\rvert_{P}=
     (1+(-1)^{N_5'}q_5)^
    {\sum_{i=3,4}\sum_{A=1}^{N_i}x^{i}_A-\sum_{j=6,7}\sum_{B=1}^{N_i}x^j_B+N_5'}
    p^* I^{{\tilde{\mc Z}_s}, S}(\vec q')\rvert_{\iota_s(P)}\,,
    \end{equation}
   under 
    \begin{align}
      &q_3'=q_3(1+(-1)^{N_5'}q_5)\,,
       q_4'=q_4(1+(-1)^{N_5'}q_5)\,,
       q_5'=q_5^{-1}\,,\nonumber\\
       &q_6'=\frac{q_6q_5}{(1+(-1)^{N_5'}q_5)}\,,
       q_7'=\frac{q_7q_5}{(1+(-1)^{N_5'}q_5)}\,, q_i'=q_i,\text{ for } i=1,2\,.
    \end{align}
\end{enumerate}
\end{prop}
\begin{proof}
By using Lemma \ref{lem:Iofquiver}, the quasimap small $I$-function of $\mc X_s$ can be written as follows, 
\begin{align}\label{eqn:Istar}
    I^{\mc X_s,S}(\vec q)=&\sum_{(\vec n^i)\in \op{Eff}_T^s}
    (\text{Irrel})\bigcdot 
    \prod_{I=1}^{N_5}\left(
    \prod_{i=3,4}\prod_{A=1}^{N_i}
    \frac{\prod_{l\leq 0}(x^i_A-x^5_I+l)}{\prod_{l\leq n^i_A-n^5_I}(x^i_A-x^5_I+l)}
    \right)\nonumber\\
    & \prod_{I\neq J}^{N_5}
    \frac{\prod_{l\leq n^5_I-n^5_J}(x^5_I-x^5_J+l)}{\prod_{l\leq 0}(x^5_I-x^5_J+l)}   \prod_{i=6,7}\prod_{B=1}^{N_i}\prod_{I=1}^{N_5}\frac{\prod_{l\leq 0}(x^5_I-x_B^i+l)}{\prod_{l\leq n^5_I-n_B^i}(x^5_I-x_B^i+l)}
    \prod_{i\in Q_0\backslash Q_f}q_i^{\abs{\vec n^i}}\,,
\end{align}
where the irrelevant factor $(\mathrm{Irrel})$ represents all factors containing no ingredients (Chern roots) of node $5$.

Restricted to the torus fixed point $P$, we have $x^6_B\rvert_P=\lambda_{B},\, x^7_B|_{P}=\lambda_{N_8+B}$. 
Relabeling the set $[N_6]\cup [N_7]$ by $[N_6+N_7]$, 
we rewrite the set $\{\lambda_{B}\}_{B=1}^{N_6}\cup \{\lambda_{B}\}_{B=1}^{N_7}$ as $\{\zeta_F \}_{F=1}^{N_6+N_7}$. 
Then restricted to $P$, $x^i_I\rvert_P=\zeta_I$, $i=1,2,3,4,5$, and $x^6_B\rvert_P=\zeta_B, x^7_B\rvert_P=\zeta_{N_6+B}$. 
Therefore, we have
\begin{subequations}
    \begin{align}
    p^*I^{\mc X_s,S}(\vec q)\rvert_{P}=&\sum_{(\vec n^i)\in \op{Eff}_T^s}
    (\text{Irrel})\bigcdot 
    \prod_{I=1}^{N_5}\left(
    \prod_{i=3,4}\prod_{A=1}^{N_i}
    \frac{\prod_{l\leq 0}(\zeta_A-\zeta_I+l)}{\prod_{l\leq n^i_A-n^5_I}(\zeta_A-\zeta_I+l)}
    \right)\label{eqn:ItoP0}\\
    & \prod_{I\neq J}^{N_5}
    \frac{\prod_{l\leq n^5_I-n^5_J}(\zeta_I-\zeta_J+l)}{\prod_{l\leq 0}(\zeta_I-\zeta_J+l)}   \prod_{F=1}^{N_6+N_7}\prod_{I=1}^{N_5}\frac{\prod_{l\leq 0}(\zeta_I-\zeta_F+l)}{\prod_{l\leq n^5_I-n_F}(\zeta_I-\zeta_F+l)}
    \prod_{i\in Q_0\backslash Q_f}q_i^{\abs{\vec n^i}}\,,\label{eqn:ItoP}
\end{align}
\end{subequations}
Notice that $m_I:=n^5_I-n_I\geq 0$. 
Replace $n_I^5=m_I+n_I$, 
and we can transform $p^*I^{\mc X_s, S}(\vec q)\rvert_{P}$ to the following formula by fixing $n^i_A$ and $n_I$ and disregarding the sum over $n^i$ for $i\neq 5$ and the irrelevant part,
\begin{align}\label{eqn:Iofstar}
     p^*I^{\mc X_s,S}(\vec q)\rvert_{P}=&\sum_{\vec m\in \Z^{N_5}_{\geq 0}} 
    \prod_{I=1}^{N_5}\left(
    \prod_{i=3,4}\prod_{A=1}^{N_i}
    \frac{\prod_{l\leq 0}(\zeta_A-\zeta_I+l)}{\prod_{l\leq n^i_A-n_I-m_I}(\zeta_A-\zeta_I+l)}
    \right)\prod_{i\in Q_0\backslash Q_f,i\neq 5}q_i^{\abs{\vec n^i}}
    \nonumber\\
    & \prod_{I\neq J}^{N_5}
    \frac{\prod_{l\leq n_I-n_J+m_I-m_J}(\zeta_I-\zeta_J+l)}{\prod_{l\leq 0}(\zeta_I-\zeta_J+l)}   \prod_{F=1}^{N_6+N_7}\prod_{I=1}^{N_5}\frac{\prod_{l\leq 0}(\zeta_I-\zeta_F+l)}{\prod_{l\leq m_I+n_I-n_F}(\zeta_I-\zeta_F+l)}
    q_5^{\abs{\vec m}+\sum_{I=1}^{N_5}n_{I}}\,.
\end{align}
We do some combinatorics as follows: we multiply the equation \eqref{eqn:Iofstar} by the following trivial formula,
\begin{align*}
    \prod_{I=1}^{N_5}\prod_{i=3,4}
    \prod_{A=1}^{N_i}
    \frac{\prod_{l\leq n^i_A-n_I}(\zeta_A-\zeta_I+l)}{\prod_{l\leq n^i_A-n_I}(\zeta_A-\zeta_I+l)}
    \prod_{I\neq J}^{N_5}
    \frac{\prod_{l\leq n_I-n_J}(\zeta_I-\zeta_J+l)}{\prod_{l\leq n_I-n_J}(\zeta_I-\zeta_J+l)}   \prod_{F=1}^{N_6+N_7}\prod_{I=1}^{N_5}\frac{\prod_{l\leq n_I-n_F}(\zeta_I-\zeta_F+l)}{\prod_{l\leq n_I-n_F}(\zeta_I-\zeta_F+l)}
\end{align*}
we can transform \eqref{eqn:Iofstar} to 
\begin{subequations}
    \begin{align}
     p^*I^{\mc X_s,S}(\vec q)\rvert_{P}=&\sum_{\vec m\in \Z^{N_5}_{\geq 0}}
    \prod_{i=3,4}\prod_{A=1}^{N_i}
    \prod_{I=1}^{N_5}
    \frac{\prod_{l\leq n^i_A-n_I}(\zeta_A-\zeta_I+l)}{\prod_{l\leq n^i_A-n_I-m_I}(\zeta_A-\zeta_I+l)}
    \prod_{i\neq 5}q_i^{\abs{\vec n^i}}q_5^{\abs{\vec m}+\sum_{I=1}^{N_5}n_{I}}\label{eqn:Is1}\\
    & \prod_{I\neq J}^{N_5}
    \frac{\prod_{l\leq n_I-n_J+m_I-m_J}(\zeta_I-\zeta_J+l)}{\prod_{l\leq n_I-n_J}(\zeta_I-\zeta_J+l)}   \prod_{F=1}^{N_6+N_7}\prod_{I=1}^{N_5}\frac{\prod_{l\leq n_I-n_F}(\zeta_I-\zeta_F+l)}{\prod_{l\leq m_I+n_I-n_F}(\zeta_I-\zeta_F+l)}\label{eqn:Is2}\\
    &
\prod_{i=3,4}\prod_{A=1}^{N_i}\prod_{I=1}^{N_5}
    \frac{\prod_{l\leq 0}(\zeta_A-\zeta_I+l)}{\prod_{l\leq n^i_A-n_I}(\zeta_A-\zeta_I+l)}  \prod_{I=1}^{N_5}\prod_{F=N_5+1}^{N_6+N_7}\frac{\prod_{l\leq 0}(\zeta_I-\zeta_F+l)}{\prod_{l\leq n_I-n_F}(\zeta_I-\zeta_F+l)}\label{eqn:Is3}.
\end{align}
\end{subequations}
Notice that for fixed $\vec n^i$, $i\neq 5$, the formula \eqref{eqn:Is3} is fixed, and  \eqref{eqn:Is1} and  \eqref{eqn:Is2} can be viewed as the restriction of degree $\vec m$-term of the equivariant quasimap small $I$-function of $S^{\oplus(N_3+N_4)}\rightarrow Gr(N_5,N_6+N_7)$ to $([N_5]\subset [N_6+N_7])$ in \eqref{eqn:Ifungrbmres}, if we let the equivariant parameters of $(\C^*)^{N_3+N_4}$-action be $\zeta_A+n^i_A$, $i=3,4,A=1,\ldots,N_i$, and equivariant parameters of $(\C^*)^{N_6+N_7}$ be  $\zeta_F+n_F$, $F=1,\ldots,N_6+N_7$.

On the other hand, consider the restriction of $p^*I^{\tilde{\mc Z}_s, R}(\vec q')$ to $\iota_s(P)$. The restrictions of the ingredients are $x^i_I\rvert_{\iota_s(P)}=\zeta_I$, for $i=1,2,3,4$, $x^5_I\rvert_{\iota_s(P)}=\zeta_{N_5+I}$, $x^6_I\rvert_{\iota_s(P)}=\zeta_I$, $x^7_I\rvert_{\iota_s(P)}=\zeta_{N_6+I}$. Then, 
\begin{align}\label{eqn:ImsP1}
    p^*I^{\tilde{\mc Z}_s, S}(\vec q')\rvert_{\iota_s(P)}
    &=\sum_{(\vec n^i)\in \op{Eff}_T^{ms}}(\mathrm{Irrel})\bigcdot
\prod_{i=3,4}\prod_{A=1}^{N_i}
\prod_{I=1}^{N_5'}\frac{\prod_{l\leq n^i_A-n^5_I}(\zeta_A-\zeta_{N_5+I}+l)}{\prod_{l\leq 0}(\zeta_A-\zeta_{N_5+I}+l)}
\nonumber\\
    &\prod_{I\neq J}^{N_5'}
    \frac{\prod_{l\leq n^5_I-n^5_J}(\zeta_{N_5+I}-\zeta_{N_5+J}+l)}{\prod_{l\leq 0}(\zeta_{N_5+I}-\zeta_{N_5+J}+l)}\prod_{F=1}^{N_6+N_7}\prod_{I=1}^{N_5'}
    \frac{\prod_{l\leq 0}(-\zeta_{N_5+I}+\zeta_F+l)}{\prod_{l\leq -n^5_I+n_F}(-\zeta_{N_5+I}+\zeta_F+l)}\nonumber\\
    &\prod_{i=3,4}\prod_{A=1}^{N_i}
    \prod_{F=1}^{N_6+N_7}
    \frac{\prod_{l\leq 0}(\zeta_A-\zeta_F+l)}{\prod_{l\leq n^i_A-n_F}(\zeta_A-\zeta_F+l)}\prod_{i\in Q_0\backslash Q_f}(q_i')^{\abs{\vec n^i}}\,.
\end{align}
By observation, we must have $m_I=-n^5_I+n_{N_5+I}\geq 0$. Otherwise, the corresponding term would vanish. 
We substitute $n_I^5=n_{N_5+I}-m_I$, do similar combinatorics as what we have done to \eqref{eqn:Iofstar}, and transform the above $I$-function to the following formula by fixing $n^i_A,i\neq 5$ and disregarding the irrelevant part,
\begin{subequations}
    \begin{align}
    &p^*I^{\tilde{\mc Z}_s, S}(\vec q')\rvert_{\iota_s(P)}
    =\sum_{\vec m^i\in \Z^{N_5'}_{\geq 0}}
    \prod_{I\neq J}^{N_5'}
    \frac{\prod_{l\leq n_{N_5+I}-n_{N_5+J}-m_I+m_J}(\zeta_{N_5+I}-\zeta_{N_5+J}+l)}{\prod_{l\leq n_{N_5+I}-n_{N_5+J}}(\zeta_{N_5+I}-\zeta_{N_5+J}+l)}
  \label{eqn:Ims1}\\
   &\prod_{i=3,4}\prod_{A=1}^{N_i}
 \prod_{I=1}^{N_5'}\frac{\prod_{l\leq n^i_A-n_{N_5+I}+m_I}(\zeta_A-\zeta_{N_5+I}+l)}{\prod_{l\leq n^i_A-n_{N_5+I}}(\zeta_A-\zeta_{N_5+I}+l)}
 \prod_{F=1}^{N_6+N_7}\prod_{I=1}^{N_5'}
    \frac{\prod_{l\leq -n_{N_5+I}+n_F}(-\zeta_{N_5+I}+\zeta_F+l)}{\prod_{l\leq -n_{N_5+I}+n_F+m_I}(-\zeta_{N_5+I}+\zeta_F+l)}\label{eqn:Ims2}\\  &\prod_{i=3,4}\prod_{A=1}^{N_i}\prod_{I=1}^{N_5}
    \frac{\prod_{l\leq 0}(\zeta_A-\zeta_I+l)}{\prod_{l\leq n^i_A-n_I}(\zeta_A-\zeta_I+l)}  \prod_{I=1}^{N_5}\prod_{F=N_5+1}^{N_6+N_7}\frac{\prod_{l\leq 0}(\zeta_I-\zeta_F+l)}{\prod_{l\leq n_I-n_F}(\zeta_I-\zeta_F+l)}(q_5')^{-\abs{\vec m}+\sum_{I=1}^{N_5'}n_{N_5+I}}\label{eqn:Ims3}.
\end{align}
\end{subequations}
We find that for fixed $\vec n^i, i\neq 5$, the formula \eqref{eqn:Ims3} is fixed, and \eqref{eqn:Ims1} and \eqref{eqn:Ims2} can be viewed as the restriction of degree $\vec m$ term of equivariant quasimap small $I$-function of $Gr^\vee$ to a torus fixed point $([N_6+N_7]\backslash [N_5]\subset [N_6+N_7])$, if we let equivariant parameters of $(\C^*)^{N_3+N_4}$ be $n^i_A+\zeta_A$, $i=3,4,\,A=1,\ldots,N_i$  and equivariant parameters of $(\C^*)^{N_6+N_7}$ be $n_F+\zeta_F$.

The irrelevant parts of $p^*I^{\mc X_s, S}\rvert_{P}$ and $p^*I^{\tilde{\mc Z}_s,S}\rvert_{\iota_s(P)}$ are equal, and formulas \eqref{eqn:Is3} and \eqref{eqn:Ims3} are equal for fixed $\vec n^i,\,i\neq 5$.
Formulas \eqref{eqn:Is1} \eqref{eqn:Is2}  and  \eqref{eqn:Ims1} \eqref{eqn:Ims2} are related by the fundamental building block in the Theorem \ref{thm:Haidong}.
In the two sets $\mathrm{Eff}_T^{s}$ and $\mathrm{Eff}_T^{ms}$, $\vec n^i$ for $i\neq 5$ are the same. For fixed $\vec n^i,\,i\neq 5$, the $\vec n^5$ in the two sets are related via the variable change. 
Hence, we have proved the proposition. 
\end{proof}
Notice that the above proposition can be extended to any pair of $S$-fixed points 
$(P,\iota_s(P))\in \mathfrak F_s^b \times \mathfrak F_s^a$.
Therefore, we have proved the Theorem \ref{thm:1st} for the star-shaped quiver by localization. 

\begin{cor}\label{cor:generalstar}
 For a general star-shaped quiver $\mc X_g$ and its quiver mutation $\tilde{\mc Z}_g$, their quasimap small $I$-functions restricted to a pair of torus fixed points $(Q,\iota(Q))\in \mathfrak F_g^b\times \mathfrak F_g^a$ are related in the same way as Proposition \ref{prop:Istartransform}. 
\end{cor}
\begin{proof}
   One can follow the proof in Proposition \ref{prop:Istartransform} step by step to prove it. 
   Without loss of generality, we choose a torus fixed point $Q=(\vec C_{[N_i]})\in \mathfrak F_g^b$ described as follows.
\begin{itemize}
    \item  For any node $p$ on the right hand side of node $k$, $\vec C_{[N_p]}=[N_p]$, in particular, $\vec C_{[N_{i_b}]}=[N_{i_b}]$ for $b=1,\ldots, l$. 
    \item  We relabel the set of integers $\bigcup_{b=1}^l\vec C_{[N_{i_b}]}$ by $\{1,\ldots,N_f(k)\}$, and choose $\vec C_{[N_i]}=[N_i]\subset[N_f(k)]$ for nodes $i$ on the left hand side of $k$.
    \end{itemize}
Denote the equivariant parameters of the torus $\prod_{i\in Q_f}(\mathbb C^*)^{N_i}$ by $\lambda_{F}^i$. 
Restricted to $Q$, we have 
$x^{i_b}_I|_Q=\lambda_{I}^i$, for $b=1,\ldots, l$. Rewrite the collection of equivariant parameters $\cup_{b=1}^l \{\lambda_{I}^i\}_{I=1}^{N_{i_b}}$ by $\zeta_1,\ldots, \zeta_{N_f(k)}$. Then $x^{j_a}_{I}\rvert_{Q}=\zeta_I$ for any $a=1\ldots h$.
Then the the quasimap small $I$-function restricted to $Q$ is similar with Equation \eqref{eqn:ItoP} except that the range of product for $i$ in the formula \eqref{eqn:ItoP0} is $1,\ldots,h$ instead of $3,4$, and the product for $F$ in the formula \eqref{eqn:ItoP} is from $1$ to $N_f(k)$. 
Then by similar combinatorics, we can get a similar result with \eqref{eqn:Is1}\eqref{eqn:Is2}\eqref{eqn:Is3}. 
Similarly, we can deal with the $I$-function of $\tilde{\mc Z}_g$ as what we have done to $\tilde{\mc Z}_s$. Hence, the $I$-functions of $\mc X_g$ and $\tilde{\mc Z}_g$ satisfy the same transformation law. 
\end{proof}

The $D_3$-type quiver in Figure \ref{fig:star1} can be viewed as a special star-shaped quiver with only one outgoing arrow and 2 incoming arrows. 
\begin{cor}\label{cor:D3}
    \begin{equation}
        I^{\mc X_0,R}(\vec q)=(1+(-1)^{N_3'}q_3)^{\sum_{I=1}^{N_1}x^1_I+\sum_{I=1}^{N_2}x^2_I-\sum_{F}^{N_3}\lambda_F+N_3'}I^{\mc Z_1,R}(\vec q')\,,
    \end{equation}
under the transformation of K\"ahler variables
 \begin{equation}\label{eqn:varI0toI1}
    q_1'=(1+(-1)^{N_3'}q_3)q_1,\,q_2'=(1+(-1)^{N_3'}q_3)q_2,\,q_3'=q_3^{-1}\,.
 \end{equation}
\end{cor}   

\subsection{Proof for Theorem \ref{thm:2nd}}\label{sec:prooftheorem2}
\subsubsection{Proofs for the equivalence among ${I^{\mc Z_1,R}}$ ${I^{\mc Z_2,R}}$ and $I^{\mc Z_3,R}$}
The Theorem \ref{thm:2nd} item (2) can be concluded by the following Proposition. 
\begin{prop}\label{prop 5.9}
Let $P_1\in \mathfrak F_1$ and  $P_2=\iota_1(P_1)\in \mathfrak F_2$ be an arbitrary pair of torus fixed points. Then
\begin{equation}
    p^*I^{\mc Z_1,R}(\vec q)\rvert_{P_1}=p^*I^{\mc Z_2,R}(\vec q')\rvert_{P_2}\,
\end{equation}
under the variable change
\begin{equation}
    q_1'=q_1^{-1}, \,q_2'=q_2,\,q_3'=q_3.
\end{equation}
\end{prop}
\begin{proof} 
Without loss of generality, we consider 
\begin{equation}\label{eqn:P1proof}
    P_1=( [N_1],{[N_2]},\vec C_{[N_3']})\in\mathfrak F_1  ,\,\, \vec C_{[N_3']}=[N_4]\backslash [N_3]\,.
\end{equation}
Its image $\iota_1(P_1)\in \mathfrak F_2$ can be represented by 
\begin{equation}
    P_2:=\iota_1(P_1)=([N_4]\backslash [N_1],[N_2], [N_4]\backslash[N_3])\,.
\end{equation}
By the description of $I$-effective classes in \eqref{eqn:ruleofeffclass}, we have
\begin{equation*}
  \op{Eff}_T^1=\{(\vec n^1, \vec n^2, \vec n^3)\in \Z^{N_1}_{\geq 0}\times \Z^{N_2}_{\geq 0}\times \Z^{N_4-N_3}_{\leq0}\}.
\end{equation*}
Being restricted to $P_1$,  $x_I^i\rvert_{P_1}=\lambda_I$ for $i=1,2$, and $x^3_I\rvert_{P_1}=\lambda_{N_3+I}$.
The restriction of $I^{\mc Z_1,R}$ to $P_1$  can be transformed to the following formula by the similar strategy with that in the proof of Proposition \ref{prop:Istartransform}
\begin{subequations}
\begin{align}
    p^*I^{\mc Z_1,R}(\vec q)\rvert_{P_1}&=
    \sum_{(\vec n^i) \in \op{Eff}_T^1}
    \prod_{\substack{I, J=1\\ I\neq J}}^{N_2}\frac{\prod_{l\leq n^2_I-n^2_J}(\lambda_I-\lambda_J+l)}{\prod_{l\leq 0}(\lambda_I-\lambda_J+l)}
   \prod_{\substack{I, J=1\\ I\neq J}}^{N_3'}\frac{\prod_{l\leq n^3_I-n^3_J}(\lambda_{N_3+I}-\lambda_{N_3+J}+l)}{\prod_{l\leq 0}(\lambda_{N_3+I}-\lambda_{N_3+J}+l)}
   \label{I1:2a}
   \\
   &\prod_{F=1}^{N_4}
   \prod_{I=1}^{N_2}\frac{\prod_{l\leq 0}(\lambda_I-\lambda_F+l)}{\prod_{l\leq n^2_I}(\lambda_I-\lambda_F+l)}
   \prod_{J=1}^{N_3'}
   \prod_{I=1}^{N_1}\frac{\prod_{l\leq 0}(-\lambda_{N_3+J}+\lambda_{N_1+I}+l)}{\prod_{l\leq -n^3_J}(-\lambda_{N_3+J}+\lambda_{N_1+I}+l)}\label{I1:2b}
   \\
   &
     \prod_{J=1}^{N_3'} \prod_{I=1}^{N_2}\frac{\prod_{l\leq n^2_I-n^3_J}(\lambda_I-\lambda_{N_3+J}+l)}{\prod_{l\leq 0}(\lambda_I-\lambda_{N_3+J}+l)}\prod_{i=1}^3 q_i^{\abs{\vec n^i}}
     \label{I1:2c}
     \\
     &
    \prod_{\substack{I, J=1\\ I\neq J}}^{N_1}\frac{\prod_{l\leq n^1_I-n^1_J}(\lambda_I-\lambda_J+l)}{\prod_{l\leq 0}(\lambda_I-\lambda_J+l)}
   \prod_{I=1}^{N_1}
   \frac{\prod_{J=1}^{N_3'}\prod_{l=1}^{ n^1_I}(\lambda_I-\lambda_{N_3+J}-n^3_J+l)}{\prod_{F=1}^{N_4}\prod_{l=1}^{n^1_I}(\lambda_I-\lambda_F+l)} 
.\label{I1:2d}
\end{align}
\end{subequations}
In the above formula, for fixed $\vec n^2,\,\vec n^3$, formulas \eqref{I1:2a} \eqref{I1:2b} and \eqref{I1:2c} are fixed, 
and the formula \eqref{I1:2d} is the restriction of quasimap small $I$-function of $Gr^\vee$ in \eqref{eqn:Ifungramres} to the torus fixed point $([N_1]\subset [N_4])$ if we let $-\lambda_{N_3+J}-n^3_J$, $J=1,\ldots,N_3'$ and $-\lambda_F$, $F=1,\ldots,N_4$ be the equivariant parameters of torus $(\C^*)^{N_3'}\times (\C^*)^{N_4}$ action on $Gr^\vee$.

Now we consider $\mc Z_2$. 
The set of $I$-effective classes is 
\begin{align}\label{eqn:Effective2}
  \op{Eff}_T^2=&\{(\vec n^1, \vec n^2, \vec n^3)\in \Z^{N_2}_{\leq 0}\times \Z^{N_2}_{\geq 0}\times \Z^{N_4-N_3}_{\leq 0}|\exists \text{ distinct integers } l_1, l_2, \ldots, l_{N_4-N_3}, \nonumber\\
&\text{such that } n^1_{l_I}-n^3_I\geq 0\}
\end{align}
Consider the restriction of $I^{\mc Z_2,R}(\vec q)$ to $P_2$, $x_I^1\rvert_{P_2}=\lambda_{N_1+I}$, $x^2_I\rvert_{P_2}=\lambda_I$, $x^3_{I}\rvert_{P_2}=\lambda_{N_3+I}$. 
Then by the similar combinatorics,
$ p^*I^{\mc Z_2,R}(\vec q')\rvert_{P_2}$ can be transformed to 
\begin{subequations}
\begin{align}
    p^*I^{\mc Z_2,R}(\vec q')\rvert_{P_2}&=\sum_{(\vec n^i)
    \in \op{Eff}_T^2}
    \prod_{\substack{I, J=1\\ I\neq J}}^{N_2}\frac{\prod_{l\leq n^2_I-n^2_J}(\lambda_I-\lambda_J+l)}{\prod_{l\leq 0}(\lambda_I-\lambda_J+l)}
   \prod_{\substack{I, J=1\\ I\neq J}}^{N_3'}\frac{\prod_{l\leq n^3_I-n^3_J}(\lambda_{N_3+I}-\lambda_{N_3+J}+l)}{\prod_{l\leq 0}(\lambda_{N_3+I}-\lambda_{N_3+J}+l)}\label{I2:1a}
    \\
    &
   \prod_{F=1}^{N_4}
   \prod_{I=1}^{N_2}\frac{\prod_{l\leq 0}(\lambda_I-\lambda_F+l)}{\prod_{l\leq n^2_I}(\lambda_J-\lambda_F+l)}
   \prod_{I=1}^{N_2}\prod_{J=1}^{N_3'}\frac{\prod_{l\leq n^2_I-n^3_J}(\lambda_I-\lambda_{N_3+J}+l)}{\prod_{l\leq 0}(\lambda_I-\lambda_{N_3+J}+l)}
   \label{I2:1b}
   \\
  &
   \prod_{I=1}^{N_2}\prod_{J=1}^{N_3'}\frac{\prod_{l\leq 0}(\lambda_{N_1+I}-\lambda_{N_3+J}+l)}{\prod_{l\leq -n^3_J}(\lambda_{N_1+I}-\lambda_{N_3+J}+l)}
   \prod_{i=1}^3(q_i')^{\abs{\vec n^i}}\label{I2:1c}
   \\
   &
   \prod_{\substack{I, J=1\\ I\neq J}}^{N_2}\frac{\prod_{l\leq n^1_I-n^1_J}(\lambda_{N_1+I}-\lambda_{N_1+J}+l)}{\prod_{l\leq 0}(\lambda_{N_1+I}-\lambda_{N_1+J}+l)}\prod_{I=1}^{N_2}\frac{\prod_{J=1}^{N_3'}\prod_{l=0}^{-n^1_I-1}(\lambda_{N_1+I}-\lambda_{N_3+J}-n^3_J-l)}
    {\prod_{F=1}^{N_4}\prod_{l=1}^{-n^1_I}(-\lambda_{N_1+I}+\lambda_F+l)}\label{I2:1d}\,.
\end{align}
\end{subequations}
For fixed $\vec n^2,\vec n^3$, the formulas \eqref{I2:1a}\eqref{I2:1b}\eqref{I2:1c} are fixed. 
The formula \eqref{I2:1d} is the degree $\vec n^1$ term of the equivariant quasimap small $I$-function of $Gr$ in \eqref{eqn:Ifungrbmres} being restricted to the $(\C^*)^{N_3'}\times (\C^*)^{N_4}$-fixed point $([N_4]\backslash[N_1]\subseteq [N_4])$, if we let $-\lambda_{N_1+J}-n^3_J$ $J=1,\ldots,N_3'$ and $-\lambda_F$ $F=1,\ldots,N_4$ denote equivariant parameters for $(\C^*)^{N_3'}\times (\C^*)^{N_4}$-action. 

Compare $p^*I^{\mc Z_1,R}(\vec q)\rvert_{P_1}$ with $p^*I^{\mc Z_2,R}(\vec q')\rvert_{P_2}$, and we find that for fixed $\vec n^2, \vec n^3$, formulas \eqref{I2:1a}\eqref{I2:1b}\eqref{I2:1c} are exactly equal to formulas \eqref{I1:2a}\eqref{I1:2b}\eqref{I1:2c}. Formulas \eqref{I1:2d} is equal to \eqref{I2:1d} for a degree $\abs{\vec n^1}$ by Theorem \ref{thm:Haidong} item 1. 

Another issue is about the difference between $\op{Eff}_T^1$ and $\op{Eff}_T^2$. 
Observe that there is a factor $\prod_{I=1}^{N_2}\prod_{J=1}^{N_3'}\prod_{l=0}^{-n_I^1-1}(\lambda_{N_1+I}-\lambda_{N_3+J}-n^3_J-l)$ in $p^*I^{\mc Z_2,R}$ which will vanish if $n^1_{N_3-N_1+J}-n^3_J+1\leq 0$.
Hence, we can enlarge the $\op{Eff}_T^2$ to $\widetilde{\op{Eff}}^2_T=\{(\vec n^1,\vec n^2,\vec n^3)\in \mathbb Z^{N_2}_{\leq 0}\times \mathbb Z^{N_2}_{\geq 0}\times\mathbb Z^{N_3'}_{\leq 0}\}$ without changing $p^*I^{\mc Z_2,R}$ because those terms that are not in $\op{Eff}^2_T$ vanish. 
Then we match $\op{Eff}^1_T$ and $\widetilde{\op{Eff}}^2_T$ by sending $(\vec n^1,\vec n^2,\vec n^3)$ to $(-\vec n^1,\vec n^2,\vec n^3)$.

Hence, we have proved the Proposition.
By generalizing the above procedure to any pair of torus fixed points $(P,\iota_1(P))\in 
\mathfrak{F}_1\times\mathfrak{F}_2$, we can 
conclude the quiver mutation $\mu_1$ preserves the equivariant quasimap small $I$-functions for $\mc Z_1$ and $\mc Z_2$.
\end{proof}
Similar combinatorics can be applied to the proof of the equivalence of equivariant quasimap small $I$-functions of $\mc Z_2,\mc Z_3$.
\begin{prop}\label{prop:I2I3}
    The quasimap small $I$-functions of $\mc Z_2$ and $\mc Z_3$ satisfy the following relation 
    \begin{align}
        p^*I^{\mc Z_2,R}(q_1,q_2,q_3)=p^*I^{\mc Z_3,R}(q_1,q_2^{-1},q_3).
    \end{align}
\end{prop}

\subsubsection{Proof for the equivalence between $I^{\mc Z_3}$ and $I^{\mc X_4}$}
We consider the equivariant quasimap small $I$-functions of $\mc Z_3$ and $\mc X_4$. 
Let $P_3\in \mathfrak F_3$ and $P_4=\iota_3(P_3)\in \mathfrak F_4$, see Section \ref{sec:appQ4} for the description $\mathfrak F_4$ and $\iota_3$. 
\begin{prop}\label{prop:I3I4}
  \begin{equation}
    I^{\mc X_4,R}(q_1,q_2,q_3)\big|_{P_4}=(1+(-1)^{N_3^{\prime}}q_3)^{\sum_{F=1}^{N_4}\lambda_F-\sum_{I=1}^{N_2}x^1_I-\sum_{I=1}^{N_1}x^2_I+N_3^{\prime}} I^{\mc Z_3,R}(q_1',q_2',q_3')\big|_{P_3}\,,
\end{equation}
under change of K\"ahler variables
\begin{equation}
   q_1'=\frac{q_3 q_1}{1+(-1)^{N_3^{\prime}}q_3}, \,
   q_2'=\frac{q_3 q_2}{1+(-1)^{N_3^{\prime}}q_3},\,
   q_3'=q_3^{-1}\,.
 \end{equation}
\end{prop}
\begin{proof}
The proof is similar with the previous situation, but this example is a little complicated. 
The $I$-effective classes for $\mc Z_3$ are 
\begin{align}
    \mathrm{Eff}_T^3=\{(\vec n^1,\vec n^2,\vec n^3)\in& \mathbb Z^{N_2}_{\leq 0}\times \mathbb Z^{N_1}_{\leq 0}\times \mathbb Z^{N_3'}\big | 
    \forall \,I\in  [N_3'], \,\exists \text{ distinct integers }\{k_I\}_{I=1}^{N_3'}\subset [N_2],\, \nonumber\\
    &\text{distinct integers }\{j_I\}_{I=1}^{N_3'}\subset [N_1]\,\mathrm{s.t.} \,
    n^1_{k_I}-n^3_{I}\geq 0, \, 
    n^2_{j_I}-n^3_I\geq 0\,
    \}.
\end{align}
Without loss of generality, we choose a torus fixed point 
$P_3=(\vec A_{[N_2]},\vec B_{[N_1]},\vec C_{[N_4-N_3]})$ such that $\vec C_{[N_4-N_3]}=\{2N_3-N_4+1,\ldots,N_3\}$, 
$\vec A_{[N_2]}=\{1,\ldots,N_3-N_1\}\cup \vec C_{[N_4-N_3]}$, 
and $\vec B_{[N_1]}=\{N_3-N_1+1,\ldots,2N_3-N_4\}\cup \vec C_{[N_4-N_3]}$.
Then define  $\zeta^1_A:=x^1_A\big|_{P_3} =\lambda_A$ for $A=1,\ldots,N_3-N_1$,  $\zeta^1_A:=x^1_A\big|_{P_3} =\lambda_{A+N_3-N_2}$ for $A=N_3-N_1+1,\ldots,N_2$, $\zeta_B^2:=x^2_B|_{P_3}=\lambda_{B+N_3-N_1}$ for all $B\in [N_1]$, and $x^3_I|_{P_3}=\lambda_{2N_3-N_4+I}$ for $I\in[N_3']$.
Then the restriction of the equivariant quasimap small $I$-function of $\mc Z_3$ to $P_3$ is 
\begin{align}
    I^{\mc Z_3,R}\big|_{P_3}&=
    \sum_{(\vec n^i)\in \mathrm{Eff}_T^3}(\mathrm{Irrel})\cdot
    \prod_{I\neq J=1}^{N_3'}
    \frac{\prod_{l\leq n^3_I-n^3_J}(\lambda_{2N_3-N_4+I}-\lambda_{2N_3-N_4+J}+l)}
    {\prod_{l\leq 0}(\lambda_{2N_3-N_4+I}-\lambda_{2N_3-N_4+J}+l)}\nonumber\\
    &\prod_{I=1}^{N_3'}
    \prod_{A=1}^{N_2}
    \frac{\prod_{l\leq 0}(\zeta_A^1-\lambda_{2N_3-N_4+I}+l)}
    {\prod_{l\leq n^1_A-n^3_I}(\zeta_A^1-\lambda_{2N_3-N_4+I}+l)}
    \prod_{B=1}^{N_1}\prod_{I=1}^{N_3'}
    \frac{\prod_{l\leq 0}(\zeta_B^2-\lambda_{2N_3-N_4+I}+l)}
    {\prod_{l\leq n^2_B-n^3_I}(\zeta_B^2-\lambda_{2N_3-N_4+I}+l)}
    \nonumber\\
    &
    \prod_{F=1}^{N_4}\prod_{I=1}^{N_3'}
    \frac{\prod_{l\leq -n^3_I}(\lambda_{F}-\lambda_{2N_3-N_4+I}+l)}
    {\prod_{l\leq 0}(\lambda_{F}-\lambda_{2N_3-N_4+I}+l)}\prod_{F=1}^{N_4}
    \prod_{A=1}^{N_2}
    \frac{\prod_{l\leq 0}(-\zeta_A^1+\lambda_F+l)}
    {\prod_{l\leq -n^1_A}(-\zeta_A^1+\lambda_F+l)}
    \nonumber\\
    &
    \prod_{F=1}^{N_4}\prod_{B=1}^{N_1}
    \frac{\prod_{l\leq 0}(\lambda_F-\zeta_B^2+l)}
    {\prod_{l\leq -n^2_B}(\lambda_F-\zeta_B^2+l)}
    \prod_{i=1}^{3}(q_i^{\prime})^{\abs{\vec n^i}}.
\end{align}
Notice that for each $I\in [N_3']$, both $n^1_{N_3-N_1+I}-n^3_I\geq 0$ and $n^2_{N_3-N_2+I}-n^3_I\geq 0$ hold. 
Then if $n^1_{N_3-N_1+I}>n^2_{N_3-N_2+I}$, we substitute $n^3_I=n^2_{N_3-N_2+I}-d_I$ and if $n^1_{N_3-N_1+I}\leq n^2_{N_3-N_2+I}$, we let $n^3_I=n^1_{N_3-N_1+I}-d_I$. This forces us to split the set $[N_3']$ into two parts depending on the relation between $n^1_{N_3-N_1+I}, n^2_{N_3-N_2+I}$. This is to avoid poles in our expressions. 
In order to simplify the writing, we assume that for each $I\in [N_3']$, we have  $n^1_{N_3-N_1+I}\leq n^2_{N_3-N_2+I}$.
Let 
\begin{equation}\label{eqn:didefinitionz3}
    d_I=n^1_{N_3-N_1+I}-n^3_I\,,
\end{equation}
and then $d_I\geq 0$\,.
We substitute $n^3_I=n^1_{N_3-N_1+I}-d_I$ with $d_I\geq 0$ (Actually, without the assumption $n^1_{N_3-N_1+I}\leq n^2_{N_3-N_2+I}$, the following procedure is similar). 

Define 
\begin{equation}
\zeta_F=
    \begin{cases}
        &\zeta_F^1+n^1_F,\,\text{ for }F=1,\ldots,N_2,\,\\
        &\zeta^2_{F-N_2}+n^2_{F-N_2},\,\text{ for }F=N_2+1,\ldots,N_4
    \end{cases}
\end{equation}
By the combinatorics we have used in previous sections, for fixed $\vec n^1$ and $\vec n^2$,  we transform the degree $(\vec n^i)$ terms of $I^{\mc Z_3,R}\big|_{P_3}$ to the following formula
    \begin{align}\label{eqn:Q3restr}
       &\sum_{\vec d\in \mathbb Z^{N_3'}_{\geq 0}}\prod_{I\neq J}
        \frac{\prod_{l\leq -d_I+d_J}(\zeta_{N_3-N_1+I}-\xi_{N_3-N_1+J}+l)}
        {\prod_{l\leq 0}(\zeta_{N_3-N_1+I}-\zeta_{N_3-N_1+J}+l)}
        \prod_{I=1}^{N_3'}
        \frac{\prod_{F=1}^{N_4}
        \prod_{l=1}^{d_I}(\zeta_F-\zeta_{N_3-N_1+I}+l)}
        {\prod_{l=1}^{d_I}
        \prod_{F=1}^{N_4}
        (\zeta_F-\zeta_{N_3-N_1+I}+l)}\nonumber\\
       & \cdot f\cdot \prod_{i=1}^{2}(q_i^{\prime})^{\abs{\vec n^i}}(q_3^{\prime})^{-\abs{\vec d}+\sum_{I=1}^{N_3'}n^1_{N_3-N_1+I}}
    \end{align}
where 
\begin{align}\label{eqn:fbetweenQ3Q4}
     f=&\prod_{I=1}^{N_3'}
        \left(\prod_{A=1}^{N_3-N_1}
        \frac{\prod_{l\leq 0}(\lambda_A-\lambda_{2N_2-N_4+I}+l)}{\prod_{l\leq n^1_A-n^1_{I+N_3-N_1}}(\lambda_A-\lambda_{2N_2-N_4+I}+l)}
        \prod_{B=1}^{N_1}
        \frac{\prod_{l\leq 0}(\zeta_B^2-\lambda_{2N_3-N_4+I}+l)}
        {\prod_{l\leq n^2_B-n^1_{N_3-N_1+I}}(\zeta_B^2-\lambda_{2N_3-N_4+I}+l)}\right)
        \nonumber\\
        &\prod_{F=1}^{N_4}
        \left(
        \prod_{B=1}^{N_1}
        \frac{\prod_{l\leq 0}(\lambda_F-\zeta^2_B+l)}{\prod_{l\leq -n^2_B}(\lambda_F-\zeta^2_B+l)}
        \prod_{A=1}^{N_3-N_1}
        \frac{\prod_{l\leq 0}(\lambda_F-\lambda_A+l)}
        {\prod_{l\leq -n^1_A}(\lambda_F-\lambda_A+l)}
        \right)
\end{align}
Notice that the first row in \eqref{eqn:Q3restr} is the quasimap small $I$-function of $Gr^\vee$ in \eqref{eqn:Ifungramres} restricted to torus fixed points $\{N_3-N_1+1,\ldots,N_2\}\subset\{1,\ldots,N_4\}$ and the equivariant parameters are $\{\zeta_F\}_{F=1}^{N_4}$ and $\{\lambda_A\}_{A=1}^{N_4}$.

On the other hand, let $P_4\in \mathfrak F_4$ be the image of $\iota_3(P_3)$.
Then restricted to $P_4$, ingredients are $x^1_A|_{P_4}=\zeta^1_A$ for $A=1,\cdots, N_2$, $x_B^2|_{P_4}=\zeta_B^2$ for $B=1,\cdots, N_1$ and $x^3_I|_{P_4}=\lambda_I$ for $I=1,\cdots, N_3$.
The $I$-effective classes for $\mc X_4$ are 
\begin{align}
    \mathrm{Eff}_T^4=\big\{&(\vec n^1,\vec n^2,\vec n^3)\in  \mathbb Z^{N_2}_{\leq 0}
    \times \mathbb Z^{N_1}_{\leq 0}
    \times \mathbb Z^{N_3}\big|\,  \exists \text{ distinct } \{k_I\}_{I\in [N_2]}\subset [N_3], \text{ s.t.} 
    n^3_{k_I}-n^1_I\geq 0\,,\nonumber\\
    &\exists \text{ distinct } \{l_J\}_{J\in [N_1]}\in [N_3], \text{ s.t.} n^3_{l_J}-n^2_J\geq 0\,, 
    \{k_I\}_{I=1}^{N_2}\cup \{l_J\}_{J=1}^{N_1}=[N_3]
    \big\}.
\end{align}
The equivariant quasimap small $I$-function $I^{\mc X_4,R}$ restricted to $P_4$ becomes
\begin{align}\label{eq5.52}
    &I^{\mc X_4,R}\big|_{P_4}=
    \sum_{(\vec n^i)\in\mathrm{Eff}_T^4}
    (\mathrm{Irrel})
    \prod_{I\neq J}
    \frac{\prod_{l\leq n^3_I-n^3_J}(\lambda_I-\lambda_J+l)}
    {\prod_{l\leq 0}(\lambda_I-\lambda_J+l)}
    \prod_{I=1}^{N_3}\prod_{A=1}^{N_2}
    \frac{\prod_{l\leq 0}(\lambda_I-\zeta^1_A+l)}{\prod_{l\leq n^3_I-n^1_A}(\lambda_I-\zeta^1_A+l)}\nonumber
    \\
    &\prod_{I=1}^{N_3}\prod_{B=1}^{N_1}
    \frac{\prod_{l\leq 0}(\lambda_I-\zeta_B^2+l)}{\prod_{l\leq n^3_I-n^2_B}(\lambda_I-\zeta_B^2+l)}
    \prod_{F=1}^{N_4}
    \frac{\prod_{l\leq 0}(\lambda_F-\lambda_I+l)}
    {\prod_{l\leq -n^3_I}(\lambda_F-\lambda_I+l)}\prod_{i=1}^3q_i^{\abs{\vec n^i}}.
\end{align}
Denote by $G:=\{f_1<f_2<\ldots<f_{N_3}\}:=[N_4]\backslash \{N_3-N_1+1,\ldots, N_2\}$.
Then $\{\zeta_F\}_{F=1}^{N_4}\backslash
\{\zeta_I\}_{I=N_3-N_1+1}^{N_2}$ can be denoted by $\{\zeta_{{f_I}}\}_{I=1}^{N_3}$.
Similar as what we have done in $I^{\mc Z_3,R}\big|_{P_3}$, we  assume that $n^2_{I-N_3+N_1}\geq n^1_{I-N_3+N_2}$ and set
\begin{equation}\label{eqn:didefinitionx4}
d_I=
    \begin{cases}
        &n^3_I-n^1_I, \text{ for }I=1,\ldots,N_3-N_1\\
        &n^3_I-n^2_{I-N_3+N_1}, \text{ for }I=N_3-N_1+1,\ldots,N_3.
    \end{cases}
\end{equation}
Then $d_I\geq 0$.
Substitute $n^3_I=d_I+n^1_I$ for $I=1,\ldots,N_3-N_1$, and $n^3_I=d_I+n^2_{I-N_3+N_1}$ for $I=N_3-N_1+1,\ldots,N_3$ in to \eqref{eq5.52}. After the usual combinatorics we have used repeatedly, for fixed $\vec n^1$ and $\vec n^2$, we transform the $I^{\mc X_4,R}\big|_{P_4}$ to the following form by disregarding the irrelevant part,
\begin{align}\label{eqn:IQ4rest}
    &\sum_{\vec d\in \mathbb Z^{N_3}_{\geq 0}}
    \prod_{I\neq J}
    \frac{\prod_{l\leq d_I-d_J}(\zeta_{f_I}-\zeta_{f_J}+l)}
    {\prod_{l\leq 0}(\zeta_{f_I}-\zeta_{f_J}+l)}
    \prod_{I=1}^{N_3}
    \frac{\prod_{l=-d_I+1}^0\prod_{A=1}^{N_4}(\lambda_A-\zeta_{f_I}+l)}{\prod_{l=1}^{d_I}\prod_{F=1}^{N_4}(\zeta_{f_I}-\zeta_F+l)}\nonumber\\
    &\cdot f\cdot\prod_{i=1}^2q_i^{\abs{\vec n^i}}
    q_3^{\abs{\vec d}+\sum_{I=1}^{N_3-N_1}n^1_I+\sum_{I=1}^{N_1}n^2_{I}},
\end{align}
where $f$ is exactly equal to \eqref{eqn:fbetweenQ3Q4}. 
The first row of \eqref{eqn:IQ4rest} is the quasimap small $I$-function of $Gr$ restricted to torus fixed point $\{N_4\}\backslash \{N_3-N_1+1,\ldots,N_2\}\subset [N_4]$ in \eqref{eqn:Ifungramres}, and the corresponding equivariant parameters are $\{\zeta_F\}_{F=1}^{N_4}$ and $\{\lambda_A\}_{A=1}^{N_4}$. 
Therefore, for any $n_1\in \Z$, $n_2\in \Z$, the formulas \eqref{eqn:IQ4rest} and \eqref{eqn:Q3restr} match by Theorem \ref{thm:Haidong}.

Comparing $\op{Eff}_T^3$ and $\op{Eff}_T^4$, one can find that $\vec n^1$ and $\vec n^2$ range in $\mathbb Z^{N_2}_{\leq 0}$ and $\mathbb Z^{N_1}_{\leq 0}$ for both $\op{Eff}_T^3$ and $\op{Eff}_T^4$. For the $\vec n^3$, 
they are related via the defined the $\vec d\in \mathbb Z^{N_3'}_{\geq 0}$ in \eqref{eqn:didefinitionz3} and the $\vec d\in \mathbb Z^{N_3}_{\geq 0}$ in \eqref{eqn:didefinitionx4} and the transformation of the fundamental building block.
Hence via the fundamental building block, we can get the relation between the restricted $I$-functions $I^{\mc Z_3,R}(q_1',q_2',q_3')\big|_{P_3}$ and $I^{\mc X_4,R}(q_1,q_2,q_3)\big|_{P_4}$.
 One can check that for any pair $(P,\iota(P))\in \mathfrak F_3\times\mathfrak F_4$ of torus fixed points, the similar relation between $I^{\mc Z_3,R}(q_1',q_2',q_3')\big|_{\iota(P)}$ and $I^{\mc X_4,R}(q_1,q_2,q_3)\big|_{P}$ holds.

\end{proof}
By localization, we can prove the Theorem \ref{thm:2nd} item (3).

\subsubsection{The equivalences among $I^{\mc X_4}$ $I^{\mc X_5}$ and $I^{\mc X_6}$, and among $I^{\mc X_7}$ $I^{\mc X_8}$ and $I^{\mc X_9}$}\label{5.4.3}
The equivalence between ${I^{\mc X_4,R}(q_1,q_2,q_3)}$ and ${{I^{\mc X_5,R}(q_1,q_2,q_3)}}$ is the same type with the equivalence between $I^{\mc X_7}$ and $I^{\mc X_8}$, where we are performing a quiver mutation to the node 1 who is only connected to the gauge node $3$. Locally, the quiver mutation is reduced to the mutation between $Gr(N_2,N_3)$ and $Gr(N_3-N_2,N_3)$ by using the combinatorics we have used for the $A_n$ \cite{zhang2021gromov}. Hence, we have
\begin{equation}
    I^{\mc X_4,R}(q_1,q_2,q_3)=I^{\mc X_5,R}(q_1^{-1},q_2,q_3q_1)\,,
\end{equation}
and 
\begin{equation}
    I^{\mc X_7,R}(q_1,q_2,q_3)=I^{\mc X_8,R}(q_1^{-1},q_2,q_3q_1)\,.
\end{equation}

The equivalence between ${I^{\mc X_5,R}(q_1,q_2,q_3)}$ and ${I^{\mc X_6,R}(q_1,q_2,q_3)}$ is the same type with the equivalence between ${I^{\mc X_8,R}(q_1,q_2,q_3)}$ and ${I^{\mc X_9,R}(q_1,q_2,q_3)}$. They can be reduced to the relation of $I$-functions between $Gr(N_1,N_3)$ and $Gr(N_3-N_1, N_3)$. Hence, we have
\begin{equation}
    I^{\mc X_6,R}(q_1,q_2,q_3)=I^{\mc X_5,R}(q_1,q_2^{-1},q_3q_2)\,.
\end{equation}
and 
\begin{equation}
    I^{\mc X_8,R}(q_1,q_2,q_3)=I^{\mc X_9,R}(q_1,q_2^{-1},q_3q_2)\,.
\end{equation}

\subsubsection{Proof for the equivalence between $I^{\mc X_6,R}$ and $I^{\mc X_7,R}$}
In this section, we carefully prove the equivalence between $I^{\mc X_6,R}$ and $I^{\mc X_7,R}$. 
Actually, this proof is similar with that between $I^{\mc Z_3,R}$ and $I^{\mc X_4,R}$. 

Let $\iota_6:\mathfrak F_{6}\rightarrow \mathfrak F_7$ be the bijection which is described in Corollary \ref{cor:appx6x7torusfixedpoints} in the Appendix. 

In this section we set $M_1=N_3-N_2, M_2=N_3-N_1, M_3=N_3, M_4=N_4$ for simplification. 

The effective classes can be written as follows via the rule in Section \ref{quasimap}.
\begin{align}
     \mathrm{Eff}^6_{T}=&
     \{(\vec n^1, \vec n^2, \vec n^3)\in \Z^{M_1}\times \Z^{M_2}\times \Z^{M_3}|\text{for }i=1,2,\, \exists\, \text{distinct integers } k_1,\ldots,k_{M_i},\nonumber \\
    & s.t.\, -n^3_{k_j}+n^i_j\geq 0,\,
    \cap_{i=1}^2
    \{k_1,\ldots,k_{M_i}\}
    =\emptyset;\, 
    \text{ for } k\in [N_3]\backslash \cup_{i=1}^2\{k_1,\ldots,k_{M_i}\},\, n^3_k\geq 0 \}\nonumber\\
     \mathrm{Eff}_{T}^7= 
     &\{(\vec n^1, \vec n^2, \vec n^3)\in \Z^{M_1}\times \Z^{M_2}\times \Z_{\geq 0}^{M_3}|\text{for } i=1, 2, \exists \text{ distinct integers } k_1, \ldots, k_{M_i}\nonumber
     \\ &
     \text{such that } n^3_{k_j}-n^i_j\geq 0 \}.
\end{align}
Let $P_6\in \mathfrak F_6$ and $P_7:=\iota_6(P_6)\in \mathfrak F_7$ be a pair of torus fixed points. 
\begin{prop}\label{prop:I6I7}
 We have
  \begin{equation}
    I^{\mathcal{X}_7, R}(q_1,q_2,q_3)\big|_{P_7}=
    I^{\mathcal{X}_6, R}(q_1',q_2',q_3')\big|_{P_6},
  \end{equation}
  under the change of K\"ahler variables
\begin{equation}\label{Kahler change}
 q_1=q_1'q_3', q_2=q_2'q_3', q_3=(q_3')^{-1}\,.
\end{equation}
\end{prop}
\begin{proof}
Without loss of generality, we consider the torus fixed point $P_7\in \mathfrak F_7$ which is $P_7=([M_1],[M_2],[M_3])$. 
Then the restriction of Chern roots are $x^i_A\big|_{P_7}=\lambda_{A}, i=1,,2,3$.
Then the restriction of $I^{\mc X^7,R}$ to $P_7$ is
  \begin{align}\label{eqn:IQ7restrict}
    I^{\mathcal{X}_7, R}(\vec q)|_{P_7}=
    &\sum_{(\vec n^i)\in \mathrm{Eff}^7_T}(\mathrm{Irrel})
    \prod_{ I\neq J}
    \frac{\prod_{l\leq n_I^3-n_J^3}(\lambda_{I}-\lambda_{J}+l)}{\prod_{l\leq 0}(\lambda_{I}-\lambda_{J}+l)}
    \prod_{F=1}^{M_4}
    \prod_{I=1}^{M_3}
    \frac{\prod_{l\leq 0}(\lambda_I-\lambda_F+l)}
    {\prod_{l\leq n^3_I}(\lambda_I-\lambda_F+l)}\nonumber\\
    &
    \prod_{I=1}^{M_3}\prod_{i=1}^2 \prod_{A=1}^{M_i}\frac{\prod_{l\leq 0}(\lambda_I-\lambda_A+l)} {\prod_{l\leq n^3_I-n^i_A}(\lambda_I-\lambda_A+l)}
    \prod_{i=1}^3(q_i)^{|\vec n^i|}\,.
  \end{align}
Modify the formula by multiplying $\prod_{i=1}^2\prod_{A=1}^{M_i}\prod_{I=1}^{M_3}\frac{\prod_{l\leq -n^i_A}(\lambda_I-\lambda_A+l)}{\prod_{l\leq -n^i_A}(\lambda_I-\lambda_A+l)}$. Then,
 \begin{subequations}
    \begin{align}
    I^{\mathcal{X}_7, R}(\vec q)|_{P_7}=&\sum_{\substack{\vec n^i\in \mathbb Z^{M_i}\\ i=1,2}}(\mathrm{Irrel})
    \prod_{i=1}^2
    \prod_{A=1}^{M_i}
    \prod_{I=1}^{M_3}
    \frac{\prod_{l\leq 0}(\lambda_I-\lambda_A+l)} {\prod_{l\leq -n^i_A}(\lambda_I-\lambda_A+l)}(q_i)^{|\vec n^i|}
    \label{subeqn:IX61}\\
    &\sum_{\vec n^3\in \mathbb Z^{M_3}_{\geq 0}} q_3^{\abs{\vec n^3}}
    \prod_{ I\neq J}\frac{\prod_{l\leq n_I^3-n_J^3}(\lambda_{I}-\lambda_{J}+l)}{\prod_{l\leq 0}(\lambda_{I}-\lambda_{J}+l)}\prod_{F=1}^{M_4}\prod_{I=1}^{M_3}\frac{\prod_{l\leq 0}(\lambda_I-\lambda_F+l)}
    {\prod_{l\leq n^3_I}(\lambda_I-\lambda_F+l)}
    \label{subeqn:IX62}\\
    &\prod_{I=1}^{M_3}\prod_{i=1}^2 \prod_{A=1}^{M_i}
    \frac{\prod_{l\leq 0}(\lambda_I-\lambda_A-n^i_A+l)} {\prod_{l\leq n^3_I}(\lambda_I-\lambda_A-n^i_A+l)}
    \label{subeqn:IX63}\,.
  \end{align}
 \end{subequations}
 Notice that in \eqref{subeqn:IX61}
 it looks that the factor $\frac{\prod_{l\leq 0}(\lambda_I-\lambda_A+l)}{\prod_{l\leq -n^i_A}(\lambda_I-\lambda_A+l)}$ vanishes if $n^i_A>0$ for some $i$ and $A$. 
 However, there is a pole in factor $\frac{\prod_{l\leq 0}(\lambda_I-\lambda_A-n^i_A+l)}{\prod_{l\leq n^3_I}(\lambda_I-\lambda_A-n^i_A+l)}$ in  \eqref{subeqn:IX63}. Hence the whole formula is well defined. 
 Moreover, 
zfor fixed $\vec n^1,\vec n^2$ the index set of $\vec n^3$ can be enlarged to be $\mathbb Z^{M_3}_{\geq 0}$, since $I^{\mc X_7,R}\big|_{P_7}$ vanishes for $n^3_I<n^i_I$ because of the factor $\prod_{I=1}^{M_3}\prod_{A=1}^{M_i}\frac{\prod_{l\leq 0}(\lambda_I-\lambda_A+l)}{\prod_{l\leq n^3_I-n^i_A}(\lambda_I-\lambda_A+l)}$ in \eqref{eqn:IQ7restrict}.

Comparing  with \eqref{eqn:Ifungrbmres}, one can find that the formulas \eqref{subeqn:IX62} and \eqref{subeqn:IX63} together is the degree $\vec n^3$ term of the quasimap small $I$-function of the  Grassmannian $Gr(M_3,M_1+M_2+M_4)$ restricted to a torus fixed point parameterized by set $[M_3]\subset [M_1]\sqcup [M_2]\sqcup [M_4]$, 
and the equivariant parameters of the torus action of $(\mathbb C^*)^{M_1+M_2+M_4}$ on $Gr(M_3,M_1+M_2+M_4)$ are $n_A^i+\lambda_A$ for $i=1,2;A=1,\ldots,M_i$ and $\lambda_F$ for $F=1,\ldots,M_4$. 

On the other hand, we consider the restriction of $I^{\mc X_6,R}$ to $P_6\in \mathfrak F_6$.
Let $P_6:=\iota_6^{-1}(P_7)=([M_1],[M_2],[M_1]\sqcup [M_2]\sqcup ([M_4]\backslash [M_3]))$. 
The restrictions of Chern roots to $P_6$ are $x^i_A\big|_{P_7}=\lambda_A$, for $i=1,2$, $A=1,\ldots,M_i$. 
Define $\eta_I=x^3_I\big|_{P_6}$. 
Then
$\eta_I:=\lambda_I$ for $I=1,\ldots,M_1$, $\eta_I=\lambda_{I-M_1}$ for $I=M_1+1,\ldots,M_1+M_2$, 
and $\eta_I=\lambda_{I-(M_1+M_2)+M_3}$ for $I=M_1+M_2+1,\ldots,M_3$.
 The restriction of $I^{\mc X_6, R}$ to $P_6$ can be written as follows,
  \begin{align}\label{eqn:IQ6restriction1}
     I^{\mathcal{X}_6, R}(\vec q')|_{P_6}&
    =\sum_{(\vec m^i)\in \mathrm{Eff}_T^6}(\mathrm{Irrel})
    \prod_{I\neq J}\frac{\prod_{l\leq m^3_I-m^3_J}(\eta_{I}-\eta_{J}+l)}{\prod_{l\leq 0}(\eta_{I}-\eta_{J}+l)}
    \prod_{I=1}^{M_3}
    \prod_{F=1}^{M_4}
    \frac{\prod_{l\leq 0}(\lambda_F-\eta_I+l)}
    {\prod_{l\leq -m^3_I}(\lambda_F-\eta_I+l)}\nonumber
    \\
    &\prod_{I=1}^{M_3}
    \prod_{i=1}^2
    \prod_{A=1}^{M_i}
    \frac{\prod_{l\leq 0}(\lambda_A-\eta_I+l)}{\prod_{l\leq m^i_A-m^3_I}(\lambda_A-\eta_I+l)}\prod_{i=1}^3(q_i')^{\abs{\vec m^i}}.
  \end{align}
  Notice that the Irrelevant part in \eqref{eqn:IQ7restrict} and that in \eqref{eqn:IQ6restriction1} are exactly equal for fixed $\vec n^1=\vec m^1$ and $\vec n^2=\vec m^2$.
Let
\begin{align}\label{index change}
  d_I=
  \begin{cases}
      m^1_I-m^3_I, & 1\leq I\leq M_1\\
      m^2_{I-M_1}-m^3_{I}, & M_1+1\leq I\leq M_1+M_2 \\
       -m^3_{I-(M_1+M_2)+M_3}, & M_1+M_2+1\leq I\leq M_3
  \end{cases}
\end{align}
Then we must have $d_I\geq 0$. Otherwise the corresponding terms will vanish.
Substitute $m^{3}_I$ by $m^1_I-d_I$ for $I=1,\ldots, M_1$, $m^2_I-d_I$ for $I=M_1+1,\ldots,M_1+M_2$ and $-d_I$ for $I=M_1+M_2+1,\ldots,M_3$ by \eqref{index change}. 
Then we can rewrite $I^{\mathcal{X}_6, R}(\vec q')|_{P_6}$ as follows after some combinatorics,
\begin{subequations}
    \begin{align}
       I^{\mathcal{X}_6, R}(\vec q')|_{P_6} =&\sum_{\substack{\vec n^i\in \mathbb Z^{M_i}\\ i=1,2}} (\mathrm{Irrel}) \prod_{i=1}^2
    \prod_{A=1}^{M_i}
    \prod_{I=1}^{M_3}
    \frac{\prod_{l\leq 0}(\lambda_I-\lambda_A+l)}{\prod_{l\leq -m^i_A}(\lambda_I-\lambda_A+l)}
    \label{subeqn:IX71}\\
    &\cdot\sum_{\vec d\in \mathbb Z^{M_3}_{\geq 0}}
    \prod_{I\neq J}
    \frac{\prod_{l\leq d_J-d_I}(\xi_I-\xi_J+l)}{\prod_{l\leq 0}(\xi_I-\xi_J+l)}
    \prod_{I=1}^{M_3}
    \prod_{i=1}^2\prod_{A=1}^{M_i}
    \frac{\prod_{l\leq 0}(\lambda_A+m^i_A-\xi_I+l)}
    {\prod_{l\leq d_I}(\lambda_A+m^i_A-\xi_I+l)}\label{subeqn:IX72}\\    &\prod_{F=1}^{M_4}
    \frac{\prod_{l\leq 0}(\lambda_F-\xi_I+l)}{\prod_{l\leq d_I}(\lambda_F-\xi_I+l)}(q_3')^{-\abs{\vec d}}
    \prod_{i=1}^2(q_i'q_3')^{|\vec m^i|}\,.
    \label{subeqn:IX73}
    \end{align}
\end{subequations}
where 
\begin{align}\label{variable change}
  \xi_I=\left\{
     \begin{array}{ll}
       \lambda_I+m^1_I, & \hbox{$1\leq I\leq M_1$;}\\
      \lambda_{I-M_1}+m^2_{I-M_1}, & \hbox{$M_1+1\leq I\leq M_1+M_2$;} \\
       \lambda_{I-(M_1+M_2)+M_3}, & \hbox{$M_1+M_2+1\leq I\leq M_3$.}
     \end{array}
   \right.
\end{align}
The subequations \eqref{subeqn:IX72} and \eqref{subeqn:IX73} can be viewed as the restriction of the degree $\vec d$ term of equivariant quasimap small $I$-function \eqref{eqn:Ifungramres} of the dual of the Grassmannian $Gr(M_3,M_1+M_2+M_4)$ restricted to a torus fixed point $[M_1]\sqcup [M_2]\sqcup ([M_4]\backslash [M_3])$, and the equivariant parameters of torus $(\mathbb C^*)^{M_1+M_2+M_4}$-action on $Gr(M_3,M_1+M_2+M_4)$ are $\{\xi_I\}_{I\in [M_1]\sqcup M_2}\cup \{\lambda_F\}_{F\in [M_4]}$.

Compare $\op{Eff}_{T}^6$ and $\op{Eff}_T^7$, and one can find both $(\vec n^1,\vec n^2)$  and $(\vec m^1,\vec m^2)$ range in $\mathbb Z^{M_1}\times \mathbb Z^{M_2}$, 
and for any fixed $\vec n^1=\vec m^1, \vec n^2=\vec m^2$ the $\vec n^3\in \mathbb Z^{M_3}_{\geq 0}$ is related with $\vec m^3\in \mathbb Z^{M_3}$ by matching the $\abs{\vec n^3}$ and  $|\vec d|$.
Comparing $I^{\mc X_6,R}\big|_{P_6}$ and $I^{\mc X_7,R}\big|_{P_7}$, we find \eqref{subeqn:IX61} and \eqref{subeqn:IX71} are equal, and \eqref{subeqn:IX62} \eqref{subeqn:IX63} and \eqref{subeqn:IX72} \eqref{subeqn:IX73} are equal by the variable change $q_3=(q_3')^{-1}$ and $q_1=q_1'q_3',\,q_2=q_2'q_3'$ according to the fundamental building block in Theorem \ref{thm:Haidong}.
Similarly, we can prove the equation for any pair of torus fixed points $(P_6,\iota_6(P_6))\in \mathfrak F_6\times \mathfrak F_7$.
\end{proof}
Then the Theorem \ref{thm:2nd} item (5) is proved.
\subsection{Proof for equivalences among $I^{\mc Z_2,R},I^{\mc Z_{10},R},I^{\mc X_{8},R}$, and $I^{\mc X_{11},R}$}\label{sec:proofI10related}
Let $P_2\in \mathfrak F_2$ and $P_{10}:=\iota_{10}(P_2)\in \mathfrak F_{10}$ be a pair of torus fixed points. 
\begin{prop}\label{prop:I2I10}
    \begin{align}\label{eqn:I10I2}
        I^{\mc X_{10},R}(q_1,q_2,q_3)\big|_{P_{10}}=(1+(-1)^{N_3-N_1}q_3)^{\sum_{B=1}^{N_2}x^2_B-\sum_{A=1}^{N_2}x^1_A+N_3-N_1}I^{\mc Z_2,R}(q_1',q_2',q_3')\big|_{P_2}
    \end{align}
    under change of K\"ahler variables
    \begin{align}
        q_1'=\frac{q_1q_3}{1+(-1)^{N_3-N_1}q_3},q_2'=q_2(1+(-1)^{N_3-N_1}q_3),q_3'=q_3^{-1}.
    \end{align}
\end{prop}
\begin{proof}
We first proof the equation \eqref{eqn:I10I2}. The $I$-effective classes for $\mc Z_{10}$ are 
\begin{align}\label{eqn:effectiveclassesZ10}
    \mathrm{Eff}_T^{10}=&\{(\vec n^1,\vec n^2,\vec n^3)\in\mathbb Z^{N_2}_{\leq 0}\times \mathbb Z^{N_2}_{\geq 0}\times \mathbb Z^{N_3-N_1} \big| \nonumber\\
    &\forall I\in [N_3-N_1], \exists \text{ distinct integers } a_I\in [N_2],\,\mathrm{s.t.} n^3_I-n^1_{a_I}\geq 0\}\,,
\end{align} 
and those of $\mc Z_2$ are given in Equation \eqref{eqn:Effective2}.
Let $P_2\in \mathfrak F_2$ be a torus fixed point parameterized by $P_2=([N_2],\{N_1+1,\ldots,N_1+N_2\}, [N_3'])$.
Then the restriction of $I^{\mc Z_2,R}$ to $P_2$ can be written as 
\begin{align}
    I^{\mc Z_2,S}\big|_{P_2}=&\sum_{(\vec n^i)\in \mathrm{Eff}_{T}^2}(\mathrm{Irrel})\cdot
    \prod_{I\neq J}^{N_3'}
    \frac{\prod_{l\leq n^3_I-n^3_J}(\lambda_I-\lambda_J+l)}
    {\prod_{l\leq 0}(\lambda_I-\lambda_J+l)}
    \prod_{I=1}^{N_3'}\prod_{A=1}^{N_2}
    \frac{\prod_{l\leq 0}(\lambda_A-\lambda_I+l)}
    {\prod_{l\leq n^1_A-n^3_I}(\lambda_A-\lambda_I+l)}\nonumber
    \\
    &
    \prod_{I=1}^{N_3'}\prod_{B=1}^{N_2}
    \frac{\prod_{l\leq n^2_B-n^3_I}(\lambda_{N_1+A}-\lambda_I+l)}
    {\prod_{l\leq 0}(\lambda_{N_1+A}-\lambda_I+l)}(q_1^{\prime})^{\abs{\vec n^1}}(q_3^{\prime})^{\abs{\vec n^3}}\,.
\end{align}
Note that $d_I:=n^1_I-n^3_I\geq 0$ for each $I=1,\ldots, N_3'$. 
Replacing $n^3_I$ by $n^1_I-d_I$, we transform the $I^{\mc Z_2,S}\big|_{P_2}$ to the following formula by similar combinatorics with that in the proof of star-shaped quivers, 
\begin{subequations}
    \begin{align}
    &\sum_{\substack{\vec n^1\in \mathbb Z^{N_2}_{\leq 0}\\ \vec n^2\in \mathbb Z^{N_2}_{\geq 0}}}(\mathrm{Irrel})\cdot
    \prod_{I=1}^{N_3'}
    \left(
    \prod_{A=N_3'+1}^{N_2}
    \frac{\prod_{l\leq 0}(\lambda_A-\lambda_I+l)}
    {\prod_{l\leq n^1_A-n^1_I}(\lambda_A-\lambda_I+l)}
    \prod_{B=1}^{N_2}
    \frac{\prod_{l\leq n^2_A-n^1_I}(\lambda_{A+N_1}-\lambda_I+l)}{\prod_{l\leq 0}(\lambda_{A+N_1}-\lambda_I+l)}
    \right)\label{subeqn:I2I101}
    \\
    &
   \sum_{\vec d\in \mathbb Z^{N_3'}_{\geq 0}} \prod_{I\neq J}^{N_3'}
    \frac{\prod_{l\leq d_I-d_J}(\eta_I-\eta_J+l)}{\prod_{l\leq 0}(\eta_I-\eta_J+l)}
    \prod_{I=1}^{N_3'}
    \frac{\prod_{B=1}^{N_2}\prod_{l=1}^{d_I}(\lambda_{B+N_1}+n^2_B-\eta_I+l)}
    {\prod_{A=1}^{N_2}\prod_{l=1}^{d_I}(\eta_A-\eta_I+l)}(q_1^{\prime})^{\abs {\vec n^1}}(q_3^{\prime})^{-\abs{\vec d}+\sum_{I=1}^{N_3'}n^1_I}\label{subeqn:I2I102}
\end{align}
\end{subequations}
where we have used a new notation 
$$\eta_I=\lambda_I+n^1_I, \,I=1,\ldots,N_2.$$
Note that subequation \eqref{subeqn:I2I102} is the restriction of the equivariant quasimap small $I$-function of the dual side $Gr^\vee$ in the fundamental building block to the fixed point $([N_3']\subset[N_2])$ with the equivariant parameters of the torus action $(\mathbb C^*)^{N_2}\times (\mathbb C^*)^{N_2}$  being $\eta_I$ for $I=1,\ldots,N_2$ and $\lambda_{B+N_1}+n^2_B$ for $B=1,\ldots,N_2$.

On the other hand, according to Lemma \ref{fixed point for Q10}, let $P_{10}=\iota_{10}(P_2)$ be the torus fixed point in $\mathfrak F_{10}$ defined by 
\begin{equation}\label{eqn:P10}
    P_{10}=([N_2],\{N_1+B\}_{B=1}^{N_2},  
    \{N_3'+I\}_{I=1}^{N_3-N_1}).
\end{equation}
The restriction of $I^{\mc Z_{10},R}$ to $P_{10}$ is 
\begin{align}\label{eqn:I10restrictstop}
    I^{\mc Z_{10},R}\big|_{P_{10}}=&
    \sum_{\substack{\vec n^1\in \mathbb Z^{N_2}_{\leq 0}\\ \vec n^2\in \mathbb Z^{N_2}_{\geq 0} }}
    (\mathrm{Irrel})\sum_{\vec n^3\in \mathbb Z^{N_3}_{\leq 0}}
    \prod_{I\neq J}^{N_3-N_1}
    \frac{\prod_{l\leq n^3_I-n^3_J}(\lambda_{N_3'+I}-\lambda_{N_3'+J}+l)}{\prod_{l\leq 0}(\lambda_{N_3'+I}-\lambda_{N_3'+J}+l)}\nonumber
    \\
    &
    \prod_{I=1}^{N_3-N_1}
    \left(\prod_{A=1}^{N_2}
    \frac{\prod_{l\leq 0}(\lambda_{N_3'+I}-\lambda_A+l)}
    {\prod_{l\leq n^3_I-n^1_A}(\lambda_{N_3'+I}-\lambda_A+l)}
    \prod_{B=1}^{N_2}
    \frac{\prod_{l\leq 0}(\lambda_{B+N_1}-\lambda_{N_3'+I}+l)}
    {\prod_{l\leq n^2_B-n^3_I}(\lambda_{B+N_1}-\lambda_{N_3'+I}+l)}
    \right)\nonumber
    \\
    &
    \prod_{A=1}^{N_2}\prod_{B=1}^{N_2}
    \frac{\prod_{l\leq n^2_B-n^1_A}(\lambda_{B+N_1}-\lambda_A+l)}{\prod_{l\leq 0}(\lambda_{B+N_1}-\lambda_A+l)} q_1^{\abs{\vec n^1}}q_3^{\abs{\vec n^3}}.
\end{align}
 We use the similar strategy as what we have done to $I^{\mc Z_2,R}\big|_{P_2}$. 
 For each $I=1,\ldots,N_3-N_1$, we have $d_I:=n^3_I-n^1_{N_3'+I}\geq 0$. Replace $n^3_I$ by $d_I+n^1_{N_3'+I}$ everywhere and we get the following formula after some combinatorics
 \begin{subequations}
     \begin{align}
    & \sum_{\substack{\vec n^1\in \mathbb Z^{N_2}_{\leq 0}\\ \vec n^2\in \mathbb Z^{N_2}_{\geq 0} }}
    (\mathrm{Irrel})
    \prod_{I=1}^{N_3'}
    \left(
    \prod_{J=N_3'+1}^{N_2}
    \frac{\prod_{l\leq 0}(\lambda_{J}-\lambda_I+l)}
    {\prod_{l\leq n^1_{J}-n^1_I}(\lambda_{J}-\lambda_I+l)}
    \prod_{B=1}^{N_2}
    \frac{\prod_{l\leq n^2_B-n^1_I}(\lambda_{B+N_1}-\lambda_I+l)}
    {\prod_{l\leq 0}(\lambda_{B+N_1}-\lambda_I+l)}
    \right)\label{subeqn:I101}
    \\
    &
    \sum_{\vec d\in \mathbb Z^{N_3-N_1}_{\geq 0}}  
    \prod_{I=1}^{N_3-N_1}
    \frac{\prod_{B=1}^{N_2}\prod_{l=-d_I+1}^0(\lambda_{N_1+B}+n^2_B-\eta_{N_3'+I}+l)}
    {\prod_{A=1}^{N_2}\prod_{l=1}^{d_I}(\eta_{N_3'+I}-\eta_A+l)}\label{subeqn:I102}
    \\
    &\prod_{I\neq J}^{N_3-N_1}
    \frac{\prod_{l\leq d_I-d_J}(\eta_{N_3'+I}-\eta_{N_3'+J}+l)}
    {\prod_{l\leq 0}(\eta_{N_3'+I}-\eta_{N_3'+J}+l)} q_1^{\abs{\vec n^1}}q_3^{\abs{\vec d}+\sum_{I=1}^{N_3-N_1}n^1_{N_3'+I}}.\label{subeqn:I103}
     \end{align}
 \end{subequations}
 
Similar as what we have done in previous Propositions, we compare $\op{Eff}_T^2$ in \eqref{eqn:Effective2} and $\op{Eff}_T^{10}$ in \eqref{eqn:effectiveclassesZ10}, and find that $(\vec n^1, \vec n^2)$ for both sets range in the same $\mathbb Z^{N_2}_{\leq  0}\times \mathbb Z^{N_2}_{\geq 0}$ and $\vec n^3$ for both are related via the degree $|\vec d|$.

The sub-equation \eqref{subeqn:I102} is the degree $\abs{\vec d}$-term of the $I$-function of the total space of $S^{\oplus N_2}\rightarrow Gr(N_3-N_1,N_2)$ restricted to a torus fixed point $([N_2]\backslash [N_3']\subset [N_2])$ with equivariant parameters  
$\{\eta_I\}_{I=1}^{N_2}$ and  
$\{\lambda_{B+N_1}+n^2_B\}_{B=1}^{N_2}$. 

Compare $I^{\mc Z_2,R}\big|_{P_2}$ and $I^{\mc Z_{10},R}\big|_{\iota_{10}(P_2)}$. 
One can find subequations \eqref{subeqn:I2I101} and  \eqref{subeqn:I101} are exactly equal for the same $\vec n^{1}$ and $\vec n^2$. Subequations \eqref{subeqn:I2I102} and \eqref{subeqn:I102} are two sides of the fundamental building block in Theorem \ref{thm:Haidong} case 3. 
\end{proof}
We can conclude the Theorem \ref{thm:2nd} item (7) by this Proposition. \\

Next, we will discuss the relation of $I^{\mc Z_{10},R}$ and $I^{\mc X_8,R}$. Performing quiver mutation $\mu_1$ to $\mathbf Q_{10}$, we get $\mathbf Q_8$ by switching nodes $1$ and $2$, $2$ and $3$. Since we don't care the order of nodes, we will relabel nodes of $\mathbf Q_8$ by numerals $2,1,3$. In the following, when we talk about $\mathbf Q_8$, we mean this relabeled one.  
The torus fixed points of $\mc X^8$ can be described as follows. 
\begin{align}
    \mathfrak{F}_8=\{(\vec A_{[N_3]},\vec B_{[N_2]},\vec C_{[N_3-N_1]})\big|\, \vec B_{[N_2]}\subset \vec A_{[N_3]}, \vec C_{[N_3-N_1]}\subset \vec A_{[N_3]}\}
\end{align}
There is a bijection 
\begin{align}
    \iota_{10}': \mathfrak F_{10}\rightarrow \mathfrak F_8
\end{align}
defined by sending a point $(\vec A_{[N_2]}, \vec B_{[N_2]},\vec C_{[N_3-N_1]})$ to $(\vec C_{[N_3-N_1]}\sqcup ([N_4]\backslash \vec A_{[N_2]}), \vec B_{[N_2]},\vec C_{[N_3-N_1]})$.
Let $P_{10}\times \iota_{10}'(P_{10})\in \mathfrak F_{10}\times \mathfrak F_8$ be an arbitrary pair of torus fixed points.
 
 \begin{prop}
     We have
    \begin{align}
        I^{\mc X_8,R}(q_1,q_2,q_3)\big|_{P_{8}}=I^{\mc Z_{10},R}(q_1',q_2',q_3')\big|_{P_{10}},
    \end{align}
    with change of K\"ahler variables
    \begin{align}
        q_1'=q_1^{-1}, \,q_2'=q_2,\,q_3'=q_3q_1.\,
    \end{align}    
 \end{prop}
\begin{proof}

The effective classes of $\mc Z_{10}$ are given in
\eqref{eqn:effectiveclassesZ10}, and those of $\mc X_8$ are as follows
\begin{align}\label{eqn:effectivex8}
    \mathrm{Eff}_T^{8}=\{(\vec n^1,\vec n^2,\vec n^3)\in \mathbb Z^{N_3}_{\geq 0}\times\mathbb Z^{N_2}_{\geq 0}\times \mathbb Z^{N_3-N_1}   \big|\, 
    \forall I\in [N_3-N_1], \exists \text{ distinct }i_I, \mathrm{s.t.} n^1_{i_I}-n^3_I\geq 0.
    \}
\end{align}
We fix a $\vec n^3\in \mathbb Z^{N_3-N_1}$ with components negative or non-negative. 
Without loss of generality, we assume that for some $p$,
\begin{align}\label{eqn:n3}
    &n^3_i\geq 0, \text{ for }i=1,\ldots,p \nonumber\\
    &n^3_{j}< 0, \text{ for }j=p+1,\ldots,N_3-N_1. 
\end{align}

We choose $P_{10}$ as in \eqref{eqn:P10}. 
Then the image of the point $P_{10}$ in $\mathfrak F_8$ is 
\begin{equation}\label{eqn:torusfixedpointsP8p10}
    \iota_{10}'(P_{10})=(\{N_3'+1,\ldots,N_4\}, \{N_1+1,\ldots,N_4\}, \{N_3'+1,\ldots,N_2\}).
\end{equation}

We rewrite the restriction $I^{\mc Z_{10},R}\big|_{P_{10}}$ as follows, 
\begin{align}\label{eqn:I10restrP10}
    I^{\mc Z_{10},R}\big|_{P_{10}}(q')&=
    \sum_{(\vec n^i)\in \mathrm{Eff}_T^{10}}(\mathrm{Irrel})\cdot
     \prod_{I\neq J}^{N_2}
    \frac{\prod_{l\leq n^1_I-n^1_J}(\lambda_I-\lambda_J+l)}
    {\prod_{l\leq 0}(\lambda_I-\lambda_J+l)}
    \prod_{I=1}^{N_2}\prod_{B=1}^{N_2}
    \frac{\prod_{l\leq -n^1_I+n^2_B}(-\lambda_I+\lambda_{B+N_1}+l)}
    {\prod_{l\leq 0}(-\lambda_I+\lambda_{B+N_1}+l)}
    \nonumber\\
    &
    \prod_{I=1}^{N_2}\left(\prod_{F=1}^{N_4}
    \frac{\prod_{l\leq 0}(-\lambda_I+\lambda_F+l)}{\prod_{l\leq -n^1_I}(-\lambda_I+\lambda_F+l)}
    \prod_{A=1}^{N_3-N_1}
    \frac{\prod_{l\leq 0}(-\lambda_I+\lambda_{N_3'+A}+l)}
    {\prod_{l\leq -n^1_I+n^3_A}(-\lambda_I+\lambda_{N_3'+A}+l)}
    \right)(q_3')^{\abs{\vec n^3}}\nonumber\\
    &
    \prod_{B=1}^{N_2}\left(
    \prod_{F=1}^{N_4}
    \frac{\prod_{l\leq 0}(\lambda_{B+N_1}-\lambda_F+l)}
    {\prod_{l\leq n^2_B}(\lambda_{B+N_1}-\lambda_F+l)}
    \prod_{A=1}^{N_3-N_1}
    \frac{\prod_{l\leq 0}(\lambda_{B+N_1}-\lambda_{N_3'+A}+l)}
    {\prod_{l\leq n^2_B-n^3_A}(\lambda_{B+N_1}-\lambda_{N_3'+A}+l)}
    \right)(q_1')^{\abs{\vec n^1}}
\end{align}
Notice that the $\mathrm{Irrel}$ represents the remaining part in the restriction of $I$-function and it is different with that in \eqref{eqn:I10restrictstop}.
By observation, one can find that the second term in the second row makes $d_I:=-n^1_I+n^3_{I-N_3'}\geq 0$ for $N_3'+p+1\leq I\leq N_2$. 
Make the replacement 
\begin{equation}\label{eqn:dIforz10}
\begin{cases}
    n^1_I=-d_I+n^3_{I-N_3'} &\text{ for } N_3'+p+1\leq I\leq N_2 \\
     n^1_I=-d_I & \text{ for } I\leq N_3'+p 
\end{cases} \,, 
\end{equation}
 and we transform the summation over all $\vec n^1\in \mathbb Z^{N_2}_{\leq 0}$ of  $I^{\mc Z_{10},R}\big|_{P_{10}}(q')$ except for the \textit{Irrelevant} part to a beautiful formula by doing some combinatorics
\begin{align}\label{eqn:I10dualgrassmann}
    &\sum_{\vec d\in \mathbb Z^{N_2}_{\geq 0}}
    \prod_{I=1}^{N_2}
    \frac{\prod_{B=1}^{N_2}\prod_{l=1}^{d_I}(\lambda_{B+N_1}+n^2_B-\eta_I+l) (q_3')^{\abs{\vec n^3}}(q_1')^{-\abs{\vec d}+\sum_{A=p+1}^{N_3-N_1}n^3_A}
    }
    {\prod_{F=1}^{N_4}\prod_{l=1}^{d_I}(-\eta_I+\lambda_F+l)\prod_{A=1}^{N_3-N_1}\prod_{l=1}^{d_I}(-\eta_I+\lambda_{N_3'+A}+n^3_A+l)}\nonumber
    \\
    &
    \prod_{I\neq J}
    \frac{\prod_{l\leq -d_I+d_J}(\eta_I-\eta_J+l)}
    {\prod_{l\leq 0}(\eta_I-\eta_J+l)}\cdot f
\end{align}
where
\begin{align}\label{subeqn:I10I8f}
       f= &
    \prod_{I=1}^{N_2}
    \prod_{A=1}^{p}
    \frac{\prod_{l\leq 0}(\lambda_{A+N_3'}-\eta_{I}+l)}
    {\prod_{l\leq n^3_A}(\lambda_{A+N_3'}-\eta_{I}+l)}
    \prod_{I=N_3'+p+1}^{N_2}
    \prod_{A=N_3'+1}^{N_4}
    \frac{\prod_{l\leq 0}(\lambda_{A}-\lambda_I+l)}
    {\prod_{l\leq -n^3_{I-{N_3'}}}(\lambda_{A}-\lambda_I+l)}\nonumber
    \\
    &
    \prod_{B=1}^{N_2}
    \left(
    \prod_{F=N_3'+p+1}^{N_4}
    \frac{\prod_{l\leq 0}(\lambda_{N_1+B}-\lambda_F+l)}
    {\prod_{l\leq n^2_B}(\lambda_{N_1+B}-\lambda_F+l)}
    \prod_{A=1}^k
    \frac{\prod_{l\leq 0}(\lambda_{B+N_1}-\lambda_{N_3'+A+l})}
    {\prod_{l\leq n^2_B-n^3_A}(\lambda_{B+N_1}-\lambda_{N_3'+A+l})}
    \right)
\end{align}
and
\begin{equation}
\eta_I=
    \begin{cases}
        \lambda_I,\,&\text{ for } 1\leq I\leq N_3'+p\\
        \lambda_I+n^3_{I-N_3'},\,&\text{ for } N_3'+p+1\leq I\leq N_2.
    \end{cases}
\end{equation}
Notice that the formula \eqref{eqn:I10dualgrassmann} without $f$ is the quasimap small $I$-function of $Gr^\vee$ in the fundamental building block restricted to torus fixed points 
$\{1,\ldots,N_3'+p\}\sqcup \{p+1,\ldots,N_3-N_1\}\subset [N_4]\sqcup [N_3-N_1]$ where equivariant parameters of torus $(\mathbb C^*)^{N_4+N_3-N_1}$-action are $\{\lambda_{F}\}_{F=1}^{N_4}\cup \{\lambda_{N_3'+A}+n^3_A\}_{A=1}^{N_3-N_1}$ and those of $(\mathbb C^*)^{N_2}$-action are $\{\lambda_{N_1+B}+n^2_B\}_{B=1}^{N_2}$.

Similarly, we restrict the quasimap small $I$-function of $\mc X_8$ to the torus fixed point in \eqref{eqn:torusfixedpointsP8p10}. 
\begin{align}\label{eqn:I8restrP8}
   I^{\mc X_8,R}\big|_{P_8}= &\sum_{(\vec n^i)\in \mathrm{Eff}_T^8}(\mathrm{Irrel})\cdot 
    \prod_{I\neq J}^{N_3}
    \frac{\prod_{l\leq n^1_I-n^1_J}(\lambda_{N_3'+I}-\lambda_{N_3'+J}+l)}{\prod_{l\leq 0}(\lambda_{N_3'+I}-\lambda_{N_3'+J}+l)}
    \prod_{I=1}^{N_3}\prod_{F=1}^{N_4}
    \frac{\prod_{l\leq 0}(\lambda_{N_3'+I}-\lambda_F+l)}
    {\prod_{l\leq n^1_I}(\lambda_{N_3'+I}-\lambda_F+l)}\nonumber\\
    &
    \prod_{I=1}^{N_3}\left(
    \prod_{A=1}^{N_3-N_1}
    \frac{\prod_{l\leq 0}(\lambda_{N_3'+I}-\lambda_{N_3'+A}+l)}
    {\prod_{l\leq n^1_I-n^3_A}(\lambda_{N_3'+I}-\lambda_{N_3'+A}+l)}
    \prod_{B=1}^{N_2}
    \frac{\prod_{l\leq 0}(\lambda_{N_1+B}-\lambda_{N_3'+I}+l)}
    {\prod_{l\leq n^2_B-n^1_I}(\lambda_{N_1+B}-\lambda_{N_3'+I}+l)}
    \right)\prod_{i=1}^3q_i^{\abs{\vec n^i}}.
\end{align}
The \textit{Irrelevant} parts in \eqref{eqn:I8restrP8} and \eqref{eqn:I10restrP10} are the same. 
For the same $\vec n^3$ described in \eqref{eqn:n3}, we have 
\begin{align}\label{eqn:dInIX8}
    &d_I:=n^1_I-n^3_I\geq 0, \text{ for } I=1,\ldots,p,\nonumber\\
    &d_I:=n^1_I,\,\text{ for } I=p+1,\ldots,N_3.
\end{align}
Otherwise the first product in the second row is zero and hence the corresponding term vanishes. 
Replace $n^1_I$ by $d_I+n^3_I$ or $d^I$ by the relation \eqref{eqn:dInIX8}.  
For the fixed $\vec n^3$ as in \eqref{eqn:n3},
we sum over all $\vec n^1$ terms in $I^{\mc X_8,R}\big|_{P_8}$ by disregarding the \textit{Irrelevant} part.
By doing some combinatorics, we transform this summation to the following formula
\begin{align}\label{eqn:I8grassmann}
        \sum_{\vec d\in \mathbb Z^{N_3}_{\geq 0}}
        \prod_{I\neq J}
        \frac{\prod_{l\leq d_I-d_J}(\zeta_I-\zeta_J+l)}
        {\prod_{l\leq 0}(\zeta_I-\zeta_J+l)}
        \prod_{I=1}^{N_3}
        \frac{\prod_{B=1}^{N_2}\prod_{l=-d_I+1}^0(\lambda_{N_1+B}+n^2_B-\zeta_I+l)q_1^{\abs{\vec d}+\sum_{I=1}^{n^3_I}}q_3^{\abs{\vec n^3}}}
        {\prod_{l=1}^{d_I-1}\prod_{F=1}^{N_4}(\zeta_I-\lambda_F+l)\prod_{A=1}^{N_3-N_1}(\zeta_I-\lambda_{N_3'+A}-n^3_A+l)}\cdot f
\end{align}
with the same $f$ in \eqref{subeqn:I10I8f} for fixed $\vec n^2,\vec n^3$.
The $\zeta_I$ in the above formula are 
\begin{equation}
\zeta_I=
    \begin{cases}
        \lambda_{N_3'+I}+n^3_I,\,&\text{ for } 1\leq I\leq p,\\
        \lambda_{N_3'+I},\, &\text{ for } p+1\leq I\leq N_3.
    \end{cases}
\end{equation}

Compare $\op{Eff}_{T}^{10}$ in \eqref{eqn:effectiveclassesZ10} and $\op{Eff}_{T}^8$ in \eqref{eqn:effectivex8} and one can find that $(\vec n^2,\vec n^3)$ range in $\mathbb Z^{N_2}_{\geq 0}\times \mathbb Z^{N_3-N_1}$ and $\vec n^1\in \mathbb Z^{N_2}$ are related via the their shifting in \eqref{eqn:dIforz10} and \eqref{eqn:dInIX8}.

Notice that the formula above without $f$ is exactly the quasimap small $I$-function of the total space $S^{\oplus N_2}\rightarrow Gr(N_3,N_4+N_3-N_1)$ restricted to torus fixed point $\{N_3'+p+1,N_3'+p+2,\ldots,N_4\}\sqcup \{1,2,\ldots, p\}\subset [N_4]\sqcup [N_3-N_1]$, 
where the equivariant parameters of torus $(\mathbb C^*)^{N_4+N_3-N_1}$ are 
$\{\lambda_F\}_{F=1}^{N_4}
\cup \{\lambda_{N_3'+A}+n^3_A\}$, 
and the equivariant parameters of the torus $(\mathbb C^*)^{N_2}$ on the fiber bundles are 
$\{\lambda_{N_1+B}+n^2_B\}_{B=1}^{N_2}$.
Since $N_4+N_3-N_1>N_2+1$, we can get the relation of formulas in \eqref{eqn:I10dualgrassmann} and \eqref{eqn:I8grassmann} by Theorem \ref{thm:Haidong}. 
Hence we have obtained the relation between $I^{\mc Z_{10},R}\big|_{P_{10}}$ and $I^{\mc X_8,R}\big|_{P_8}$, and have proved the Proposition. 
\end{proof}

The equivariant quasimap small $I$-functions of $\mc Z_{10}$ and $\mc X_{11}$ are related in a similar way. Hence we have concluded the Theorem \ref{thm:2nd} item (8) by localization. 

\section*{Acknowledgment}
The authors are grateful to Prof. Yongbin Ruan for proposing such an exciting topic. The first author would like to thank Yaoxiong Wen and Hai Dong for many helpful discussions on quiver varieties. The second author would like to thank Prof. Aaron Pixton for helpful conversations, and Peng Zhao for giving his physicist's suggestions and insights. Most of the work was done while the second author was at the University of Michigan, and she is thankful for the wonderful environment on the campus and great assistance from the department. He is supported by NSFC grant 12422104, and Zhang by NSFC 12301080.

\appendix
\addcontentsline{toc}{section}{Appendices}
\section{Computation for semistable loci of quivers that are mutation equivalent to the $D_3$-quiver}\label{sec:Appsemistablelocus}
In this section, we will investigate the semistable loci of the quivers $\mathbf Q_4, \ldots,\mathbf Q_8$ and $\mathbf Q_{10}$  listed in Figure \ref{fig:D4-6mutations} $(a)$ to $(e)$ and in Figure \ref{fig:extrafig2}. 
The main tool we use here is the Hilbert-Munford criterion \cite{GIT:Mumford}.  

For each quiver $\mathbf Q_i$, we let the input data be $(V_i,G_i,\theta_i)$, where $\theta_i(g)=\prod_{i=1}^3\det(g_i)^{\sigma_i}$ is the character of the gauge group $G_i$. 
For an one-dimensional subgroup $g(\lambda)$, $\lambda\in \mathbb C$, we define $<\theta_i, g(\lambda)>$ to be the exponent of $\lambda$ in $\theta_i(g(\lambda))$.

We adopt notations in the Section \ref{quivermutation}, and use the letters $A_i$ to represent arrows as in Figure \ref{fig:D4-6mutations}.

We will find the semistable loci in the proposed phases in Table \ref{table:phases}.

\subsection{Quiver $\mathbf Q_4$}\label{sec:appQ4}
\subsubsection{Semistable locus}
\begin{lem}\label{ss locus for X4}
 For the quiver $\mathbf Q_4$, in the phase $\sigma_3>0, \,\sigma_1+\sigma_3<0, \,\sigma_2+\sigma_3<0$, we have
 \begin{equation}\label{eqn:Q4semistab}
     V^{ss}_{4,\theta_4}(G_4)=\{(A_1, A_2, A_3)\big|A_1,A_2,\,\begin{bmatrix}A_1&A_2\end{bmatrix} ,A_3A_1 ,A_3A_2 \quad \text{all non-degenerate}\}.
 \end{equation}
\end{lem}
\begin{proof}
   We first prove that $V^{ss}_{4,\theta_4}$ is contained in the set of the right hand side in \eqref{eqn:Q4semistab}.
   It is easy to find that a point $(A_1, A_2, A_3)$ is unstable if  $A_1$, $A_2$ or $\begin{bmatrix}A_1&A_2\end{bmatrix}$ is degenerate, since $N_i<N_3$ $N_3<N_1+N_2$ and $\sigma_i<0$ for $i=1, 2$, $\sigma_3>0$.
   In the following argument, we will assume $A_i$, $i=1, 2$ and $\begin{bmatrix}A_1&A_2\end{bmatrix}$ are non-degenerate.

   We claim that $A_3A_1$ and $A_3A_2$ are both non-degenerate. Otherwise, if $A_3A_1$ is degenerate, then in the $G_4$-orbit we can find a representative such that 
   \begin{itemize}
      \item  the first column of $A_3A_1$ is zero,
      \item and  $A_1=\begin{bmatrix} 1&\mathbf 0\\ \mathbf 0&*\end{bmatrix}$, $A_3=\begin{bmatrix} \mathbf 0&*\end{bmatrix}$.
   \end{itemize}
   We then choose an one-parameter subgroup $g(\lambda)$ of $G_4$  by letting $g_1=\mathrm{diag}\{\lambda, 1, \ldots, 1\}$, $g_2=\mathrm{Id}_{N_1}$, $g_3=\mathrm{diag}\{\lambda, 1, \ldots, 1\}$. 
   Then one can check that $\lim_{\lambda\rightarrow 0}g(\lambda)\cdot (A_1,A_2,A_3)$ exists and  $\theta(g)=\lambda^{\sigma_1+\sigma_3}=\lambda^{<0}$, which contradicts the stability of $(A_1,A_2,A_3)$. We similarly can prove that $A_3A_2$ is nondegenerate.

   On the other hand, suppose that a point $(A_1,A_2,A_3)$ belongs to the set of the right hand side of \eqref{eqn:Q4semistab}, we are going to prove that it is semistable. 
  Let $g(\lambda)\subset G_4$ be an arbitrary one-parameter subgroup with $g_1=\mathrm{diag}(\lambda^{a_1}, \lambda^{a_2} , \ldots, \lambda^{a_{N_2}})$, $g_1=\mathrm{diag}(\lambda^{b_1}, \lambda^{b_2} , \ldots, \lambda^{b_{N_1}})$, such that $g_3=\mathrm{diag}(\lambda^{c_1}, \cdots,\lambda^{c_{N_3}})$ such that $\lim_{\lambda\rightarrow 0}g(\lambda)\cdot (A_1,A_2,A_3)$ exists. 
  Since $A3A_1$ and $A_3A_2$ are nondegenerate, we have $a_i,b_j<0$ for $i=1,\ldots,N_2, j=1,\ldots,N_1$.
  Since $\begin{bmatrix}A_1&A_2\end{bmatrix}$ is non-degenerate, 
  for each $k\in \{1,\ldots,N_3\}$, there is a $i_k\in \{1,\ldots,N_2\}$ such that $c_k-a_{i_k}\geq 0$ or there is a $j_k\in \{1,\ldots,N_1\}$ such that $c_k-b_{j_k}\geq 0$.
  Without loss of generality, we assume that for $k=1,\ldots,l$, $c_k\geq a_{i_k}$, and for $k=l+1,\ldots,N_3$, $c_k\geq b_{j_k}$.
  Then 
  \begin{align}
      <\theta_4, g(\lambda)>=
      &\sigma_1(\sum_{i=1}^{N_2}a_i)+ \sigma_2(\sum_{j=1}^{N_1}b_j)+
      \sigma_3(\sum_{i=1}^{N_3}c_i) \nonumber\\
      &\geq \sigma_1(\sum_{i=1}^{N_2}a_i)+ \sigma_2(\sum_{j=1}^{N_1}b_j)+
      \sigma_3(\sum_{k=1}^la_{i_k}+ \sum_{k=l+1}^{N_3}b_{j_k})\nonumber\\
      &\geq (\sigma_1+\sigma_3)(\sum_{k=1}^la_{i_k})+
      (\sigma_2+\sigma_3)(\sum_{k=l+1}^{N_3}b_{j_k})\geq 0.
  \end{align}
\end{proof}
\subsubsection{Torus fixed points}
The follow lemma gives the $R$-fixed locus in $\mc X_4$.
\begin{lem}\label{fixed pt for X4}
  The $R$-fixed locus of $\mathcal{X}_4$ is parameterized by the following finite set
    \begin{align}
   \mathfrak F_4 =\{(\vec A_{[N_2]},\vec B_{[N_1]},\vec C_{[N_4-N_3]})\big| \,\vec C_{[N_4-N_3]}\subset \vec A_{[N_2]}\cap\vec B_{[N_1]},\, \vec A_{[N_2]},\vec B_{[N_1]} \subset [N_4] \}
\end{align}
An element $(\vec A_{[N_2},\vec B_{[N_1]},\vec C_{[N_4-N_3]})\in \mathfrak{F}_4$ represents a $G_4$-orbit of the following form $(A_1, A_2, A_3)$.

Let $\vec D_{[M]}:=\vec A_{[N_2]}\cup \vec B_{[N_1]}$, $\vec E_{[m]}:=(\vec A_{[N_2]}\cap \vec B_{[N_1]})\backslash \vec C_{[N_4-N_3]}$. Define a map $$d: \vec D_{[M]}=\{l_1<l_2<\cdots < l_{M}\}\rightarrow \{1,\ldots ,M\}$$  sending $l_k$ to $k$.

We let $A_3$ and $A_1$ be column reduced echelon forms with row numbers of pivots being $\vec D_{[M]}$ and $d(\vec A_{[N_2]})$ respectively. 
Denote by $\vec e_i$ a row vector with $i$-th component 1 and all others zero. 

To obtain $A_2$, we first let $A$ be a reduced column echelon form such that $d(\vec B_{[N_1]})$ are the sets of row numbers of pivots. Notice $N_3-|\vec A_{[N_2]}\cup \vec B_{[N_1]}|=|\vec A_{[N_2]}\cap \vec B_{[N_1]}|-(N_4-N_3)=m$, so the last $m$ rows of $A$ are zero. We then replace the last $m$ zero rows with $\vec e_k$'s with $k\in d(\vec E_{[m]})$.

\end{lem}
The proof of the above lemma is elementary and tedious, we omit it here. We will illustrate the idea of proof via the following example.
\begin{ex}
  Suppose $N_1=N_2=2$, $N_3=3$, $N_4=4$. We consider $(\vec A, \vec B, \vec C)=(\{12\}, \{13\}, \{1\})$, then $\vec D=\{123\}$, $\vec E=\emptyset$, the fixed points $(A_1, A_2, A_3)$ is
  \begin{align*}
  A_1=
  \begin{bmatrix}
  1&0\\ 
  0&1\\
  0&0
  \end{bmatrix}, \,
  A_2=
  \begin{bmatrix}
  1&0\\ 
  0&0\\
  0&1
  \end{bmatrix},\,
  A_3=
  \begin{bmatrix}
  1&0&0\\ 
  0&1&0\\
  0&0&1\\
  0&0&0
  \end{bmatrix}
\end{align*}
If $(\vec A, \vec B, \vec C)=(\{12\}, \{12\}, \{1\})$, then $\vec D=\{12\}$, $\vec E=\{1\}$, the fixed points $(A_1, A_2, A_3)$ is
  \begin{align*}
  A_1=\begin{bmatrix}
  1&0\\ 
  0&1\\
  0&0
  \end{bmatrix}, \,
  A_2=
  \begin{bmatrix}
  1&0\\ 
  0&1\\
  0&1
  \end{bmatrix},\,
  A_3=
  \begin{bmatrix}
  1&0&0\\ 
  0&1&0\\
  0&0&0\\
  0&0&0
  \end{bmatrix}
\end{align*}
Next we compute all the $R$-fixed locus in this case. 
Notice that a general $A_3$ should be $G_4$-equivalent to an $A_3$ which is a reduced column echelon form. If it is fixed by $R$, then there is at most one nonzero component in each column of $A_3$. 
Furthermore, from the proof of Lemma \ref{ss locus for X4}, we know there is at most $N_3-N_2=1$ zero column in $A_3$, 
so $A_3$ should be of following two types
\begin{align*}
 (1)\, 
 \begin{bmatrix}
 1&0&0\\ 
 0&1&0\\ 
 0&0&1\\ 
 0&0&0
 \end{bmatrix}
 \quad (2)\,
 \begin{bmatrix} 
 1&0&0\\ 
 0&1&0\\ 
 0&0&0\\ 
 0&0&0
 \end{bmatrix}
\end{align*}

For case (1), since $(A_1, A_2, A_3)$ is $R$-fixed, there is at most one nonzero component in each column of $A_1$ and $A_2$, and they will be $G_4$-equivalent to reduced column echelon forms. 
Furthermore, notice that $\begin{bmatrix}
A_1&A_2
\end{bmatrix}$ 
are non-degenerate. 
There are $C_{4}^3=4$ choices for $A_3$. 
For each $A_3$, there are $C_3^2\times 2=6$ choices for $(A_1, A_2)$. 
So there are 24 fixed points of type (1).

For case (2), since $A_3A_1$ and $A_3A_2$ are both non-degenerate, $A_1$ and $A_2$ are of the following form
\begin{align*}
 A_1=\begin{bmatrix}
 1&0\\ 
 0&1\\ 
 a&b
 \end{bmatrix}\quad A_2=
 \begin{bmatrix} 
 1&0\\ 
 0&1\\ 
 c&d\end{bmatrix}
\end{align*}
Again by $R$-fixed condition, there is at most one nonzero element among $a, b, c, d$. On the other hand, by non-degeneracy of $\begin{bmatrix}A_1&A_2\end{bmatrix}$, $a, b, c, d$ cannot all vanish. Finally, notice that the case $a=1$ and $c=1$ is $G_4$-equivalent, the same as $b$ and $d$. Hence, after fixing $A_3$, there are two choices of $(A_1, A_2)$. 
Then there are in total $C_4^2\times 2=12$ $R$-fixed points in case (2).

In conclusion, there are $24+12=36$ $R$-fixed points.
\end{ex}
\begin{cor}\label{cor:appZ3x4torusfixedpoints}
  There is a natural one-to-one correspondence between the fixed loci $\mathfrak F_3$ and $\mathfrak F_4$.
\end{cor}
\begin{proof}
  Since the fixed points set of $\mathcal{X}_3$ and $\mathcal{X}_4$ are both parameterized by the same set.
\end{proof}

\subsection{Quiver $\mathbf Q_5$}
\subsubsection{Semistable locus}
\begin{lem}\label{lem:Q5phase2nd}
    In the proposed phase
    \begin{align}\label{eqn:Q5phase2nd}
        \sigma_3<0,\,
        \sigma_1+\sigma_3>0,\,
        \sigma_1+\sigma_2+\sigma_3<0, 
    \end{align}
 the semistable locus is 
 \begin{align}\label{eqn:Q5sslocu2ndphase}
        V^{ss}_{5,\theta_5}=\{(A_1,A_2,A_3)\big| A_1,A_2 \begin{bmatrix}A_1\\A_3 \end{bmatrix}, A_1A_2,A_3A_2 \text{ all non-degenerate}\}
    \end{align}
\end{lem}
\begin{proof}
    By $\sigma_1>0,\sigma_2<0,\sigma_3<0$, it is easy to see that $(A_1,A_2,A_3)$ is semistable only if $A_1,A_2$ and $\begin{bmatrix} A_1\\A_3\end{bmatrix}$ are non-degenerate
    
    Furthermore, we must have $A_3A_2$ non-degenerate. 
    Otherwise, 
    $A_3A_2$ can be transformed to $\begin{bmatrix}\mathbf 0&*\end{bmatrix}$ by the action of $GL(N_1)$. Since $A_2$ is column full rank, without loss of generality, we assume that the first column of $A_2$ is a column vector $\vec e_1$ whose first component is 1 and all other components are zero. Then the first column of $A_3$ must be a zero column vector. Since $\begin{bmatrix} A_1\\A_3\end{bmatrix}$ is nondegenerate, the first column of $A_1$ must be nonzero which can be transformed to $\vec e_1$ via the action of $GL(N_3-N_1)$. 
    Then we can choose an one-parameter subgroup $g(\lambda)$ of $G$ such that 
    \begin{align}
        g_1(\lambda)=\mathrm{diag}(\lambda, 1,\ldots,1)\,\,
        g_2(\lambda)=\mathrm{diag}(\lambda,1,\ldots,1)\,\,
        g_3(\lambda)=\mathrm{diag}(\lambda,1,\ldots,1).
    \end{align}
    One can check that $\lim_{\lambda\rightarrow 0}g(\lambda)\cdot(A_1,A_2,A_3)$ exists and $\theta_5(g)=\lambda^{\sigma_1+\sigma_2+\sigma_3<0}$. Hence degeneracy of $A_3A_2$ makes the element unstable. Therefore, we must have $A_3A_2$ non-degenerate. 

    Similarly, one can check that the multiplication $A_1A_2$ must be non-degenerate mimicking the above paragraph. 

    Until now, we have proved that in Equation \eqref{eqn:Q5sslocu2ndphase} the left hand side is contained in the right hand side. In the remaining part, we will prove the inclution in the other direction. 
    Assume $(A_1,A_2,A_3)$ is semistable. Let $(g(\lambda))$ be an one-parameter subgroup such that $g(\lambda)\cdot (A_1,A_2,A_3)$ exists. Then via the gauge group action, we can assume that 
    \begin{align}
        g_1(\lambda)=\mathrm{diag}(\lambda^{a_1},\ldots,\lambda^{a_{N_3-N_2}}),\,
        g_2(\lambda)=\mathrm{diag}(\lambda^{b_1},\ldots,\lambda^{b_{N_1}}),\,
        g_3(\lambda)=\mathrm{diag}(\lambda^{c_1},\ldots,\lambda^{c_{N_3}})
    \end{align}
    We conclude the following relation among those $a_i,b_i,c_i$.
    We have $b_i<0$ for all $i$ by the non-degeneracy of $A_3A_2$,  
    $ \forall i=1,\ldots,N_3-N_2, \exists j_i, s.t. a_i>b_{j_i}$ by the non-degeneracy of $A_1A_2$, 
    $\forall i=1,\ldots,N_3-N_2$, $\exists \,k_i\in \{1,\ldots,N_3\}$, s.t. $a_i>c_{k_i}$ and for the remaining $j'\in \{1,\ldots,N_3\}\backslash \{k_i\}$, $c_{j'}<0$.
    Then 
    \begin{align}\label{eqn:Q5thetag}
        <\theta, g(\lambda)>&=\sigma_1(\sum_{i=1}^{N_3-N_2}a_i)+\sigma_2(\sum_{i=1}^{N_1} b_i)+\sigma_3(\sum_{i=1}^{N_3}c_i)\nonumber\\
        &\geq \sigma_1(\sum_{i=1}^{N_3-N_2}a_i)+\sigma_2(\sum_{i=1}^{N_1}b_i)+\sigma_3(\sum_{i=1}^{N_3-N_2}a_i)+\sigma_3(\sum_{j'}c_j')\nonumber\\
        &\geq(\sigma_1+\sigma_3)(\sum_{i=1}^{N_3-N_2}a_i)+\sigma_2(\sum_{i=1}^{N_1}b_i)
        \nonumber\\
       & \geq (\sigma_1+\sigma_3)(\sum_{i=1}^{N_3-N_2}b_{j_i})+\sigma_2(\sum_{i=1}^{N_1}b_i)
       \nonumber\\
       &\geq (\sigma_1+\sigma_2+\sigma_3)(\sum_{i=1}^{N_3-N_2}b_{j_i})+ \sigma_2(\sum_{j\neq j_i}^{N_1}b_j)\geq 0
    \end{align}
Therefore, each element $(A_1,A_2,A_3)$ in the set of the right hand side in Equation \eqref{eqn:Q5sslocu2ndphase} is semistable. 
\end{proof}
\subsubsection{Torus fixed locus}
\begin{lem}\label{fixed pt for X5}
  The $S$-fixed locus of $\mathcal{X}_5$ is a disjoint union of isolated fixed points. The isolated fixed points can be parameterized by the following set
    \begin{align}
   \mathfrak F_5 =\{(\vec A_{[N_1]},\vec B_{[N_2]},\vec C_{[N_4-N_3]})\big| \,\vec C_{[N_4-N_3]}\subset \vec A_{[N_1]}\cap\vec B_{[N_2]},\, \vec A_{[N_1]},\vec B_{[N_2]} \subset [N_4] \}
\end{align}
An element $(\vec A_{[N_1]},\vec B_{[N_2]},\vec C_{[N_4-N_3]})\in \mathfrak{F}_5$ represents a $_5$-orbit of the following form $(A_1, A_2, A_3)$. Here $A_2, A_3$ are constructed in the same way as Lemma \ref{fixed pt for X4}, and  $A_1$ is row reduced echelon forms with row numbers of pivots being $\{1, 2, \ldots, N_3-N_2\} \backslash d(\vec B_{[N_2]})$ , where $d$ is defined in Lemma \ref{fixed pt for X4}.
\end{lem}

\begin{ex}
  Suppose $N_1=N_2=2$, $N_3=3$, $N_4=4$. We consider $(\vec A, \vec B, \vec C)=(\{12\}, \{13\}, \{1\})$, then $\vec D=\{123\}$, $d(\vec B)=\{13\}$, the fixed point $(A_1, A_2, A_3)$ is
  \begin{align*}
  A_1=
  \begin{bmatrix}
  0&1&0
  \end{bmatrix}, 
  A_2=
  \begin{bmatrix}
  1&0\\ 
  0&1\\
  0&0
  \end{bmatrix}, 
  A_3=\begin{bmatrix}
  1&0&0\\ 
  0&1&0\\
  0&0&1\\
  0&0&0
  \end{bmatrix}
\end{align*}
If $(\vec A, \vec B, \vec C)=(\{12\}, \{12\}, \{1\})$, then $\vec D=\{12\}$, $d(\vec B)=\{12\}$, the fixed point $(A_1, A_2, A_3)$ is
  \begin{align*}
  A_1=
  \begin{bmatrix}
  0&0&1
  \end{bmatrix}, 
  A_2=\begin{bmatrix}
  1&0\\ 
  0&1\\
  0&1
  \end{bmatrix}, 
  A_3=
  \begin{bmatrix}
  1&0&0\\ 
  0&1&0\\
  0&0&0\\
  0&0&0
  \end{bmatrix}
\end{align*}

Next we compute all $R$-fixed points. In $G_5$-orbit $A_3$ can be a reduced column echelon form. 
If it is fixed by $R$, then there is at most one nonzero component in each column of $A_3$. 
Furthermore, from the proof of Lemma \ref{ss locus for X4}, we know there is at most $N_3-N_2=1$ zero column in $A_3$, so $A_3$ should be one of the following two types
\begin{align*}
 (1)\, 
 \begin{bmatrix}1&0&0\\ 0&1&0\\ 0&0&1\\ 0&0&0\end{bmatrix}\quad (2)\,\begin{bmatrix} 1&0&0\\ 0&1& 0\\ 0& 0& 0\\ 0&0&0\end{bmatrix}
\end{align*}

For case (1), there is at most one nonzero component in each column (row) of $A_2$ ($A_1$) in an element of $G_5$-orbit, and they will be $G_5$-equivalent to reduced column (row) reduced echelon forms. 
By the non-degeneracy of $A_3A_2$, one can find there are $C_3^2$ choices of $A_2$. 
Once we have fixed $A_2$, there are two choices of $A_1$ since $A_1A_2$ is non-degenerate. 
Since $A_3$ is non-degenerate,  there are $C_{4}^3=4$ choices for $A_3$ by varying the positions of pivots.
For each $A_3$, there are $C_3^2\times 2=6$ choices for $(A_1, A_2)$. Therefore, there are 24 $R$-fixed points in case (1).

For case (2), when we perform a $R$-action on the above canonical form $A_3$, the $g_3$-action will force that there is at most one non-vanishing component in each column (row) of $A_2^\prime$ ($A_1^\prime$), where $A_2^\prime$ ($A_1^\prime$) is the submatrix obtained by the first 2 rows (columns) of $A_2$ ($A_1$). Since $A_3A_2$ is non-degenerate, $A_2$ is of the form $\begin{bmatrix} 1, 0\\ 0, 1\\ c, d\end{bmatrix}$. Since $\begin{bmatrix}A_1\\A_3 \end{bmatrix}$ is non-degenerate, $A_2$ is of the form $\begin{bmatrix} a, b ,1\end{bmatrix}$, with $ab=0$. WLOG, suppose $a=0$, if $b\neq 0$, then $c=0$ by $R$-action but this will make $A_1A_2$ degenerate. So $a=b=0$. Again by $R$-action and the non-degeneracy of $A_1A_2$, there is exactly 1 non-vanishing element in $c, d$. So after fixing $A_3$, $A_1$ is determined, and there are 2 choices of $A_2$. Then there are totally $C_4^2\times 2=12$ fixed points in case (2).

In conclusion, there are $24+12=36$ fixed points.
\end{ex}

\subsection{Quiver $\mathbf Q_6$} 
Let $M_1=N_4-N_2,M_2=N_4-N_1,M_3=N_3,M_4=N_4$.
\begin{lem}
    In the proposed phase,
    \begin{align}\label{eqn:Q6phase2nd}
        \sigma_1+\sigma_3<0,\,
        \sigma_2+\sigma_3<0,\, \sigma_1+\sigma_2+\sigma_3>0,
    \end{align}
   the semistable locus is 
    \begin{align}\label{eqn:Q6semistable}
        V^{ss}_{6,\theta_6}=\{(A_1,A_2,A_3)\big| A_1,A_2 \begin{bmatrix}A_1\\A_2 \end{bmatrix}, \begin{bmatrix}A_2\\A_2 \end{bmatrix}, \begin{bmatrix}A_1\\A_3 \end{bmatrix} \text{ non-degenerate}\}.
    \end{align}
\end{lem}

\begin{proof}
  Firstly, it is easy to find $(A_1, A_2, A_3)$ is unstable if $A_1$, $A_2$ or $\begin{bmatrix} A_1\\A_2\\A_3\end{bmatrix}$ is degenerate, since $\sigma_1>0,\sigma_2>0,\sigma_3<0$. 
  In the following argument, we will assume that the above three matrices are all nondegenerate. 

   We claim that if a point $(A_1,A_2,A_3)$ is semistable, then $\begin{bmatrix}A_1\\A_2 \end{bmatrix}$ must be  nondegenerate. 
   Otherwise, 
   at least one row vector of $A_1$ is the same with that of $A_2$, and we assume that the first row vector of $A_1$ is the same with the first row vector of $A_2$. Assume further that the first component of this vector is nonzero. Then, under $G_6$ action,  $$\begin{bmatrix}A_1\\A_2\end{bmatrix}=\begin{bmatrix}1&\mathbf 0\\ \mathbf 0&* \\ 1&\mathbf 0\\ \mathbf 0 &* \end{bmatrix}.$$
   Let $g(t)=(g_1(t),g_2(t),g_3(t))$, $t\in \mathbb C$ be an one-parameter subgroup of $G_6$ such that $g_1(t), g_2(t), g_3(t)$ are of the form $\mathrm{diag}(t^a, 1, \ldots, 1)$, $a<0$. One can find
   $\lim_{t\rightarrow 0}g(t)\cdot(A_1, A_2, A_3)$ exists and $\theta(g(t))=t^{a(\sigma_1+\sigma_2+\sigma_3)}=t^{<0}$. So $(A_1, A_2, A_3)$ is unstable.

  We claim that the augmented matrix $\begin{bmatrix}A_2\\A_3\end{bmatrix}$ is also non-degenerate. Otherwise,  via $G_6$-action, $\begin{bmatrix}A_1\\A_2\\A_3\end{bmatrix}$ can be transformed to the matrix whose first column is $(1, 0, 0, \ldots, 0)^T$. Let $g(t)$ be an one-parameter subgroup of $G_6$ such that $g_2=\mathrm{Id}$, 
  $g_1(t), g_3(t)$ are of the form $\mathrm{diag}(t^a, 1, \ldots, 1)$, $a>0$. 
  We have $\lim_{t\rightarrow 0}g(t)\cdot (A_1,A_2,A_3)$ exists and $\theta(g)=t^{a(\sigma_1+\sigma_3)}=t^{<0}$, which contradicts the condition that $(A_1, A_2, A_3)$ is semistable.
  
  One can prove that when a point $(A_1,A_2,A_3)$ is semistable, then the augmented matrix $\begin{bmatrix} A_1\\A_3\end{bmatrix}$ is non-degenerate by the same argument as above by using the condition that $\sigma_2+\sigma_3<0$.

  On the other hand, we are going to prove that points in the set of the right hand side of \eqref{eqn:Q6semistable} are semistable.  
  Let $(A_1,A_2,A_3)$ be such a point.
  Let $g(t)=(g_1(t),g_2(t),g_3(t))$ be an one-parameter subgroup of $G_6$ with $g_1=\mathrm{diag}(t^{a_1}, t^{a_2} , \ldots, t^{a_{M_1}})$, $g_1=\mathrm{diag}(t^{b_1}, t^{b_2} , \ldots, t^{b_{M_2}})$, $g_3=\mathrm{diag}(t^{c_1}, t^{c_2} , \ldots, t^{c_{M_3}})$, such that the limit
  \begin{equation}\label{eqn:Q6limit}
      \lim_{t\rightarrow 0}g(t)\cdot (A_1,A_2,A_3)
  \end{equation} exists.  
  Suppose $c_i>0$, for $1\leq i \leq k$, and $c_i\leq 0$, for $k+1\leq i \leq M_3$. 
  Then a quick result of this assumption is that the first $k$ columns of $A_3$ are zero. 
  Since $\begin{bmatrix}A_i\\A_3 \end{bmatrix}, i=1,2$ are non-degenerate, there exists distinct $l_i$, $1\leq i$, and distinct $m_j$, $j\leq k$ such that 
  \begin{equation*}
      (A_1)_{l_i, i}\neq 0,(A_2)_{m_j, j}\neq 0, 1\leq i, j\leq k.
  \end{equation*}
 To simplify notations, we assume $l_i=m_i=i$, $1\leq i\leq k$. 
 By non-degeneracy of $\begin{bmatrix}A_1\\A_2\end{bmatrix}$, we can find distinct $n_1, \ldots, n_{M_1-k}, n_{M_1-k+1},\ldots,n_{M_1+M_2-2k}$, such that $n_i>k$,  and
 \begin{equation}
    ( A_1)_{ k+i,n_i}\neq 0, (A_2)_{ k+j,n_{M_1-k+j}}\neq 0.\nonumber
 \end{equation}
 Again, we assume $n_i=k+i$, then
       \begin{align*}
           a_i\geq  c_i,\, \forall i;
           \quad     b_j\geq 
           \begin{cases}
               c_j, & {1\leq j\leq k;} \\
           c_{M_1+j-k}, & {j\geq k+1.}
           \end{cases}
       \end{align*}
Then 
      \begin{align*}
      <\theta_6, g(t)>&=(\sum_{i=1}^{M_1} a_i)\sigma_1+(\sum_{j=1}^{M_2} b_j)\sigma_2+(\sum_{l=1}^{M_3} c_l)\sigma_3\nonumber\\
      &
      \geq (\sum_{i=1}^{M_1} c_i)\sigma_1+(\sum_{j=1}^{k} c_{j}+\sum_{j=M_1+1}^{M_1+M_2}c_j)\sigma_2+(\sum_{l=1}^{M_3} c_l)\sigma_3\\
      &\geq (\sum_{i=1}^{k} c_i)(\sigma_1+\sigma_2+\sigma_3)+(\sum_{j=k+1}^{M_1} c_{j})(\sigma_1+\sigma_3)+(\sum_{j=k+1}^{M_1+M_2} c_{j})(\sigma_2+\sigma_3)\geq 0
      \end{align*}
    Hence, such a point must be semistable. 
\end{proof}
\subsection{Quiver $\mathbf Q_7$}
We adopt the notations as the previous subsection by letting $M_1=N_3-N_2,M_2=N_3-N_1,M_3=N_3,M_4=N_4$.
\begin{lem}
    In the proposed phase
    \begin{align}
        \sigma_1<0,\,\sigma_2<0,\,
        \sigma_1+\sigma_2+\sigma_3>0
    \end{align}
    the semistable locus is 
    \begin{align}
        V^{ss}_{7,\theta_7}=\{(A_1,A_2,A_3) \big| A_1,A_2,A_3 \text{ non-degenerate}\}.
    \end{align}
\end{lem}
\begin{proof}
    We first prove that if a point $(A_1,A_2,A_3)$ is semistable, then in the proposed phase, $A_1,A_2,A_3$ are all nondegenerate. 
    A quick result of the phase conditions $\sigma_1<0,\sigma_2<0,\sigma_3>0$ is that matrices $A_1,A_2$ and $\begin{bmatrix}A_1&A_2&A_3 \end{bmatrix}$ are non-degenerate. 
    Furthermore, we claim that the matrix $A_3$ is also nondegenerate under the condition $\sigma_1+\sigma_2+\sigma_3>0$. 
    Otherwise, the matrix $A_3$ is equivalent to $A_3=\begin{bmatrix}0\\ *\end{bmatrix}$ under $G_7$ action. The non-degeneracy of augmented matrix $\begin{bmatrix}A_1&A_2&A_3 \end{bmatrix}$ tells us that the first row of one of $A_1$ and $A_2$ is nonzero, which we assume to be $A_1$. Then matrix $A_1$ can be transformed to
    \begin{equation}\label{eqn:Q7ssA1}
        A_1=\begin{bmatrix} 1&\mathbf 0\\ *&* \end{bmatrix}
    \end{equation} by column operations without changing the formula of $A_3$. If the first row of $A_2$ is zero, we do nothing to $A_2$. However, if the first row of $A_2$ is nonzero, we can also transform $A_2$ to the formula in \eqref{eqn:Q7ssA1} by column operations, without changing the formulas $A_1$ and $A_3$. 
    Then we can choose an one-parameter subgroup $g(t)\subset G_7$ such that $g_1(t),g_2(t), g_3(t)$ are of the form $\mathrm{diag}=(t^{-1},1,\ldots,1)$. One can check that $\lim_{t\rightarrow 0}g(t)\cdot (A_1,A_2,A_3)$ exists and $\theta_7(g(t))=t^{-\sigma_1-\sigma_2-\sigma_3<0}$, which contradicts the condition that $(A_1,A_2,A_3)$ is semistable. 

    On the other hand, suppose that all $A_1,A_2,A_3$ are nondegenerate, we assert that such a point $(A_1,A_2,A_3)$ is semistable. Let $g(t)=(g_1(t),g_2(t),g_3(t))$ be an arbitrary one-parameter subgroup of $G_7$ with $g_1(t)=\mathrm{diag}(t^{a_1},\ldots, t^{a_{M_1}}), g_2=\mathrm{diag}(t^{b_1},\ldots, t^{b_{M_2}}), g_3(t)=\mathrm{diag}(t^{c_1},\ldots, t^{c_{M_3}})$, such that $\lim_{t\rightarrow }g(t)\cdot (A_1,A_2,A_3)$ exits. 
    
    The nondegeneracy of $A_3$ implies that $c_i\geq 0$. The nondegeneracy of $A_1$ and $A_2$ tells us that there are distinct integers $\{k_i\}_{i=1}^{M_1}\subset [M_3]$ such that $a_i\leq c_{k_i}$, and there are distinct integers $\{j_i\}_{i=1}^{M_2}\subset [M_3]$ such that $b_i\leq c_{j_i}$. 
    Then 
    \begin{align}
        <\theta_7,g(t)>&=
        \sigma_1(\sum_{i=1}^{M_1}a_i)+
        \sigma_2(\sum_{i=1}^{M_2}b_i)+
        \sigma_3(\sum_{i=1}^{M_3}c_i)
        \nonumber\\
        &\geq \sigma_1(\sum_{i=1}^{M_1}
        c_{k_i})
        +
        \sigma_2(\sum_{i=1}^{M_2}
        c_{j_i})+
        \sigma_3(\sum_{i=1}^{M_3}c_i)
        \nonumber\\
        &\geq 
        \sigma_1(\sum_{i=1}^{M_3}
        c_i)
        +
        \sigma_2(\sum_{i=1}^{M_3}
        c_i)+
        \sigma_3(\sum_{i=1}^{M_3}c_i)
        \nonumber\\
        &\geq (\sigma_1+\sigma_2+\sigma_3)(\sum_{i=1}^{M_3}c_i)\geq 0.
    \end{align} 
   Therefore,
   such point is semistable. 
 \end{proof}
 \subsubsection{Torus fixed points of $\mc X_6$ and $\mc X_7$}
 Let $R=(\mathbb C^*)^{M_4}$. The torus $R$ acts on both $\mc X_6$ and $\mc X_7$. We will find the torus fixed loci $\mc X_6^R$ and $\mc X_7^R$, and prove that there is a bijection between these two loci. 
\begin{lem}\label{fixed pt for X6}
  The $R$-fixed locus of $\mathcal{X}_6$ is a finite set of isolated fixed points. It can be parameterized by the following set
  \begin{equation}
    \mathfrak{F}_6=\big\{(\vec C_{[M_1]},\vec C_{[M_2]},\vec C_{[M_3]})\,\rvert \,\vec C_{[M_1]}\subset\vec C_{[M_3]}\subset [M_4],\,\vec C_{[M_2]}\subset\vec C_{[M_3]} \big\}\,.
\end{equation}
An element $(\vec C_{[M_1]},\vec C_{[M_2]},\vec C_{[M_3]})\in \mathfrak{F}_6$ represents a $G_6$-orbit of the following form $(A_1, A_2, A_3)$.
Define a map $$d_i: \vec C_{[M_i]}=\{l_1<l_2<\cdots < l_{M_i}\}\rightarrow \{1,\ldots ,M_i\}$$ that sends $l_k$ to $k$.
Denote by $\alpha_{i_1, i_2, \ldots, i_k}$ a column vector whose $i_1,i_2,\ldots,i_k$-th components are 1 and others are 0.
Let $\begin{bmatrix}A_1\\A_2\\A_3\end{bmatrix}$ be a matrix whose column vectors are listed as follows without ordering.
\begin{enumerate}[(1)]
  \item $\alpha_{d_1(i), M_1+M_2+i}, \quad i\in \vec C_{[M_1]}\backslash (\vec C_{[M_1]}\cap\vec C_{[M_2]})$,
  \item $\alpha_{M_1+d_2(j), M_1+M_2+j}, \quad j\in \vec C_{[M_2]}$,
  \item $\alpha_{d_1(k), M_1+d_2(k)}, \quad k\in \vec C_{[M_1]}\cap\vec C_{[M_2]}$,
  \item $\alpha_{M_1+M_2+l}$, $\quad l\in [M_4]\backslash \vec C_{[M_3]}$.
\end{enumerate}

\end{lem}
The proof of the above lemma is elementary, and we omit it here. We will illustrate the statement and the idea via the following example.
\begin{ex}
  We consider the case $M_1=M_2=1, M_3=3, M_4=4$. Then $(\{1\}, \{1\}, \{1, 2, 3\})\in \mathfrak{F}_6$ represents a point whose $G_6$-orbit admits an element of the following form,
\begin{align*}
   \begin{bmatrix}A_1\\A_2\\A_3 \end{bmatrix}=\begin{bmatrix}0& 1&0\\ 1&1& 0\\ 1&0& 0\\ 0& 0& 0\\ 0& 0& 0\\ 0& 0& 1\end{bmatrix}
\end{align*}
$(\{1\}, \{2\}, \{1, 2, 3\})$ represents a point of the following form,
\begin{align*}
  \begin{bmatrix}
  A_1\\A_2\\A_3
  \end{bmatrix}
  =
  \begin{bmatrix}
  1& 0& 0\\ 
  0& 1& 0\\ 
  1& 0& 0\\ 
  0& 1& 0\\ 
  0& 0& 0\\ 
  0& 0& 1
  \end{bmatrix}
\end{align*}
Notice that a general $A_3$ should be $G_6$-equivalent to a new $A_3$ which is a reduced column echelon form. If it is fixed by $R$, it has at most one nonzero component in each column.
Since augmented matrices $\begin{bmatrix}A_i\\A_3\end{bmatrix}$, $i= 1, 2$ are non-degenerate, $ \begin{bmatrix}A_1\\A_2\\A_3\end{bmatrix}$ should be one of the following two types
\begin{align*}
   (1)\, \begin{bmatrix}
         a_1& a_2& a_3\\  
         b_1& b_2& b_3\\ 
         1& 0& 0\\ 
         0&1& 0\\ 
         0&0&1\\ 
         0&0&0
         \end{bmatrix}
         \quad (2)\,
         \begin{bmatrix}0& 0& 1\\ 
         b_1& b_2& 1\\ 
         1& 0& 0\\ 
         0& 1& 0\\ 
         0& 0& 0\\ 
         0& 0& 0
         \end{bmatrix}
\end{align*}
Case $(1)$ has $C_4^3$ possibilities and case $(2)$ has $C_4^2$ possibilities by varying the positions of pivots of $A_3$.

For case $(1)$, 
since $\begin{bmatrix}A_1\\A_2\end{bmatrix}$ is non-degenerate, there exist $i$ and $j$, $i\neq j$, such that $a_i\neq 0$, $b_j\neq 0$. 
Since the point $(A_1, A_2, A_3)$ is fixed by $R$, the remaining $a_k=0$ for $k\neq i$, $b_k=0$ for $k\neq j$. 
Therefore, the case $(1)$ has in total  $C_4^3\times C_3^2=24$ possibilities.

Now we consider the case $(2)$, since $\begin{bmatrix}A_1\\A_2\end{bmatrix}$ is non-degenerate, and $\begin{bmatrix}A_2\\A_3\end{bmatrix}$ is fixed by $R$, there is exactly one $b_i$ non-vanishing. If $b_1$ is non-vanishing, then matrix is
  \begin{align*}
  \begin{bmatrix}
  A_1\\A_2\\A_3
  \end{bmatrix}
  =
  \begin{bmatrix}
  0& 0&1\\ 
  1&0& 1\\ 
  1&0&0\\ 
  0&1&0\\ 
  0&0&0\\ 
  0&0&0
  \end{bmatrix}
\end{align*}
 The case (2) has in total $C_4^2\times 2=12$ choices.

In conclusion, there are $24+12=36$ $R$-fixed points. This matches the quantity of fixed points in $\mc X_7$, which is $C_4^3\times C_3^1\times C_3^1=36$.
\end{ex}

\begin{lem}
    The $R$-fixed points in $\mc X_7$ can be parameterized by the following set
    \begin{align}
        \mathfrak F_7=
        \big\{(\vec C_{[M_1]},\vec C_{[M_2]},\vec C_{[M_3]})\,\rvert \,\vec C_{[M_1]}\subset\vec C_{[M_3]}\subset [M_4],\,\vec C_{[M_2]}\subset\vec C_{[M_3]} \big\}\,.
    \end{align}
    An element $(\vec C_{[M_1]},\vec C_{[M_2]},\vec C_{[M_3]})$ in $\mathfrak F_7$ represents a point $(A_1,A_2,A_3)$ in $\mc X_7^R$ of the following form. 
    The matrix $A_3$ is in row reduced echelon form with the column numbers of pivots being $\vec C_{[M_3]}$. 
    Matrices $A_1$ and $A_2$ are both reduced column echelon forms.     Relabel the rows of $A_1,A_2$ by numbers in $\vec C_{[M_3]}$. Row numbers of pivots of $A_1$ and $A_2$ are $\vec A_{[M_1]}$ and $\vec A_{[M_2]}$.
\end{lem}
\begin{proof}
    Since for any element  $(A_1,A_2,A_3)\in V^{ss}_{7,\theta_7}$, matrices $A_1,A_2,A_3$ are nondegenerate, in $G_7$-orbit, we can always find a representative such that all three matrices $A_i$, $i=1,2,3$ are in reduced row/column echelon forms. They are $R$-fixed, so their non-pivots entries all vanish. 
    The set $(\vec A_{[M_1]},\vec B_{[M_2]},\vec C_{[M_3]})$ in $\mathfrak F_7$ is taking the positions of pivots of matrices $A_1,A_2,A_3$ down.
    Then the lemma can be obtained.  
\end{proof}
\begin{cor}\label{cor:appx6x7torusfixedpoints}
  There is a canonical one-to-one correspondence between the fixed points set of $\mathcal{X}_6$ and $\mathcal{X}_7$.
\end{cor}
\begin{proof}
  The bijection is due to the fact that the two $R$-fixed loci $\mathcal{X}_6^R$ and $\mathcal{X}_7^R$ are both parameterized by the same sets $\{(\vec A_{[M_1]},\vec B_{[M_2]},\vec C_{[M_3]})\}$. 
\end{proof}

\subsection{Quiver $\mathbf Q_8$}
\begin{lem}
    In the phase 
    \begin{align}
        \sigma_1>0,\,\sigma_2<0,\,
        \sigma_2+\sigma_3>0,\,
    \end{align}
    the semistable locus is 
    \begin{align}
        V^{ss}_{8,\theta_8}=\{(A_1,A_2,A_3)\big| A_1,A_2,A_3\text{ non-degenerate}\}.
    \end{align}
\end{lem}
\begin{proof}
    The proof is easy and similar with the proof for $V^{ss}_{7,\theta_7}$. We omit it. 
\end{proof}
\begin{lem}
    The torus fixed locus $\mc X_8^R$ can be parameterized by the following set
    \begin{equation}
        \mathfrak F_8=\{(\vec A_{[N_2]},\vec B_{[N_3-N_1]},\vec C_{[N_3]})\big|\vec A_{[N_2]}\subset \vec C_{[N_3]} \subset [N_4],\, \vec B_{[N_3-N_1]}\subset \vec C_{[N_3]} \}.
    \end{equation}
\end{lem}
\begin{proof}
    The proof is easy. Since in semistable locus, $A_1,A_2,A_3$ are all nondegenerate, in $G_8$-orbit, we can find a representative such that the three matrices are reduced row/column echelon forms. Since the point is fixed by $R$-action, all entries except for the pivots vanish. Integers in the set $\vec C_{[N_3]}$  the column numbers of pivots of $A_3$, and those in the set $\vec A_{[N_2]}(\vec B_{[N_3-N_1]})$ are the column(row) numbers of pivots of $A_1(A_2)$ after we relabel columns (rows) of matrix $A_1(A_2)$.
\end{proof}
\subsection{Quiver $\mathbf Q_{10}$}
\subsubsection{Semistable locus}
Adopt the notations for the quiver $\mathbf Q_{10}$ in Section \ref{quivermutation}.
\begin{lem}\label{lem:Q10semapp}
  Choose phase of character $\theta_{10}$ as 
   \begin{align}
   \sigma_2>0,\,
    \sigma_3>0,\, 
   \sigma_1+\sigma_3<0.
    \end{align}
The semistable locus is 
    \begin{align}\label{eqn:appsem10}
        Z_{10}^{ss}(G_{10})=\{C=0,A_2A_1+B_2B_1=0\big|\, B_1,A_1,A_2 \text{ non-degenerate }\}.
    \end{align}
\end{lem}
\begin{proof}
    We first can easily find that when a point $(A_i,B_i,C)$ is semistable, $B_1,\begin{bmatrix}A_2&B_2 \end{bmatrix}, \begin{bmatrix}A_1\\B_1 \end{bmatrix}$ are all nondegenerate by the condition that $\sigma_1<0,\sigma_2>0,\sigma_3>0$. 
    The nondegeneracy of $\begin{bmatrix}A_2&B_2 \end{bmatrix}$ combining equations $CA_2=0, CB_2=0$ in $Z(dW)$ makes $C=0$. 

    We further claim that $A_1$ is nondegenerate. Otherwise, under the action of gauge group, the matrix $A_1$ can be transformed to a formula with one zero column which without loss of generality we assume to be the last column $A_1=\begin{bmatrix} *&\mathbf 0\end{bmatrix}$. Since $\begin{bmatrix}A_1\\B_1 \end{bmatrix}$ is nondegenerate, the last column of $B_1$ is nonzero, and via gauge group action, the matrix $B_1$ can be transformed to $B_1=\begin{bmatrix} \mathbf 0&1\\ *&\mathbf 0\end{bmatrix}$ without changing the format of $A_1$. 
    Since $A_2A_1+B_2B_1=0$ and the representatives of $A_1,B_1$ we choose in the $G$-orbit as above, the first column of $B_2$ must be a zero vector. 
    We choose a one parameter subgroup of $g(\lambda)\subset G_{10}$ as follows
    \begin{align}
        g_1(\lambda)=\mathrm{diag}(1,\ldots,1,\lambda),\,
        g_3(\lambda)=\mathrm{diag}(\lambda,1,\ldots,1),\,g_2=Id_{N_2}\,.
    \end{align}
    One can check that $\lim_{\lambda\rightarrow 0}g(\lambda)\cdot (A_1,B_1,A_2,B_2)$ exists, and $\theta(g(\lambda))=\lambda^{\sigma_1+\sigma_3<0}$, which contradicts the assumption that the point $(A_i,B_i,C)$ is semistable. 

    Furthermore, we claim that $A_2$ is non-degenerate. Since $A_2A_1+B_2B_1=0$, the non-degeneracy of $B_1$ confirms that columns of $B_2$ are linear combinations of those of $A_2$. Hence rank of $\begin{bmatrix}A_2&B_2\end{bmatrix}$ is equal to the rank of $A_2$, which is equal to $N_2$. Therefore, matrix $A_2$ is non-degenerate. 
    Until now, we have proved that $Z^{ss}_{10,\theta_{10}}$ is contained in the right hand side set in Equation \eqref{eqn:appsem10}. 

    On the other hand, let $(A_i,B_i,C)$ be a point in the set of right hand side of \eqref{eqn:appsem10}. Let $g(\lambda)=(g_1(\lambda),g_2(\lambda),g_3(\lambda))\subset G_{10}$ be an arbitrary subgroup such that $\lim_{\lambda\rightarrow 0}g(\lambda)\cdot (A_i,B_i,C)$ exists. 
    Suppose that 
    \begin{align}
        g_1(\lambda)=\mathrm{diag}(\lambda^{a_1},\ldots, \lambda^{a_{N_2}}),\,g_2(\lambda)=\mathrm{diag}(\lambda^{b_1},\ldots, \lambda^{b_{N_2}}),\,
        g_3(\lambda)=\mathrm{diag}(\lambda^{c_1},\ldots, \lambda^{c_{N_3-N_1}})\,.
    \end{align}
    Then we must have 
    \begin{align}
        &a_i\leq 0,\,b_i\geq 0, \forall i,\,\nonumber\\
        & \forall i\in \{1,\ldots,N_3-N_1\}, \exists j_i, \, \mathrm{s.t.} \,c_i\geq a_{j_i}.
    \end{align}
    Then 
    \begin{align}
        &<\theta_{10}, g(\lambda)>=\sigma_1(\sum_{i=1}^{N_2}a_i)+\sigma_2(\sum_{i=1}^{N_2}b_i)+\sigma_3(\sum_{i=1}^{N_3-N_1}c_i)\nonumber\\
        &\geq \sigma_1(\sum_{i=1}^{N_2}a_i)+\sigma_3(\sum_{i=1}^{N_3-N_1}a_{j_i})\geq 
      (\sigma_1+\sigma_3)(\sum_{i=1}^{N_3-N_1}a_{j_i})+ \sigma_1(\sum_{j\neq j_i}a_{j})\geq 0
    \end{align} 
    where in each step we have abandoned the terms that are obviously non-negative. 
\end{proof}
\begin{lem}\label{fixed point for Q10}
    The $R$-fixed locus $\mathfrak F_{10}$ can be described as follows 
    \begin{align}
        \{\vec A_{[N_2]},\vec B_{[N_2]}, \vec C_{[N_3-N_1]}\big |\, \vec C_{[N_3-N_1]}\subset \vec A_{[N_2]}\subset [N_4],\, \vec B_{[N_2]}\subset ([N_4]\backslash \vec A_{[N_2]})\cup \vec C_{[N_3-N_1]} \}
    \end{align}
    There is a bijection 
    \begin{equation}
        \iota_{10}: \mathfrak F_2\rightarrow \mathfrak F_{10}
    \end{equation}
    which sends $(\vec A_{[N_2]},\vec B_{[N_2]},\vec C_{[N_3']})$ to $(\vec A_{[N_2]}, \vec B_{[N_2]}, \vec A_{[N_2]}\backslash \vec C_{[N_3']})$. 
\end{lem}
\begin{proof}
     The inclusion $\vec C_{[N_3-N_1]}\subset \vec A_{[N_2]}\subset [N_4]$ is due to the non-degeneracy of matrices $A_1$ and $B_1$ which can be written as reduced column and reduced row echelon forms with non pivots vanishing, and then we use $\vec A_{[N_2]}$ and $\vec C_{[N_3-N_1]}$ to label numbers of rows and columns respectively. 
     
     The matrix $A_2$ itself is non-degenerate, so we can write $A_2$ as a reduced row echelon form and use $\vec B_{[N_2]}$ to represent such a matrix. The relation $A_2A_1+B_2B_1$ says that columns of $A_2$ in $\vec A_{[N_2]}\backslash \vec C_{[N_3-N_1]}$ must vanish. Hence we get the condition $\vec B_{[N_2]}\subset ([N_4]\backslash \vec A_{[N_2]})\cup \vec C_{[N_3-N_1]}$.

    The bijection of the map $\iota_{10}$ is easy and we omit it.
\end{proof}

 \bibliographystyle{amsalpha}
\bibliography{Seibergreferences2}

\end{document}